\documentclass[a4paper]{article}
\usepackage[left=3cm,right=3cm,top=2.5cm,bottom=2.5cm]{geometry}
\usepackage{amsmath,amsthm,amssymb,bm}
\usepackage{graphicx,subfigure} 
\usepackage{authblk}  
\usepackage{cite}     
\usepackage{booktabs} 
\usepackage{algorithm}
\usepackage{enumitem} 
\usepackage{rotating} 
\usepackage{placeins}  
\usepackage{xcolor} 
\newtheorem{mytheorem}{Theorem}[section] 
\newtheorem{mylemma}{Lemma}[section]
\newtheorem{mycorollary}{Corollary}[section]
\newtheorem{myproposition}{Proposition}[section]
\theoremstyle{definition}
\newtheorem{mydefinition}{Definition}[section]
\newtheorem{myremark}{Remark}[section]
\numberwithin{equation}{section}  %
\numberwithin{figure}{section}    %
\numberwithin{table}{section}     %
\numberwithin{algorithm}{section}
\newcommand{\myvec}[1]{\textup{\textbf{#1}}}
\newcommand{\dif}{\mathop{}\!\mathrm{d}}
\newcommand{\RS}{\mathbb{R}}
\newcommand{\NS}{\mathbb{N}}
\newcommand{\DS}{\mathbb{D}}
\newcommand{\diag}{\mathrm{diag}\,} 
\newcommand{\tridiag}{\mathrm{tridiag}\,} 
\newcommand{\ball}{\bm{\mathrm{B}}} 

\newcommand{\mysum}{\sum\limits}
\newcommand{\mytran}[1]{{#1}^{\top}} 
\begin{document}

\title{Convergence analysis of Anderson acceleration for nonlinear equations
  with H\"older continuous derivatives}

\author{
    Yonghui Ling\thanks{Corresponding author. \textit{E-mail address:} \texttt{yhling@mnnu.edu.cn}}, \
    Zikang Xiong and Juan Liang
    \and
    Department of Mathematics, Minnan Normal University, Zhangzhou 363000, China
    }

\maketitle
\begin{abstract}
  This work investigates the local convergence behavior of Anderson acceleration in solving nonlinear systems.
  We establish local R-linear convergence results for Anderson acceleration with general depth $m$ 
  under the assumptions that the Jacobian of the nonlinear operator is H\"older continuous
  and the corresponding fixed-point function is contractive.
  In the Lipschitz continuous case, we obtain a sharper R-linear convergence factor.
  We also derive a refined residual bound for the depth $m = 1$
  under the same assumptions used for the general depth results.
  Two applications are presented.
  For a nonsymmetric algebraic Riccati equation from transport theory,
  Anderson acceleration yields comparable results to several existing fixed-point methods
  in the regular cases and brings significant reductions 
  in both the number of iterations and computation time, 
  even in challenging cases involving nearly singular or large-scale problems.
  For a semilinear convection-diffusion-reaction equation,
  the Jacobian is shown to be H\"older continuous with a mesh-independent constant
  but not Lipschitz continuous in any neighborhood of the solution,
  so that the existing Lipschitz-based convergence results do not apply,
  whereas all the constants involved in our convergence results are computable 
  and do not degenerate under mesh refinement.


\end{abstract}
\section{Introduction}
\label{sec:Introduction}

Our aim in this paper is to study convergence acceleration techniques 
for nonlinear fixed-point iteration of the form
\begin{equation}
  \label{it:FixedPoint}
  \myvec{x}_{k+1} = \myvec{g}(\myvec{x}_k), \quad k = 0,1,2,\ldots,
\end{equation}
which is commonly used in approximating the solution of nonlinear system
\begin{equation}
  \label{eq:f(x)=0}
  \myvec{f}(\myvec{x}) = \myvec{g}(\myvec{x}) - \myvec{x} = \myvec{0},
\end{equation}
where $\myvec{g}: \DS \subset \RS^n \to \DS$ is a given Fr\'echet continuously differentiable nonlinear operator, 
$\DS$ is an open convex subset of $\RS^n$.
The fixed-point iteration \eqref{it:FixedPoint}, 
also referred to as the nonlinear Richardson or Picard iteration,
is a fundamental method for solving the fixed-point problem $\myvec{g}(\myvec{x}) = \myvec{x}$.
However, the sequence generated by \eqref{it:FixedPoint} often exhibits slow convergence 
or may even fail to converge in certain cases.
Consequently, achieving faster convergence rates has been a central focus
in optimization and numerical analysis communities.
Extrapolation methods \cite{Anderson2019,BrezinskiRS2018,BrezinskiR2019}
provide effective strategies for accelerating the convergence of iterative sequences.
A classical example is Anderson acceleration,
also referred to as Pulay mixing or direct inversion in the iterative subspace
in quantum chemistry and physics communities \cite{Pulay1980}.

Anderson acceleration, originally proposed in \cite{Anderson1965},
has received increasing interest over the past decade as a powerful technique
for accelerating fixed-point iterations \eqref{it:FixedPoint}.
Initially designed to solve nonlinear integral equations,
Anderson acceleration has since demonstrated remarkable efficiency 
in accelerating convergence, 
particularly for problems arising from discretized partial differential equations 
\cite{LipnikovSV2013,AnJW2017,PollockRX2019,PollockR2021,RebholzVX2021,PollockR2023,PollockRTX2025}, 
smooth and nonsmooth optimization
\cite{ChenK2019,BianCK2021,BianC2022,ScieurAB2020,FuZB2020,ZhangDB2020,DeSterckH2021,BarreTDA2022,BollapragadaSD2023,OuyangLM2024},
data analysis \cite{HighamS2016,Bach2021},
and machine learning \cite{MaiJ2020,WangHS2021,PasiniYRS2021,PasiniYRS2022}.
The underlying idea of Anderson acceleration is to generate new iterates
by constructing an optimized linear combination of the previous iterates
and their corresponding residuals, thereby significantly improving the rate of convergence.

The following algorithm describes the use of Anderson acceleration to solve
the fixed-point problem $\myvec{g}(\myvec{x}) = \myvec{x}$.
The parameter $m$ is referred to as the depth or window-size of Anderson acceleration,
and the coefficients $\alpha_j^{(k)}$ are called acceleration parameters.

\begin{algorithm}
  \caption{Anderson acceleration for solving \eqref{eq:f(x)=0}} 
  \label{alg:AA}
  Given the depth $m \in \NS$. Choose an initial point $\myvec{x}_0 \in \RS^n$ 
  and set $\myvec{x}_1 = \myvec{g}(\myvec{x}_0)$.
  For $k = 1,2,\ldots$ until convergence, do:
  \begin{itemize}[leftmargin=1em,itemindent=3.5em,parsep=0em,itemsep=0em,topsep=0em]
    \item[Step 1.]
    Set $m_k = \min\{m,k\}$.
    \item[Step 2.]
    Compute $\myvec{f}_k \triangleq \myvec{f}(\myvec{x}_k) = \myvec{g}(\myvec{x}_k) - \myvec{x}_k$.
    \item[Step 3.]
    Solve the convex optimization problem
    \begin{equation}
      \label{optim:alpha}
      \min_{\bm{\alpha}_k = \mytran{(\alpha_0^{(k)},\ldots,\alpha_{m_k}^{(k)})}} 
      \left\|\sum_{j=0}^{m_k} \alpha_j^{(k)}\myvec{f}_{k-m_k+j}\right\| \quad \text{s.t.} \quad
      \sum_{j=0}^{m_k} \alpha_j^{(k)} = 1.
    \end{equation}
    \item[Step 4.]
    Set $\myvec{x}_{k+1} = \mysum_{j=0}^{m_k} \alpha_j^{(k)} \myvec{g}(\myvec{x}_{k-m_k+j})$.
  \end{itemize}
\end{algorithm}

The update rule in Step 4 of Algorithm \ref{alg:AA} can be
extended to the general Anderson mixing form:
$$
\myvec{x}_{k+1} = (1 - \beta_k) \sum_{j=0}^{m_k} \alpha_j^{(k)} \myvec{x}_{k-m_k+j}
  + \beta_k \sum_{j=0}^{m_k} \alpha_j^{(k)} \myvec{g}(\myvec{x}_{k-m_k+j}),
$$
where $\beta_k \in (0,1]$ is a damping parameter.
In this work, we focus on the undamped case $\beta_k \equiv 1$, 
which has received significant attention for its theoretical advantages
in convergence analysis
\cite{TothK2015,ChenK2019,BianCK2021,BianC2022,PollockRX2019,OuyangLM2024}.
Recent developments on the damped case can be found in
\cite{EvansPRX2020,PollockR2021,RebholzX2023,PollockR2023}.
When applied to linearly converging fixed-point iterations,
the minimization step in Anderson acceleration often helps to 
improve both convergence rates and robustness.
This observation has motivated the development of convergence analysis
for Anderson acceleration under minimization-based frameworks,
with early results for contractive mappings \cite{TothK2015,ChenK2019} 
and further generalizations to include noncontractive cases \cite{PollockR2021}.
However, acceleration is not guaranteed 
when the underlying iterations are already quadratically convergent,
as noted in \cite{EvansPRX2020}.

The behavior of Anderson acceleration has been further investigated
by examining its connection to other nonlinear solvers.
It has been demonstrated \cite{FangS2009} that Anderson acceleration
is equivalent to a specialized version of the generalized Broyden method.
In this formulation, 
the approximate inverse Jacobian is computed implicitly 
by solving a constrained optimization problem that satisfies secant conditions 
on recent iterates.
This connection motivates the study of Anderson acceleration 
using tools from the convergence theory of quasi-Newton methods.
We refer to \cite{Kelley2018,Saad2025} for comprehensive reviews on this topic.

While Anderson acceleration had demonstrated good numerical performance 
in various applications,
its first rigorous local convergence result was provided by Toth and Kelley \cite{TothK2015} in 2015.
They proved that if the nonlinear operator $\myvec{g}$ is continuously differentiable 
with Lipschitz continuous derivative, 
then the sequence generated by Algorithm \ref{alg:AA} 
converges R-linearly (see Definition \ref{def:R-Linear}) to a fixed point, 
provided the initial guess is sufficiently close. 
In particular, when the Euclidean norm is used and the depth $m = 1$, 
the sequence converges Q-linearly (see Definition \ref{def:Q-Linear}).
An improved version of the local convergence result was later established by Kelley \cite{Kelley2018}.
This was further developed by Chen and Kelley \cite{ChenK2019},
who both relaxed the required assumptions and improved the R-linear convergence factor.
Bian et al. \cite{BianCK2021,BianC2022} then extended the convergence results
to the case of nonsmooth fixed-point problems.
Subsequent efforts \cite{PollockRX2019,EvansPRX2020,PollockR2021}
provided a comprehensive resolution to the question of
how Anderson acceleration improves the convergence of linearly converging fixed-point iterations.
Specifically, the study in \cite{PollockRX2019} revealed that the linear convergence of fixed-point iterations 
applied to the steady Navier-Stokes equations is improved by Anderson acceleration 
via the gain factor from the underlying optimization.
Evans et al. \cite{EvansPRX2020} subsequently generalized to general contractive mappings,
and Pollock and Rebholz \cite{PollockR2021} further extended
and refined the convergence results to include noncontractive mappings.
Recently, Rebholz and Xiao \cite{RebholzX2023} showed that Anderson acceleration
may reduce the order of convergence for superlinearly converging methods, but
it significantly accelerates the convergence for sublinearly converging methods.

The convergence analyses in \cite{TothK2015,EvansPRX2020,PollockR2021}
rely on the assumption that the derivative of the nonlinear operator is Lipschitz continuous.
To generalize these results and extend the applicability of Anderson acceleration, 
a natural approach is to weaken this assumption to H\"older continuity.
This relaxation is particularly motivated by optimization theory,
where H\"older continuity provides a bridging framework between smooth and nonsmooth problems 
\cite{Nesterov2015,GrapigliaN2017,CartisGT2019,MarumoT2025}.
More specifically,  a H\"older exponent of zero indicates a bounded derivative,
an exponent in $(0,1)$ corresponds to a continuous but potentially non-differentiable derivative,
while an exponent of $1$ represents a Lipschitz continuous derivative,
which is again differentiable.
Notably, there exist problems for which the derivative is H\"older continuous
but fails to be Lipschitz continuous in every neighborhood of the solution.
Subsection \ref{subsec:CDR} provides a concrete 
semilinear convection-diffusion-reaction problem exhibiting precisely this behavior.

The goal of this paper is to establish local convergence results for Anderson acceleration
applied to the nonlinear system \eqref{eq:f(x)=0},
assuming that the nonlinear operator $\myvec{f}$ has a H\"older continuous Jacobian.
Our main theoretical contribution is to show that,
when the first derivative of the nonlinear operator $\myvec{f}$ is H\"older continuous
and the associated fixed-point function $\myvec{g}$ is contractive,
the sequence generated by Algorithm \ref{alg:AA} with general depth $m \geq 1$
converges R-linearly to the solution $\myvec{x}^*$ of the nonlinear system \eqref{eq:f(x)=0}.
In the special case where the Jacobian is Lipschitz continuous,
we provide an implicit characterization of the R-factor
in terms of the contraction factor, the gain of the optimization problem,
and the condition number of the Jacobian at $\myvec{x}^*$.
Additionally, we obtain a refined residual bound for Anderson acceleration
with depth $m = 1$ under the same assumptions used for the general depth results.
We apply our theoretical framework to compute an approximation 
of minimal positive solution of a nonsymmetric algebraic Riccati equation (NARE)
arising from transport theory.
Comprehensive numerical results demonstrate that Anderson acceleration
substantially outperforms several existing fixed-point methods,
yielding significant reductions in 
both the number of iterations and computation time,
even in challenging cases involving nearly singular or large-scale problems.
We further construct a discrete semilinear convection-diffusion-reaction problem
whose Jacobian is H\"older continuous with exponent $\nu \in (0,1]$ 
and a H\"older constant bounded independently of the mesh size, but which, 
for $\nu \in (0,1)$, fails to be Lipschitz continuous in every neighborhood of the solution.
The convergence results in \cite{ChenK2019,EvansPRX2020,PollockR2021,TothK2015}
are therefore not applicable to this problem, whereas our convergence result
applies with constants that can be evaluated explicitly
and that remain uniformly bounded under mesh refinement.

The remainder of the paper is organized as follows.
In Section \ref{sec:Preliminaries},
we introduce the necessary notation and present preliminary results for the convergence analysis.
Section \ref{sec:ConvResults} presents our main theoretical results
on the local convergence behavior of Anderson acceleration.
Section \ref{sec:Applications} contains two applications of Anderson acceleration,
one to a nonsymmetric algebraic Riccati equation 
and the other to a semilinear convection-diffusion-reaction equation.
Finally, we conclude the paper in Section \ref{sec:Conclusions}.

\section{Preliminaries}
\label{sec:Preliminaries}

Throughout this paper, vectors are columns by default and are denoted by bold lowercase letters, e.g., $\myvec{v}$,
while matrices are denoted by regular uppercase letters, e.g., $V$, which is clear from the context.
We use $\diag(\myvec{v})$ to denote the diagonal matrix with the vector $\myvec{v}$ on its diagonal,
and use $I$ to denote the identity operator or the identity matrix with proper dimension.
If there is potential confusion, we will use $I_n$ to denote identity matrix of dimension $n$.
The symbol $\myvec{e}_i = (0,\ldots,0,\underset{i}{1},0,\ldots,0)^{\top} \in \RS^n$ is $i$th column of the identity matrix $I_n$.
Let $\myvec{e} = (1,1,\ldots,1)^{\top}$ with proper dimension.
For a square nonsingular matrix $A \in \RS^{n \times n}$, we use $\kappa(A)$ to denote the condition number of $A$.

For $\myvec{x} \in \RS^n$ and real number $r > 0$, we use $\ball(\myvec{x},r)$
to stand for the open ball with center $\myvec{x}$ and radius $r$.
Recall that the Jacobian matrix of a continuously differentiable nonlinear operator
$\myvec{h}: \RS^n \to \RS^n$ at point $\myvec{x} \in \RS^n$ is represented by $\myvec{h}'(\myvec{x})$.
We recall the notions of Q- and R-order of convergence for a convergent sequence $\{\myvec{x}_k\}$ in $\RS^n$.
For more details about these two notions, 
see \cite{Jay2001,Potra1989} or the recent review paper \cite{Catinas2021} and the references therein.

\begin{mydefinition}
  \label{def:Q-Linear}
  A sequence $\{\myvec{x}_k\}$ is said to converge to $\myvec{x}^* \in \RS^n$ 
  with Q-order (at least) $q \geq 1$ if there exist constants $c \geq 0$ and $N \geq 0$ 
  such that
  $$
  \|\myvec{x}_{k+1} - \myvec{x}^*\| \leq c\|\myvec{x}_k - \myvec{x}^*\|^q
  \quad \text{for all} \quad k \geq N.
  $$
  In the special case $q = 1$, assuming $c < 1$, we say that 
  the sequence $\{\myvec{x}_k\}$ converges Q-linearly to $\myvec{x}^*$.
\end{mydefinition}

\begin{mydefinition}
  \label{def:R-Linear}
  A sequence $\{\myvec{x}_k\}$ is said to converge to $\myvec{x}^* \in \RS^n$
  with R-order (at least) $q \geq 1$ if there exists a positive real sequence $\{t_k\}$ 
  converging to zero with Q-order at least $q$ such that 
  $\|\myvec{x}_k - \myvec{x}^*\| \leq t_k$.
  When $q = 1$, we say that the sequence $\{\myvec{x}_k\}$ converges R-linearly to $\myvec{x}^*$.
\end{mydefinition}

The notion about H\"older continuous functions is as follows.

\begin{mydefinition}
  Let $\myvec{f}: \DS \subset \RS^n \to \RS^n$ be a Fr\'echet continuously differentiable operator,
  $\DS$ open and convex.
  We say that the Jacobian $\myvec{f}'$ is H\"older continuous with exponent $\nu \in (0,1]$ 
  if there exists a constant $H_\nu > 0$ such that
  \begin{equation}
    \label{cond:Holder}
    \|\myvec{f}'(\myvec{x}) - \myvec{f}'(\myvec{y})\| \leq H_{\nu}\|\myvec{x} - \myvec{y}\|^{\nu}
    \quad \text{for all} \quad \myvec{x},\myvec{y} \in \DS.
  \end{equation}
\end{mydefinition}

It should be noted that the operator $\myvec{f}$ determines $H_\nu$ for each $\nu \in (0,1]$,
while $\nu$ is not a constant determined by $\myvec{f}$.
In practice, determining the H\"older constant $H_\nu$ of a real-world operator
for a given $\nu \in (0,1]$ is often a challenging problem.
H\"{o}lder continuity, which generalizes Lipschitz continuity, 
has been extensively applied in complexity analyses of optimization methods
\cite{DevolderGN2014,Nesterov2015,CartisGT2017,CartisGT2019,GrapigliaN2017,GhadimiLZ2019,GrapigliaN2019,GrapigliaN2020,MarumoT2025}. 
Moreover, classical convergence results for Newton's method 
under H\"{o}lder-type continuity have been developed in 
\cite{Rokne1972,Argyros1992,Hernandez2001,Huang2002,Huang2004,LiS2008,ShenL2008}
for solving general nonlinear operator equations in Banach spaces.

The equation $\myvec{g}(\myvec{x}) = \myvec{x}$, 
an alternative formulation of the nonlinear equation \eqref{eq:f(x)=0}, 
is widely known as the fixed-point problem. 
A point $\myvec{x}^*$ satisfying $\myvec{g}(\myvec{x}^*) = \myvec{x}^*$ is called a fixed point of $\myvec{g}$. 
The operator $\myvec{g}$ is referred to as the fixed-point function.
It is worth noting that if the Jacobian of the nonlinear operator $\myvec{f}$ 
is H\"older continuous with exponent $\nu \in (0,1]$,
then the Jacobian of the fixed-point function $\myvec{g}$ is also H\"older continuous with the same exponent $\nu$.
Indeed, for any $\myvec{x}, \myvec{y} \in \DS$, we have
$$
\|\myvec{g}'(\myvec{x}) - \myvec{g}'(\myvec{y})\|
  = \|(\myvec{g}'(\myvec{x}) - I) - (\myvec{g}'(\myvec{y}) - I)\|
  = \|\myvec{f}'(\myvec{x}) - \myvec{f}'(\myvec{y})\|
  \leq H_\nu \|\myvec{x} - \myvec{y}\|^{\nu}.
$$
We say that the fixed-point function $\myvec{g}$ is contractive 
if there exists a constant $\theta \in (0,1)$ such that
\begin{equation}
  \label{cond:Contractive}
  \|\myvec{g}(\myvec{x}) - \myvec{g}(\myvec{y})\| \leq \theta \|\myvec{x} - \myvec{y}\|
\end{equation}
holds for all $\myvec{x}, \myvec{y} \in \DS$.
The constant $\theta$ is referred the contraction factor of $\myvec{g}$.
This condition implies that
\begin{equation}
  \label{cond:Contractive2}
  \|\myvec{g}'(\myvec{x})\| \leq \theta < 1
\end{equation}
holds for any $\myvec{x} \in \DS$.
It is guaranteed by the contraction mapping theorem \cite{OrtegaR1970}
that $\myvec{g}$ has a unique fixed point $\myvec{x}^* \in \DS$, 
which is the unique solution of the nonlinear system \eqref{eq:f(x)=0}.

The following lemma will be used in the convergence analysis of Anderson acceleration.

\begin{mylemma}
  \label{lem:norm_f(x)}
  Assume that there is $\myvec{x}^* \in \RS^n$ such that $\myvec{f}(\myvec{x}^*) = \myvec{0}$
  and $\myvec{f}'(\myvec{x}^*)^{-1}$ exists.
  If the Jacobian $\myvec{f}'$ is H\"older continuous with exponent $\nu$ 
  in $\ball(\myvec{x}^*,r_\nu)$, where
  \begin{equation}
    \label{cons:r_nu}
    r_\nu = \left(\frac{1}{H_\nu \|\myvec{f}'(\myvec{x}^*)^{-1}\|}\right)^{1/\nu},
  \end{equation}
  then for any $\myvec{x} \in \ball(\myvec{x}^*,r_\nu)$ we have
  \begin{equation}
    \label{ineq:norm_f(x)}
    \frac{\nu}{(1+\nu)\|\myvec{f}'(\myvec{x}^*)^{-1}\|} \|\myvec{x} - \myvec{x}^*\|
      \leq \|\myvec{f}(\myvec{x})\| 
      \leq \frac{(2+\nu)\|\myvec{f}'(\myvec{x}^*)\|}{1+\nu} \|\myvec{x} - \myvec{x}^*\|.
  \end{equation}
\end{mylemma}

\begin{proof}
  Take $\myvec{x} \in \ball(\myvec{x}^*,r_\nu)$. Since
  \begin{align*}
    \myvec{f}(\myvec{x}) 
      &= \myvec{f}(\myvec{x}) - \myvec{f}(\myvec{x}^*)
        = \int_0^1 \myvec{f}'(\myvec{x}^* + t(\myvec{x} - \myvec{x}^*)) (\myvec{x} - \myvec{x}^*)\dif{t} \\
      &= \int_0^1 [\myvec{f}'(\myvec{x}^* + t(\myvec{x} - \myvec{x}^*)) - \myvec{f}'(\myvec{x}^*)] (\myvec{x} - \myvec{x}^*)\dif{t}
        + \myvec{f}'(\myvec{x}^*) (\myvec{x} - \myvec{x}^*),
  \end{align*}
  it follows from the H\"older condition \eqref{cond:Holder} that
  \begin{align*}
    \|\myvec{f}(\myvec{x})\|
      &\leq \int_0^1 \|\myvec{f}'(\myvec{x}^* + t(\myvec{x} - \myvec{x}^*)) - \myvec{f}'(\myvec{x}^*)\| \|\myvec{x} - \myvec{x}^*\|\dif{t}
        + \|\myvec{f}'(\myvec{x}^*)\| \|\myvec{x} - \myvec{x}^*\| \\
      &\leq H_\nu \int_0^1 t^\nu\|\myvec{x} - \myvec{x}^*\|^{1+\nu} \dif{t} + \|\myvec{f}'(\myvec{x}^*)\| \|\myvec{x} - \myvec{x}^*\| \\
      &= \left(\frac{H_\nu}{1+\nu} \|\myvec{x} - \myvec{x}^*\|^\nu + \|\myvec{f}'(\myvec{x}^*)\|\right) \|\myvec{x} - \myvec{x}^*\|.
  \end{align*}
  Noting that $\|\myvec{x} - \myvec{x}^*\| < r_\nu$, we obtain
  $$
  \frac{H_\nu}{1+\nu} \|\myvec{x} - \myvec{x}^*\|^\nu
    < \frac{1}{(1+\nu)\|\myvec{f}'(\myvec{x}^*)^{-1}\|}
    \leq \frac{\|\myvec{f}'(\myvec{x}^*)\|}{1+\nu}.
  $$
  This leads to
  $$
  \|\myvec{f}(\myvec{x})\| 
    \leq \left(\frac{2+\nu}{1+\nu}\right) \|\myvec{f}'(\myvec{x}^*)\| \|\myvec{x} - \myvec{x}^*\|.
  $$
  On the other hand, we apply standard analytical techniques to deduce that
  \begin{align*}
    \myvec{f}'(\myvec{x}^*)^{-1}\myvec{f}(\myvec{x})
      &= \myvec{f}'(\myvec{x}^*)^{-1} 
        \int_0^1 \myvec{f}'(\myvec{x}^* + t(\myvec{x} - \myvec{x}^*)) (\myvec{x} - \myvec{x}^*)\dif{t} \\
      &= (\myvec{x} - \myvec{x}^*) - 
        \int_0^1 \myvec{f}'(\myvec{x}^*)^{-1}[\myvec{f}'(\myvec{x}^*) 
          - \myvec{f}'(\myvec{x}^* + t(\myvec{x} - \myvec{x}^*))](\myvec{x} - \myvec{x}^*)\dif{t}.
  \end{align*}
  Then, the H\"{o}lder condition \eqref{cond:Holder} can be applied again to yield
  \begin{align*}
    \|\myvec{f}'(\myvec{x}^*)^{-1}\myvec{f}(\myvec{x})\|
      &\geq \|\myvec{x} - \myvec{x}^*\| - 
        H_\nu\|\myvec{f}'(\myvec{x}^*)^{-1}\| \int_0^1 t^\nu\|\myvec{x} - \myvec{x}^*\|^{1+\nu}\dif{t}\\
      &= \left(1 - \frac{H_\nu}{1+\nu}\|\myvec{f}'(\myvec{x}^*)^{-1}\| \|\myvec{x} - \myvec{x}^*\|^\nu\right) 
        \|\myvec{x} - \myvec{x}^*\|.
  \end{align*}
  Thus, we have
  \begin{align*}
    \|\myvec{f}(\myvec{x})\| 
      \geq \frac{\|\myvec{f}'(\myvec{x}^*)^{-1}\myvec{f}(\myvec{x})\|}{\|\myvec{f}'(\myvec{x}^*)^{-1}\|}
      &\geq \left(\frac{1}{\|\myvec{f}'(\myvec{x}^*)^{-1}\|} - \frac{H_\nu\|\myvec{x} - \myvec{x}^*\|^\nu}{1+\nu}\right)
        \|\myvec{x} - \myvec{x}^*\| \\
      &\geq \left(\frac{\nu}{(1+\nu)\|\myvec{f}'(\myvec{x}^*)^{-1}\|}\right)\|\myvec{x} - \myvec{x}^*\|.
  \end{align*}
  With this, the proof of the lemma is complete.
\end{proof}

\begin{myremark}
  By applying the Banach lemma \cite{OrtegaR1970} with standard techniques, as presented in \cite{Ferreira2009}, 
  we can show that the H\"older continuity of $\myvec{f}'$ with exponent $\nu$ 
  in $\ball(\myvec{x}^*,r_\nu)$ guarantees that $\myvec{f}'(\myvec{x})$ is nonsingular for any $\myvec{x} \in \ball(\myvec{x}^*,r_\nu)$,
  where $r_\nu$ is given by \eqref{cons:r_nu}.
\end{myremark}

Applying the above lemma, we arrive at the following result.

\begin{mylemma}
  Let the assumptions of Lemma $\ref{lem:norm_f(x)}$ hold.
  If $\myvec{x}_0 \in \ball(\myvec{x}^*,r_\nu)$, then for any $\myvec{x} \in \ball(\myvec{x}^*,r_\nu)$,
  we have
  \begin{equation}
    \label{ineq:norm_f(x)/f(x0)}
    \frac{\nu}{(2+\nu)\kappa(\myvec{f}'(\myvec{x}^*))}
      \cdot \frac{\|\myvec{x} - \myvec{x}^*\|}{\|\myvec{x}_0 - \myvec{x}^*\|}
    \leq \frac{\|\myvec{f}(\myvec{x})\|}{\|\myvec{f}(\myvec{x}_0)\|}
    \leq \frac{(2+\nu)\kappa(\myvec{f}'(\myvec{x}^*))}{\nu}
      \cdot \frac{\|\myvec{x} - \myvec{x}^*\|}{\|\myvec{x}_0 - \myvec{x}^*\|}.
  \end{equation}
\end{mylemma}

The following lemma is also used in the convergence analysis of Anderson acceleration. 
The proof is straightforward and hence omitted.

\begin{mylemma}
  \label{lem:q}
  Let $m$ be a nonnegative integer and $\tau \in (0,1)$.
  If $0 \leq \zeta < 1 - \tau$, then the equation
  \begin{equation}
    \label{eq:q}
    q^{m+1} - \tau q^m - \zeta = 0
  \end{equation}
  has a unique root in the open interval $\big(m\tau/(m+1),1\big)$.
\end{mylemma}

In \cite{TothK2015}, a key idea in establishing convergence is to assume that the inequality
$$
  \left\|\sum_{j=0}^{m_k} \alpha_j^{(k)} \myvec{f}(\myvec{x}_{k-m_k+j})\right\| \leq \|\myvec{f}(\myvec{x}_k)\|
$$
holds when $\{\alpha_j^{(k)}\}_{j=0}^{m_k}$ is the solution of the optimization problem \eqref{optim:alpha}. 
This assumption is relaxed in \cite{ChenK2019}, 
where it is only required that norm of the linear combination of residuals does not exceed 
that of the most recent residual. 
To understand how Anderson acceleration achieves faster convergence, 
Pollock et al. introduced in \cite{PollockRX2019} the following optimization gain
\begin{equation}
  \label{cons:eta_k}
  \eta_k := \frac{\left\|\mysum_{j=0}^{m_k} \alpha_j^{(k)} \myvec{f}(\myvec{x}_{k-m_k+j})\right\|}{\|\myvec{f}(\myvec{x}_k)\|}.
\end{equation}
This quantity, satisfying $0 \leq \eta_k \leq 1$, was subsequently employed in \cite{EvansPRX2020,PollockR2021,RebholzX2023,PollockR2023}
to demonstrate that Anderson acceleration improves linear convergence factor of fixed-point iterations.
Using the optimization gain, we establish in this work the local convergence of Anderson acceleration 
for depth $m=1$, assuming H\"older continuity of the derivative.

\section{Convergence results}
\label{sec:ConvResults}

This section is devoted to the local convergence analysis of Anderson acceleration
for solving the nonlinear system \eqref{eq:f(x)=0} with H\"older continuous derivatives.

\subsection{Local convergence}

We first present a local convergence result for Anderson acceleration with general depths $m$. 
The following theorem establishes that the sequence $\{\myvec{x}_k\}$ generated by Algorithm \ref{alg:AA}
converges R-linearly to the solution $\myvec{x}^*$ of the nonlinear system \eqref{eq:f(x)=0}
under the assumptions that the Jacobian $\myvec{f}'$ is H\"older continuous with exponent $\nu \in (0,1]$
and the fixed-point function $\myvec{g}$ is contractive with factor $\theta \in (0,1)$ 
in a prescribed ball of the solution $\myvec{x}^*$.

\begin{mytheorem}
  \label{th:LocalConvAnderson(m)}
  Assume that there is an $\myvec{x}^* \in \RS^n$ such that $\myvec{f}(\myvec{x}^*) = \myvec{0}$
  and $\myvec{f}'(\myvec{x}^*)$ is nonsingular. Suppose that:
  \begin{enumerate}[label=\textup{(\roman*)}]
    \item There is a constant $M_\alpha$ such that $\mysum_{j=0}^{m_k} |\alpha_j^{(k)}| \leq M_\alpha$ for all $k \geq 0$.
    \item The Jacobian $\myvec{f}'$ is H\"older continuous with exponent $\nu \in (0,1]$ in $\ball(\myvec{x}^*,\hat{r}_\nu)$,
    where $0 < \hat{r}_\nu \leq r_\nu$, $r_\nu$ is defined in \eqref{cons:r_nu}.
    \item The fixed-point function $\myvec{g}$ is contractive with factor $\theta \in (0,1)$ in $\ball(\myvec{x}^*,\hat{r}_\nu)$.
  \end{enumerate}
  Let $\{\myvec{x}_k\}$ be the sequence generated by Algorithm $\ref{alg:AA}$
  with starting point $\myvec{x}_0 \in \ball(\myvec{x}^*,r)$,
  where $r := \hat{r}_\nu/M_\alpha$.
  Set $\tau := \theta\bar\eta$ with $\bar\eta := \sup\limits_{k \geq 0} \eta_k$,
  where $\eta_k$ is the optimization gain defined by \eqref{cons:eta_k}.
  Define
  \begin{equation}
    \label{cons:zeta}
    \zeta := \frac{(2+\nu)^\nu H_\nu M_\alpha (1 + M_\alpha^\nu)}{\nu^{1+\nu}}
  \cdot\kappa^\nu(\myvec{f}'(\myvec{x}^*)) \|\myvec{f}'(\myvec{x}^*)^{-1}\|\|\myvec{x}_0 - \myvec{x}^*\|^\nu.
  \end{equation}
  If $\zeta < 1 - \tau$ and $\theta M_\alpha \leq 1$,
  then the sequence $\{\myvec{x}_k\}$ is contained in $\ball(\myvec{x}^*,r)$ 
  and converges R-linearly to the solution $\myvec{x}^*$
  in the sense that
  \begin{equation}
    \label{ineq:ConvRate}
    \limsup_{k \to \infty} \left(\frac{\|\myvec{f}(\myvec{x}_{k})\|}{\|\myvec{f}(\myvec{x}_0)\|}\right)^{1/k} \leq q,
  \end{equation}
  where $q$ is the unique root of equation \eqref{eq:q} in the interval $\big(m\tau/(m+1),1\big)$.
\end{mytheorem}

\begin{proof}
  We proceed by induction. The assumption on the history that 
  $\|\myvec{f}(\myvec{x}_{\ell})\| \leq q^\ell\|\myvec{f}(\myvec{x}_0)\|$
  leads to \eqref{ineq:ConvRate} being satisfied for $0 \leq k \leq m$. 
  We now assume that $\myvec{x}_\ell \in \ball(\myvec{x}^*,r)$ 
  and \eqref{ineq:ConvRate} holds for all $0\leq \ell \leq k$ with $k > m$.
  Since $\sum_{j=0}^{m_k} \alpha_j^{(k)} = 1$, we have $M_\alpha \geq 1$,
  which gives that $r \leq \hat{r}_\nu \leq r_\nu$.
  Set
  $$
  \myvec{z}_k := \sum_{j=0}^{m_k} \alpha_j^{(k)} \myvec{x}_{k-m_k+j}.
  $$
  We first show that $\myvec{z}_k \in \ball(\myvec{x}^*,\hat{r}_\nu)$. In fact, since
  $$
  \myvec{z}_k - \myvec{x}^* 
  = \sum_{j=0}^{m_k} \alpha_j^{(k)} \myvec{x}_{k-m_k+j} - \myvec{x}^*
  = \sum_{j=0}^{m_k} \alpha_j^{(k)} (\myvec{x}_{k-m_k+j} - \myvec{x}^*),
  $$  
  we use the inductive hypothesis to yield
  $$
  \|\myvec{z}_k - \myvec{x}^*\|
    \leq \sum_{j=0}^{m_k} |\alpha_j^{(k)}| \|\myvec{x}_{k-m_k+j} - \myvec{x}^*\|
    \leq M_\alpha \|\myvec{x}_0 - \myvec{x}^*\|
    < M_\alpha r = \hat{r}_\nu.
  $$
  This means that $\myvec{z}_k \in \ball(\myvec{x}^*,\hat{r}_\nu)$.
  Next, by using the contractivity of the fixed-point function $\myvec{g}$ in $\ball(\myvec{x}^*,r)$
  and the inductive hypothesis, 
  one has that
  \begin{align*}
    \|\myvec{x}_{k+1} - \myvec{x}^*\|
      &= \left\|\sum_{j=0}^{m_k} \alpha_j^{(k)} \big[\myvec{g}(\myvec{x}_{k-m_k+j}) - \myvec{g}(\myvec{x}^*)\big]\right\| \\
      &\leq \theta \sum_{j=0}^{m_k} |\alpha_j^{(k)}| \|\myvec{x}_{k-m_k+j} - \myvec{x}^*\| \\
      &\leq \theta M_\alpha \|\myvec{x}_0 - \myvec{x}^*\|
      \leq \|\myvec{x}_0 - \myvec{x}^*\|,
  \end{align*}
  which gives that $\myvec{x}_{k+1} \in \ball(\myvec{x}^*,r)$.

  To show \eqref{ineq:ConvRate} by induction, we first use \eqref{ineq:norm_f(x)} to yield 
  \begin{align*}
    \|\myvec{z}_k - \myvec{x}^*\|
    &\leq \sum_{j=0}^{m_k} |\alpha_j^{(k)}| \|\myvec{x}_{k-m_k+j} - \myvec{x}^*\| \\
    &\leq \frac{1+\nu}{\nu}\|\myvec{f}'(\myvec{x}^*)^{-1}\| \sum_{j=0}^{m_k} |\alpha_j^{(k)}| \|\myvec{f}(\myvec{x}_{k-m_k+j})\| \\
    &\leq \frac{1+\nu}{\nu}\|\myvec{f}'(\myvec{x}^*)^{-1}\| \sum_{j=0}^{m_k} |\alpha_j^{(k)}| \cdot q^{k-m_k+j}\|\myvec{f}(\myvec{x}_0)\|.
  \end{align*}
  In light of the fact that $k - m_k + j = k - \min\{m,k\} + j \geq k - m$,
  we further obtain that 
  \begin{align}
    \|\myvec{z}_k - \myvec{x}^*\|
    &\leq \frac{1+\nu}{\nu}\|\myvec{f}'(\myvec{x}^*)^{-1}\| \sum_{j=0}^{m_k} |\alpha_j^{(k)}| \cdot q^{k-m}\|\myvec{f}(\myvec{x}_0)\| \nonumber \\
    &\leq \frac{1+\nu}{\nu}\|\myvec{f}'(\myvec{x}^*)^{-1}\| M_\alpha \cdot q^{k-m}\|\myvec{f}(\myvec{x}_0)\|. \label{ineq:norm_zk-x*}
  \end{align}
  To estimate the bound of $\|\myvec{f}(\myvec{x}_{k+1})\|$, we notice that
  \begin{equation}
    \label{eq:fk+1}
    \myvec{f}(\myvec{x}_{k+1}) = \myvec{g}(\myvec{x}_{k+1}) - \myvec{x}_{k+1}
    = \big[\myvec{g}(\myvec{x}_{k+1}) - \myvec{g}(\myvec{z}_k)\big]
      + \big[\myvec{g}(\myvec{z}_k) - \myvec{x}_{k+1}\big].
  \end{equation}
  On the one hand, we have from the contractivity of $\myvec{g}$ in $\ball(\myvec{x}^*,\hat{r}_\nu)$ that
  \begin{align}
    \|\myvec{g}(\myvec{x}_{k+1}) - \myvec{g}(\myvec{z}_k)\|
      &\leq \theta \|\myvec{x}_{k+1} - \myvec{z}_k\| \nonumber\\
      &= \theta \left\|\sum_{j=0}^{m_k} \alpha_j^{(k)}[\myvec{g}(\myvec{x}_{k-m_k+j}) - \myvec{x}_{k-m_k+j}]\right\| \nonumber\\
      &= \theta \left\|\sum_{j=0}^{m_k} \alpha_j^{(k)} \myvec{f}(\myvec{x}_{k-m_k+j})\right\| \nonumber \\
      &= \theta\eta_k \|\myvec{f}(\myvec{x}_k)\|. \label{ineq:norm_g(xk+1)-g(z_k)}
  \end{align}
  On the other hand, since
  \begin{align*}
    \myvec{g}(\myvec{z}_k) 
      &= \myvec{g}(\myvec{z}_k) - \myvec{g}(\myvec{x}^*) + \myvec{g}(\myvec{x}^*) \\
      &= \myvec{g}(\myvec{x}^*) + \myvec{g}'(\myvec{x}^*) (\myvec{z}_k - \myvec{x}^*) 
        + \int_0^1 [\myvec{g}'(\myvec{x}^* + t(\myvec{z}_k - \myvec{x}^*)) - \myvec{g}'(\myvec{x}^*)](\myvec{z}_k - \myvec{x}^*)\dif{t} \\
      &= \sum_{j=0}^{m_k}\alpha_j^{(k)} 
      \big[\myvec{g}(\myvec{x}^*) + \myvec{g}'(\myvec{x}^*)(\myvec{x}_{k-m_k+j} - \myvec{x}^*)\big] \\
      &\quad + \int_0^1 [\myvec{g}'(\myvec{x}^* + t(\myvec{z}_k - \myvec{x}^*)) - \myvec{g}'(\myvec{x}^*)](\myvec{z}_k - \myvec{x}^*)\dif{t},
  \end{align*}
  we have
  \begin{align*}
    \myvec{g}(\myvec{z}_k) - \myvec{x}_{k+1}
      &= \sum_{j=0}^{m_k}\alpha_j^{(k)} 
        \big[\myvec{g}(\myvec{x}^*) + \myvec{g}'(\myvec{x}^*)(\myvec{x}_{k-m_k+j} - \myvec{x}^*)\
          - \myvec{g}(\myvec{x}_{k-m_k+j})\big] \\
      &\quad + \int_0^1 [\myvec{g}'(\myvec{x}^* + t(\myvec{z}_k - \myvec{x}^*)) 
          - \myvec{g}'(\myvec{x}^*)](\myvec{z}_k - \myvec{x}^*)\dif{t} \\
      &= - \sum_{j=0}^{m_k} \alpha_j^{(k)}
        \left(\int_0^1 \big[\myvec{g}'(\myvec{x}^* + t(\myvec{x}_{k-m_k+j} - \myvec{x}^*)) 
          - \myvec{g}'(\myvec{x}^*)\big](\myvec{x}_{k-m_k+j} - \myvec{x}^*)\dif{t}\right) \\
      &\quad + \int_0^1 [\myvec{g}'(\myvec{x}^* + t(\myvec{z}_k - \myvec{x}^*)) 
          - \myvec{g}'(\myvec{x}^*)](\myvec{z}_k - \myvec{x}^*)\dif{t}.
  \end{align*}
  Recall that if the Jacobian of the nonlinear operator $\myvec{f}$ 
  is H\"older continuous with exponent $\nu \in (0,1]$,
  then the Jacobian of the corresponding fixed-point function $\myvec{g}$ 
  inherits the same H\"older continuity with the same exponent $\nu$.
  Then, the H\"older condition \eqref{cond:Holder} is applicable to deduce that
  \begin{align*}
    \|\myvec{g}(\myvec{z}_k) - \myvec{x}_{k+1}\|
      &\leq H_\nu \sum_{j=0}^{m_k} |\alpha_j^{(k)}| \int_0^1 t^\nu \|\myvec{x}_{k-m_k+j} - \myvec{x}^*\|^{1+\nu}\dif{t} \\
      &\quad + H_\nu \int_0^1 t^\nu \|\myvec{z}_k - \myvec{x}^*\|^{1+\nu}\dif{t} \\
      &= \frac{H_\nu}{1+\nu} \left(\|\myvec{z}_k - \myvec{x}^*\|^{1+\nu}
        + \sum_{j=0}^{m_k} |\alpha_j^{(k)}| \|\myvec{x}_{k-m_k+j} - \myvec{x}^*\|^{1+\nu}\right).
  \end{align*}
  It follows from \eqref{ineq:norm_zk-x*} and \eqref{ineq:norm_f(x)} that
  \begin{align*}
    \lefteqn{\|\myvec{z}_k - \myvec{x}^*\|^{1+\nu}} \\
    &\leq \frac{(1+\nu)^{1+\nu}}{\nu^{1+\nu}} \left(q^{k-m}\right)^{1+\nu}
      M_\alpha^{1+\nu} \|\myvec{f}'(\myvec{x}^*)^{-1}\|^{1+\nu} \|\myvec{f}(\myvec{x}_0)\|^{1+\nu} \\
    &\leq \frac{(1+\nu)^{1+\nu}}{\nu^{1+\nu}} \left(q^{k-m}\right)^{1+\nu}
    M_\alpha^{1+\nu} \|\myvec{f}'(\myvec{x}^*)^{-1}\|^{1+\nu}
    \left(\frac{2+\nu}{1+\nu}\|\myvec{f}'(\myvec{x}^*)\|\|\myvec{x}_0 - \myvec{x}^*\|\right)^\nu\|\myvec{f}(\myvec{x}_0)\| \\
    &= \frac{(1+\nu)(2+\nu)^\nu}{\nu^{1+\nu}}q^{(1+\nu)(k-m)}M_\alpha^{1+\nu} 
    \|\myvec{f}'(\myvec{x}^*)^{-1}\|\kappa^\nu(\myvec{f}'(\myvec{x}^*))
    \|\myvec{x}_0 - \myvec{x}^*\|^\nu \|\myvec{f}(\myvec{x}_0)\|.
  \end{align*}
  In addition, by using the same argument as above, we have
  \begin{align*}
    \lefteqn{\sum_{j=0}^{m_k} |\alpha_j^{(k)}| \|\myvec{x}_{k-m_k+j} - \myvec{x}^*\|^{1+\nu} } \\
    &\leq \sum_{j=0}^{m_k} |\alpha_j^{(k)}| 
      \left(\frac{1+\nu}{\nu}\|\myvec{f}'(\myvec{x}^*)^{-1}\|\|\myvec{f}(\myvec{x}_{k-m_k+j})\|\right)^{1+\nu} \\
    &\leq \frac{(1+\nu)(2+\nu)^\nu}{\nu^{1+\nu}}q^{(1+\nu)(k-m)}M_\alpha
    \|\myvec{f}'(\myvec{x}^*)^{-1}\|\kappa^\nu(\myvec{f}'(\myvec{x}^*))
    \|\myvec{x}_0 - \myvec{x}^*\|^\nu \|\myvec{f}(\myvec{x}_0)\|.
  \end{align*}
  The above two estimates allow us to get
  \begin{align*}
    \|\myvec{g}(\myvec{z}_k) - \myvec{x}_{k+1}\|
    &\leq \frac{H_\nu M_\alpha(1+M_\alpha^\nu)(2+\nu)^\nu}{\nu^{1+\nu}}q^{(1+\nu)(k-m)}
    \kappa^\nu(\myvec{f}'(\myvec{x}^*))\|\myvec{f}'(\myvec{x}^*)^{-1}\|
    \|\myvec{x}_0 - \myvec{x}^*\|^\nu \|\myvec{f}(\myvec{x}_0)\|.
  \end{align*}
  This together with \eqref{eq:fk+1} and \eqref{ineq:norm_g(xk+1)-g(z_k)} leads to
  \begin{align*}
    \lefteqn{\|\myvec{f}(\myvec{x}_{k+1})\| } \\
    &\leq \|\myvec{g}(\myvec{x}_{k+1}) - \myvec{g}(\myvec{z}_k)\| + \|\myvec{g}(\myvec{z}_k) - \myvec{x}_{k+1}\| \\
    &\leq q^k\|\myvec{f}(\myvec{x}_0)\|\left[\theta\eta_k 
    + \frac{H_\nu M_\alpha(1+M_\alpha^\nu)(2+\nu)^\nu}{\nu^{1+\nu}}q^{(1+\nu)(k-m)-k}
    \kappa^\nu(\myvec{f}'(\myvec{x}^*))\|\myvec{f}'(\myvec{x}^*)^{-1}\|
    \|\myvec{x}_0 - \myvec{x}^*\|^\nu\right] \\
    &\leq q^k\|\myvec{f}(\myvec{x}_0)\|\left[\theta\eta_k 
    + \frac{H_\nu M_\alpha(1+M_\alpha^\nu)(2+\nu)^\nu}{\nu^{1+\nu}}q^{-m}
    \kappa^\nu(\myvec{f}'(\myvec{x}^*))\|\myvec{f}'(\myvec{x}^*)^{-1}\|
    \|\myvec{x}_0 - \myvec{x}^*\|^\nu\right].
  \end{align*}
  Noting that $\tau = \theta\bar\eta  \in (0,1)$, 
  and that $\zeta$, as given in \eqref{cons:zeta}, satisfies $\zeta < 1 - \tau$,
  we conclude from Lemma \ref{lem:q} that
  $$
  \|\myvec{f}(\myvec{x}_{k+1})\|
  \leq (\tau + \zeta q^{-m})q^k \|\myvec{f}(\myvec{x}_0)\|
  = q \cdot q^k \|\myvec{f}(\myvec{x}_0)\| = q^{k+1} \|\myvec{f}(\myvec{x}_0)\|.
  $$
  Therefore, by induction, all claims in the theorem are verified. 
  The proof is complete.
\end{proof}

\begin{myremark}
  Theorem \ref{th:LocalConvAnderson(m)} shows that the sequence $\{\myvec{x}_k\}$ converges R-linearly to the solution $\myvec{x}^*$.
  The convergence factor is determined by the unique root $q$ of equation \eqref{eq:q} in the interval $\big(m\tau/(m+1),1\big)$.
  In particular, if $m=1$, then $q = (\tau + \sqrt{\tau^2 + 4\zeta})/2$.
  This means that the convergence factor is determined by 
  the contraction factor $\theta$, the optimization gain $\eta_k$ and the quantity $\zeta$.
\end{myremark}

\begin{myremark}
  \label{remark:ConvergenceRadius}
  The inequality $\zeta < 1 - \tau$ in Theorem \ref{th:LocalConvAnderson(m)} can generally be satisfied 
  when the initial value $\myvec{x}_0$ is taken sufficiently close to the solution $\myvec{x}^*$.
  More precisely, define
  \begin{equation}
    \label{cons:r_nu_tilde}
    \tilde{r}_\nu := \left(\frac{(1-\theta)\nu}{M_\alpha(1+M_\alpha^\nu)}\right)^{1/\nu}
    \cdot \frac{\nu r_\nu}{(2 + \nu) \kappa(\myvec{f}'(\myvec{x}^*))},
  \end{equation}
  where $r_\nu$ is defined in \eqref{cons:r_nu}.
  Then we obtain
  $$
  \zeta = (1-\theta) \left(\frac{\|\myvec{x}_0 - \myvec{x}^*\|}{\tilde{r}_\nu}\right)^\nu.
  $$
  This means that $\|\myvec{x}_0 - \myvec{x}^*\| < \tilde{r}_\nu$
  if and only if $\zeta < 1 - \theta$.
  Since $\tau = \theta\bar\eta \leq \theta$, 
  it follows that $\|\myvec{x}_0 - \myvec{x}^*\| < \tilde{r}_\nu$ implies $\zeta < 1 - \tau$.
  Thus, if we set $r = \min\{\hat{r}_\nu/M_\alpha,\tilde{r}_\nu\}$,
  then the condition $\zeta < 1 - \tau$ is satisfied for every $\myvec{x}_0 \in \ball(\myvec{x}^*,r)$.
  Consequently, in Theorem \ref{th:LocalConvAnderson(m)} 
  this condition need not be verified separately once $\myvec{x}_0 \in \ball(\myvec{x}^*,r)$ is assumed.
  An important observation from \eqref{cons:r_nu_tilde} is that 
  a well-conditioned Jacobian $\myvec{f}'(\myvec{x}^*)$ yields a large $\tilde{r}_\nu$.
  The convergence ball $\ball(\myvec{x}^*,r)$ can accordingly become larger,
  which in turn makes the condition $\zeta < 1 - \tau$ easier to satisfy.
\end{myremark}

\begin{myremark}
  The H\"older exponent $\nu \in (0,1]$ affects the convergence radius,
  the convergence factor and the associated constants in
  Theorem \ref{th:LocalConvAnderson(m)}.
  Specifically, the H\"older continuity of the Jacobian yields nonlinear remainder terms of order $1+\nu$.
  Estimating these terms together with \eqref{ineq:norm_f(x)} leads to 
  the $\nu$-dependent quantity $\zeta$ given in \eqref{cons:zeta},
  which involves the factor $(2 + \nu)^\nu/\nu^{1+\nu}$, the H\"older constant $H_\nu$,
  the condition-number factor $\kappa^{\nu}(\myvec{f}'(\myvec{x}^*))$,
  and the term $\|\myvec{x}_0 - \myvec{x}^*\|^\nu$.
  Since the convergence factor $q$ is the unique root of \eqref{eq:q},
  the H\"older exponent $\nu$ affects $q$ through $\zeta$.
  For fixed $m$ and $\tau$, the convergence factor $q$ is strictly increasing with respect to $\zeta$.
  The guaranteed convergence radius $r$ also depends on $\nu$
  through $r_\nu$ and $\tilde{r}_\nu$, given in \eqref{cons:r_nu} and \eqref{cons:r_nu_tilde}, respectively.
  In general, however, these $\nu$-dependent expressions do not imply a monotonic dependence 
  of the convergence radius or the convergence factor on $\nu$.
  For the problem considered in Subsection \ref{subsec:CDR}, Table \ref{tab:CDR_HolderExp}
  quantifies this dependence and shows that the guaranteed convergence radius $r$ decreases
  and the convergence factor $q$ becomes less favorable as $\nu$ decreases.
\end{myremark}

The R-linear convergence of the error with R-factor $q$ is a direct consequence of Theorem \ref{th:LocalConvAnderson(m)}.

\begin{mycorollary}
  \label{cor:RLinearConvAnderson(m)}
  Let the assumptions of Theorem $\ref{th:LocalConvAnderson(m)}$ hold.
  Then we have
  $$
  \limsup_{k \to \infty} \left(\frac{\|\myvec{x}_k - \myvec{x}^*\|}{\|\myvec{x}_0 - \myvec{x}^*\|}\right)^{1/k} \leq q.
  $$
\end{mycorollary}

\begin{proof}
  Noting from \eqref{ineq:norm_f(x)/f(x0)} that
  $$
  \frac{\|\myvec{x}_k - \myvec{x}^*\|}{\|\myvec{x}_0 - \myvec{x}^*\|}
    \leq  \frac{(2+\nu)\kappa(\myvec{f}'(\myvec{x}^*))}{\nu} 
    \cdot \frac{\|\myvec{f}(\myvec{x}_k)\|}{\|\myvec{f}(\myvec{x}_0)\|}.
  $$
  This allows us to obtain from \eqref{ineq:ConvRate} that
  $$
  \limsup_{k \to \infty} \left(\frac{\|\myvec{x}_k - \myvec{x}^*\|}{\|\myvec{x}_0 - \myvec{x}^*\|}\right)^{1/k}
  \leq \lim_{k \to \infty} \left[\frac{2+\nu}{\nu}\kappa(\myvec{f}'(\myvec{x}^*))\right]^{1/k}
    \limsup_{k \to \infty} \left(\frac{\|\myvec{f}(\myvec{x}_k)\|}{\|\myvec{f}(\myvec{x}_0)\|}\right)^{1/k} 
  \leq q,
  $$
  which yields the desired result.
\end{proof}

For the case $\nu = 1$, 
the H\"older continuity assumption on the Jacobian $\myvec{f}'$ reduces to the classical Lipschitz continuity:
\begin{equation}
  \label{cond:Lipschitz}
  \|\myvec{f}'(\myvec{x}) - \myvec{f}'(\myvec{y})\| \leq L\|\myvec{x} - \myvec{y}\|, 
    \quad \myvec{x}, \myvec{y} \in \ball(\myvec{x}^*,\hat{r}),
\end{equation}
where $0 < \hat{r} \leq (L\|\myvec{f}'(\myvec{x}^*)^{-1}\|)^{-1}$.
Then we have the following local convergence result from 
Theorem \ref{th:LocalConvAnderson(m)} and Corollary \ref{cor:RLinearConvAnderson(m)}
for Anderson acceleration under Lipschitz condition \eqref{cond:Lipschitz}.

\begin{mycorollary}
  \label{cor:LocalConvLipCondAnderson(m)}
  Assume that there is an $\myvec{x}^* \in \RS^n$ such that $\myvec{f}(\myvec{x}^*) = \myvec{0}$
  and $\myvec{f}'(\myvec{x}^*)$ is nonsingular. Suppose that:
  \begin{enumerate}[label=\textup{(\roman*)}]
    \item There is a constant $M_\alpha$ such that $\sum_{j=0}^{m_k} |\alpha_j^{(k)}| \leq M_\alpha$ for all $k \geq 0$.
    \item The Jacobian $\myvec{f}'$ satisfies the Lipschitz condition \eqref{cond:Lipschitz} in $\ball(\myvec{x}^*,\hat{r})$,
    with $0 < \hat{r} \leq (L\|\myvec{f}'(\myvec{x}^*)^{-1}\|)^{-1}$.
    \item The fixed-point function $\myvec{g}$ is contractive with constant $\theta \in (0,1)$ in $\ball(\myvec{x}^*,\hat{r})$.
  \end{enumerate}
  Set $r := \hat{r}/M_\alpha$, 
  $\bar\eta := \sup_{k \geq 0} \eta_k$
  and $\tau := \theta\bar\eta$,
  where $\eta_k$ is the optimization gain defined by \eqref{cons:eta_k}.
  Let $\{\myvec{x}_k\}$ be the sequence generated by Algorithm $\ref{alg:AA}$
  with starting point $\myvec{x}_0 \in \ball(\myvec{x}^*,r)$,
  and set $\zeta := 3L M_\alpha (1 + M_\alpha)
      \kappa(\myvec{f}'(\myvec{x}^*)) \|\myvec{f}'(\myvec{x}^*)^{-1}\|\|\myvec{x}_0 - \myvec{x}^*\|$.
  If $\zeta < 1 - \tau$ and $\theta M_\alpha \leq 1$,
  then the sequence $\{\myvec{x}_k\}$ is contained in $\ball(\myvec{x}^*,r)$ 
  and converges R-linearly to the solution $\myvec{x}^*$
  in the sense that
  \begin{equation*}
    \limsup_{k \to \infty} \left(\frac{\|\myvec{f}(\myvec{x}_{k})\|}{\|\myvec{f}(\myvec{x}_0)\|}\right)^{1/k}
      \leq q \quad \text{and} \quad
    \limsup_{k \to \infty} \left(\frac{\|\myvec{x}_k - \myvec{x}^*\|}{\|\myvec{x}_0 - \myvec{x}^*\|}\right)^{1/k}
    \leq q,
  \end{equation*}
  where $q$ is the unique root of equation \eqref{eq:q} in the interval $\big(m\tau/(m+1),1\big)$.
\end{mycorollary}

\begin{myremark}
  \label{remark:ConvergenceFactor_Lip}
  Toth and Kelley \cite{TothK2015} showed that Anderson acceleration with depth $m$
  converges R-linearly with an R-factor in $(\theta,1)$, 
  provided $\myvec{g}$ is Lipschitz continuously differentiable.
  Later, Chen and Kelley \cite{ChenK2019} relaxed this assumption and proved 
  convergence with R-factor $\theta^{1/(m+1)}$.
  Our work reveals that the R-factor depends not only on the contraction factor $\theta$,
  but also on the optimization gain $\eta_k$ and the condition number $\kappa(\myvec{f}'(\myvec{x}^*))$,
  leading to a more refined characterization of convergence than previous results 
  based only on the contraction factor $\theta$.
  In the special case $m=1$, 
  the R-factor can be explicitly expressed as $q = (\tau + \sqrt{\tau^2 + 4\zeta})/2$.
  For the problem considered in Subsection \ref{subsec:CDR} with $\nu = 1$,
  the Jacobian is Lipschitz continuous, 
  so the results in \cite{TothK2015,ChenK2019} are also applicable.
  Table \ref{tab:CDR_Lipschitz} shows that the convergence factor $q$ 
  obtained from Corollary \ref{cor:LocalConvLipCondAnderson(m)} is smaller than
  the one given in \cite{ChenK2019} for each depth considered, 
  while the result in \cite{TothK2015} only guarantees that the convergence factor lies in the interval $(\theta,1)$.
  This comparison provides a quantitative illustration of the refinement described above.
\end{myremark}

\subsection{Convergence rate for depth $m=1$}

In this subsection, we obtain a refined residual bound of Anderson acceleration for depth $m=1$ with the $\ell^2$ norm.
For the depth $m = 1$, the optimization problem \eqref{optim:alpha} becomes
$$
\alpha_k = \arg\min_{\alpha \in \RS} 
\left\|(1-\alpha)\myvec{f}(\myvec{x}_k) + \alpha \myvec{f}(\myvec{x}_{k-1})\right\|,
$$
and admits a closed-form solution
\begin{equation}
  \label{cons:alpha_k(m=1)}
  \alpha_k = \frac{\mytran{\myvec{f}(\myvec{x}_k)}(\myvec{f}(\myvec{x}_k) - \myvec{f}(\myvec{x}_{k-1}))}
  {\|\myvec{f}(\myvec{x}_k) - \myvec{f}(\myvec{x}_{k-1})\|^2}.
\end{equation}
Moreover, we have
$$
\myvec{x}_{k+1} = (1-\alpha_k)\myvec{g}(\myvec{x}_k) + \alpha_k \myvec{g}(\myvec{x}_{k-1}).
$$
The optimization gain $\eta_k$ given by \eqref{cons:eta_k} now becomes
\begin{equation}
  \label{cons:eta_k(m=1)}
  \eta_k = \frac{\|(1-\alpha_k)\myvec{f}(\myvec{x}_k) + \alpha_k \myvec{f}(\myvec{x}_{k-1})\|}{\|\myvec{f}(\myvec{x}_k)\|}.
\end{equation}
By combining \eqref{cons:alpha_k(m=1)} and \eqref{cons:eta_k(m=1)}, 
we arrive at 
\begin{equation}
  \label{eq:norm_f(xk+1)-f(xk)}
  |\alpha_k| \|\myvec{f}(\myvec{x}_k) - \myvec{f}(\myvec{x}_{k-1})\|_2
    = \sqrt{1 - \eta_k^2} \|\myvec{f}(\myvec{x}_k)\|.
\end{equation}
Details of the derivation can be found in \cite[p. 2848]{PollockR2021}.
As a consequence, we obtain two useful inequalities that provide an upper bound 
on the difference between consecutive iterations using the residual $\myvec{f}(\myvec{x}_k)$.

\begin{mylemma}
  \label{lem:norm_xk+1-xk}  
  Let $m = 1$ in Algorithm $\ref{alg:AA}$. For any $k \geq 1$, we have
  \begin{align}
    \|\myvec{x}_{k} - \myvec{x}_{k-1}\| 
      &\leq \frac{\sqrt{1 - \eta_k^2}}{|\alpha_k|(1 - \theta)} \|\myvec{f}(\myvec{x}_k)\|, \label{ineq:norm_xk-xk-1_m=1} \\
    \|\myvec{x}_{k+1} - \myvec{x}_k\|
      &\leq \left(1 + \frac{\theta}{1-\theta}\sqrt{1-\eta_k^2}\right)\|\myvec{f}(\myvec{x}_k)\|, \label{ineq:norm_xk+1-xk_m=1}
  \end{align}
  where $\alpha_k$ is given by \eqref{cons:alpha_k(m=1)},
  $\eta_k$ is given by \eqref{cons:eta_k(m=1)},
  and $\theta$ is the contraction factor.
\end{mylemma}

\begin{proof}
  By the contractivity of the fixed-point function $\myvec{g}$, we have
  \begin{align*}
    \|\myvec{f}(\myvec{x}_k) - \myvec{f}(\myvec{x}_{k-1})\| 
      &= \|(\myvec{g}(\myvec{x}_k) - \myvec{x}_k) - (\myvec{g}(\myvec{x}_{k-1}) - \myvec{x}_{k-1})\| \\
      &\geq \|\myvec{x}_k - \myvec{x}_{k-1}\| - \|\myvec{g}(\myvec{x}_k) - \myvec{g}(\myvec{x}_{k-1})\| \\
      &\geq (1 - \theta)\|\myvec{x}_k - \myvec{x}_{k-1}\|.
  \end{align*}
  Combining this with \eqref{eq:norm_f(xk+1)-f(xk)} yields the inequality \eqref{ineq:norm_xk-xk-1_m=1}.
  On the other hand, we observe that
  \begin{align*}
    \|\myvec{x}_{k+1} - \myvec{x}_k\| 
      &= \left\|[(1-\alpha_k)\myvec{g}(\myvec{x}_k) + \alpha_k \myvec{g}(\myvec{x}_{k-1})] - \myvec{x}_k\right\| \\
      &\leq \|\myvec{g}(\myvec{x}_k) - \myvec{x}_k\| + |\alpha_k| \|\myvec{g}(\myvec{x}_{k}) - \myvec{g}(\myvec{x}_{k-1})\| \\
      &\leq \|\myvec{f}(\myvec{x}_k)\| + \theta |\alpha_k| \|\myvec{x}_k - \myvec{x}_{k-1}\|,
  \end{align*}
  which, together with \eqref{ineq:norm_xk-xk-1_m=1}, gives \eqref{ineq:norm_xk+1-xk_m=1}.
\end{proof}

The following theorem provides a refined bound for the residual $\myvec{f}(\myvec{x}_{k+1})$
in terms of the residual $\myvec{f}(\myvec{x}_k)$ and the optimization gain $\eta_k$.

\begin{mytheorem}
  \label{th:LocalConvAnderson(m=1)}
  Let the assumptions of Theorem $\ref{th:LocalConvAnderson(m)}$ hold.
  If $m = 1$, then we have the following bound for the residual $\myvec{f}(\myvec{x}_{k+1})$:
  \begin{equation}
    \label{ineq:norm_f(xk+1)_m=1}
    \begin{aligned}
    \|\myvec{f}(\myvec{x}_{k+1})\| 
    &\leq \theta\left(1 + \frac{1+\theta}{1-\theta}\sqrt{1-\eta_k^2}\right)\|\myvec{f}(\myvec{x}_k)\| \\
    &\quad  + \frac{H_\nu}{1+\nu}\left[
        \left(1+\frac{\theta}{1-\theta}\sqrt{1-\eta_k^2}\right)^{1+\nu}
        + \left(\frac{\sqrt{1-\eta^2_k}}{1-\theta}\right)^{1+\nu}\frac{1}{|\alpha_k|^\nu}
    \right]\|\myvec{f}(\myvec{x}_k)\|^{1+\nu}.
    \end{aligned}
  \end{equation}
\end{mytheorem}

\begin{proof}
  Since
  \begin{align*}
    \myvec{f}(\myvec{x}_{k+1})
      &= \myvec{g}(\myvec{x}_{k+1}) - \myvec{x}_{k+1} \\
      &= [\myvec{g}(\myvec{x}_{k+1}) - \myvec{g}(\myvec{x}_k)]
        + \alpha_k [\myvec{g}(\myvec{x}_{k}) - \myvec{g}(\myvec{x}_{k-1})] \\
      &= \int_0^1 \myvec{g}'(\myvec{x}_k + t(\myvec{x}_{k+1} - \myvec{x}_k))(\myvec{x}_{k+1} - \myvec{x}_k)\dif{t} \\
      &\quad + \alpha_k \int_0^1 \myvec{g}'(\myvec{x}_{k-1} + t(\myvec{x}_{k} - \myvec{x}_{k-1}))(\myvec{x}_{k} - \myvec{x}_{k-1})\dif{t},
  \end{align*}
  we get 
  \begin{align*}
    \lefteqn{\myvec{f}(\myvec{x}_{k+1}) - \myvec{g}'(\myvec{x}_k)(\myvec{x}_{k+1} - \myvec{x}_k)
      - \alpha_k \myvec{g}'(\myvec{x}_{k-1})(\myvec{x}_{k} - \myvec{x}_{k-1}) } \\
    &= \int_0^1 \left[\myvec{g}'(\myvec{x}_k + t(\myvec{x}_{k+1} - \myvec{x}_k)) 
      - \myvec{g}'(\myvec{x}_k)\right](\myvec{x}_{k+1} - \myvec{x}_k)\dif{t} \\
    &\quad + \alpha_k \int_0^1 \left[\myvec{g}'(\myvec{x}_{k-1} + t(\myvec{x}_{k} - \myvec{x}_{k-1}))
      - \myvec{g}'(\myvec{x}_{k-1})\right](\myvec{x}_{k} - \myvec{x}_{k-1})\dif{t}.
  \end{align*}
  By the H\"older condition \eqref{cond:Holder}, we have
  \begin{align*}
    \lefteqn{\|\myvec{f}(\myvec{x}_{k+1}) - \myvec{g}'(\myvec{x}_k)(\myvec{x}_{k+1} - \myvec{x}_k)
      - \alpha_k \myvec{g}'(\myvec{x}_{k-1})(\myvec{x}_{k} - \myvec{x}_{k-1})\| } \\
    &\leq H_\nu \int_0^1 t^\nu \|\myvec{x}_{k+1} - \myvec{x}_k\|^{1+\nu}\dif{t} 
      + H_\nu |\alpha_k| \int_0^1 t^\nu \|\myvec{x}_{k} - \myvec{x}_{k-1}\|^{1+\nu}\dif{t} \\
    &= \frac{H_\nu}{1+\nu} \left(\|\myvec{x}_{k+1} - \myvec{x}_k\|^{1+\nu}
      + |\alpha_k| \|\myvec{x}_{k} - \myvec{x}_{k-1}\|^{1+\nu}\right).
  \end{align*}
  Then we can use \eqref{ineq:norm_xk-xk-1_m=1} and \eqref{ineq:norm_xk+1-xk_m=1} to obtain
  \begin{align}
    \lefteqn{\|\myvec{f}(\myvec{x}_{k+1}) - \myvec{g}'(\myvec{x}_k)(\myvec{x}_{k+1} - \myvec{x}_k)
      - \alpha_k \myvec{g}'(\myvec{x}_{k-1})(\myvec{x}_{k} - \myvec{x}_{k-1})\| } \nonumber \\
    &\leq \frac{H_\nu}{1+\nu} \left[\left(1 + \frac{\theta}{1-\theta}\sqrt{1-\eta_k^2}\right)^{1+\nu}
      + \left(\frac{\sqrt{1-\eta^2_k}}{1-\theta}\right)^{1+\nu} \frac{1}{|\alpha_k|^{\nu}}\right]
      \|\myvec{f}(\myvec{x}_k)\|^{1+\nu}. \label{ineq:norm_f(xk+1)-g'(xk)(xk+1-xk)}
  \end{align}
  Recall that the contractivity of $\myvec{g}$ implies that 
  $$
  \|\myvec{g}'(\myvec{x}_k)\| \leq \theta
  \quad \text{and}\quad 
  \|\myvec{g}'(\myvec{x}_{k-1})\| \leq \theta.
  $$
  This together with  \eqref{ineq:norm_xk-xk-1_m=1} and \eqref{ineq:norm_xk+1-xk_m=1} permits us to get
  \begin{align}
    \lefteqn{\|\myvec{g}'(\myvec{x}_k)(\myvec{x}_{k+1} - \myvec{x}_k)
    + \alpha_k \myvec{g}'(\myvec{x}_{k-1})(\myvec{x}_{k} - \myvec{x}_{k-1})\| } \nonumber \\
    &\leq \theta \|\myvec{x}_{k+1} - \myvec{x}_k\| + \theta |\alpha_k| \|\myvec{x}_{k} - \myvec{x}_{k-1}\| \nonumber \\
    &\leq \theta \left(1 + \frac{\theta}{1-\theta}\sqrt{1-\eta_k^2}\right)\|\myvec{f}(\myvec{x}_k)\| 
      + \frac{\theta}{1-\theta}\sqrt{1-\eta_k^2}\|\myvec{f}(\myvec{x}_k)\| \nonumber \\
    &= \theta\left(1 + \frac{1+\theta}{1-\theta}\sqrt{1-\eta_k^2}\right)\|\myvec{f}(\myvec{x}_k)\|. \label{ineq:norm_g'(xk)(xk+1-xk)+alpha_k g'(xk-1)(xk-xk-1)}
  \end{align}
  Combining \eqref{ineq:norm_f(xk+1)-g'(xk)(xk+1-xk)} 
  and \eqref{ineq:norm_g'(xk)(xk+1-xk)+alpha_k g'(xk-1)(xk-xk-1)} 
  yields \eqref{ineq:norm_f(xk+1)_m=1}.
  The proof is complete.
\end{proof}

\begin{myremark}
  \label{remark:LocalConvAnderson(m=1)}
  Theorem \ref{th:LocalConvAnderson(m=1)} shows that the optimization gain $\eta_k$ 
  affects the linear and higher-order terms in different ways.
  For fixed $\theta$, a smaller $\eta_k$ leads to a larger coefficient of the linear term,
  whereas the coefficient of the higher-order term depends on both $\eta_k$ and $|\alpha_k|$.
  Thus, the effect of $\eta_k$ on the higher-order term is not monotone in general.
\end{myremark}

\section{Applications}
\label{sec:Applications}
We build on the local convergence results established 
in the preceding section and consider two specific applications.
The first concerns the minimal positive solution of the NARE arising in neutron transport theory. 
The second concerns a two-dimensional steady semilinear convection-diffusion-reaction equation 
and demonstrates that our local convergence results apply in settings where
the existing Lipschitz-based theory does not.

\subsection{Nonsymmetric algebraic Riccati equation}
\label{subsec:NARE}

In this subsection, 
we explore Anderson acceleration to solve a special nonlinear equation 
which is obtained by the NARE arising from neutron transport theory.
We begin with some new definitions and notations.
For any real matrices $A = (a_{ij})_{m \times n}$ and $B = (b_{ij})_{m \times n}$,
we write $A \geq B$ (respectively, $A > B$) if $a_{ij} \geq b_{ij}$ 
(respectively, $a_{ij} > b_{ij}$) for all $i = 1,2,\ldots,m$ and $j = 1,2,\ldots,n$.
A real matrix $A = (a_{ij})_{m \times n}$ is called nonnegative 
if all its entries satisfy $a_{ij} \geq 0$, 
and positive if $a_{ij} > 0$ for all $i = 1,2,\ldots,m$ and $j = 1,2,\ldots,n$.
We denote these as $A \geq 0$ and $A > 0$, respectively.
We denote by $A \circ B = (a_{ij}\cdot b_{ij})_{m \times n}$ the Hadamard product of $A$ and $B$.
Moreover, for any real vectors $\myvec{a} = (a_1,a_2,\ldots,a_n)^{\top}$
and $\myvec{b} = (b_1,b_2,\ldots,b_n)^{\top}$, 
we write $\myvec{a} \geq \myvec{b}$ (respectively, $\myvec{a} > \myvec{b}$) 
if $a_i \geq b_i$ (respectively, $a_i > b_i$) for all $i = 1,2,\ldots,n$.
The vector of all zero components is denoted by $\myvec{0}$.
A vector $\myvec{v} \in \RS^n$ is called nonnegative if $\myvec{v} \geq \myvec{0}$,
and positive if $\myvec{v} > \myvec{0}$.

\subsubsection{Problem setting}

The form of NARE from neutron transport theory is as follows:
\begin{equation}
  \label{eq:NARE}
  XCX - XD - AX + B = 0,
\end{equation}
where $X \in \RS^{n \times n}$ is an unknown matrix, 
and $A, B, C, D \in \RS^{n \times n}$ are known matrices given by 
\begin{equation}
  \label{mat:ABCD}
  A = \Delta - \myvec{e}\myvec{p}^{\top}, \quad
  B = \myvec{e}\myvec{e}^{\top}, \quad
  C = \myvec{p}\myvec{p}^{\top}, \quad
  D = \hat{\Delta} - \myvec{p}\myvec{e}^{\top}.
\end{equation}
Here $\Delta = \diag(\delta_1,\delta_2,\ldots,\delta_n)$ with $\delta_i = 1/(c\omega_i(1+a)) > 0$,
$\hat{\Delta} = \diag(\hat{\delta}_1,\hat{\delta}_2,\ldots,\hat{\delta}_n)$ with $\hat{\delta}_i = 1/(c\omega_i(1-a)) > 0$,
and $\myvec{p} = (p_1,p_2,\ldots,p_n)^{\top}$ with $p_i = c_i/(2\omega_i) > 0$.
The matrices and vectors above depend on a pair of parameters $(a,c)$ with
\begin{equation}
  \label{cons:alpha_c}
  a \in [0,1) \quad \text{and} \quad c \in (0,1].
\end{equation}
Moreover, the sets $\{\omega_i\}_{i=1}^n$ and $\{c_i\}_{i=1}^n$ 
represent the nodes and weights, respectively, of the Gauss-Legendre quadrature 
on the interval $[0,1]$, satisfying
$$
0 < \omega_n < \cdots < \omega_2 < \omega_1 < 1 \ \text{and} \ \sum_{i=1}^n c_i = 1 \ \text{with} \ c_i > 0.
$$
Clearly, the sequences $\{\delta_i\}_{i=1}^n$ and $\{\hat{\delta}_i\}_{i=1}^n$ are
strictly monotonically increasing, and
$$
\begin{cases}
\delta_i = \hat{\delta}_i, & \mbox{when $a = 0$},\\
\delta_i \neq \hat{\delta}_i, & \mbox{when $a \neq 0$},
\end{cases}
\quad i = 1,2,\ldots,n.
$$
The NARE \eqref{eq:NARE} is obtained by a discretization of an integrodifferential equation
describing neutron transport during a collision process.
The solution of interest from a physical perspective is the minimal nonnegative solution,
as discussed in previous studies
\cite{Juang1995,JuangL1998,Guo2001}.

Lu \cite{Lu2005a} demonstrated that the solution to equation \eqref{eq:NARE} can be expressed by:
$$
X = T \circ (\myvec{u} \myvec{v}^{\top}) = (\myvec{u} \myvec{v}^{\top}) \circ T,
$$
where $T = (t_{ij})_{n \times n} = \left(\frac{1}{\delta_i + \hat{\delta}_j}\right)_{n\times n}$,
$\myvec{u}$ and $\myvec{v}$ are vectors satisfying
\begin{equation}
\label{eq:VecEq_uv}
\left\{
\begin{aligned}
\myvec{u} &= \myvec{u} \circ (P\myvec{v}) + \myvec{e},\\
\myvec{v} &= \myvec{v} \circ (\widetilde{P}\myvec{u}) + \myvec{e},
\end{aligned}
\right.
\end{equation}
with
\begin{equation}
\label{mat:P_Ptilde}
  P
    = (p_{ij})_{n\times n}
    = \left(\frac{p_j}{\delta_i + \hat{\delta}_j}\right)_{n\times n}, \quad 
  \widetilde{P}
    = (\widetilde{p}_{ij})_{n\times n}
    = \left(\frac{p_j}{\hat{\delta}_i + \delta_j}\right)_{n\times n}.
\end{equation}
We set $\myvec{x} = [\myvec{u}^{\top},\myvec{v}^{\top}]^{\top} \in \RS^{2n}$.
Then the objective of finding the minimal nonnegative solution of \eqref{eq:NARE} is equivalent to
finding solutions for the nonlinear system
\begin{equation}
\label{eq:f(u,v)=0}
\myvec{f}(\myvec{x}) = 
\myvec{f}(\myvec{u},\myvec{v}) \stackrel{\text{def}}{=}
\begin{bmatrix} 
  \myvec{u} - \myvec{u} \circ (P\myvec{v}) - \myvec{e} \\ 
  \myvec{v} - \myvec{v} \circ (\widetilde{P}\myvec{u}) - \myvec{e}
\end{bmatrix} = \myvec{0},
\end{equation}
or alternatively, to finding the fixed point of the fixed-point problem 
\begin{equation}
  \label{eq:g(u,v)=0}
  \myvec{x} = \myvec{g}(\myvec{x}) = 
  \myvec{g}(\myvec{u},\myvec{v}) \stackrel{\text{def}}{=}
  \begin{bmatrix} 
    \myvec{u} \circ (P\myvec{v}) + \myvec{e} \\ 
    \myvec{v} \circ (\widetilde{P}\myvec{u}) + \myvec{e}
  \end{bmatrix}.
\end{equation}
The advantage in representing \eqref{eq:VecEq_uv} as the nonlinear system \eqref{eq:f(u,v)=0} is that
we now can use the Newton-type methods to solve it.
Since Lu's Newton-based algorithm was introduced in \cite{Lu2005b},
extensive research has been conducted on Newton-type methods, 
focusing either on improving their effectiveness 
or on accelerating convergence via higher-order techniques
(see, e.g. \cite{BiniIP2008,LinBW2008,LinB2008,HuangKH2010,LingX2017,LingLL2022}).
Moreover, there has been significant interest in developing more effective fixed-point iterative algorithms 
for solving the fixed-point problem \eqref{eq:g(u,v)=0}, such as the ones described in
\cite{BaoLW2006,BaiGL2008,Lin2008,GuoL2010,LinBW2011,HuangM2014,HuangM2020}.

Let $P \in \RS^{n \times n}$
be partitioned column-wise as $P = [\myvec{p}_1,\myvec{p}_2,\ldots,\myvec{p}_n]$, and similarly for 
$\tilde{P} = [\tilde{\myvec{p}}_1,\tilde{\myvec{p}}_2,\ldots,\tilde{\myvec{p}}_n] \in \RS^{n \times n}$.
We note that the nonlinear operator $\myvec{f}$ defined by \eqref{eq:f(u,v)=0}
is continuously Fr\'echet differentiable.
The Jacobian at point $(\myvec{u},\myvec{v})$ is given by (see \cite{Lu2005b})
\begin{equation*}
  \label{mat:Jacobian}
  \myvec{f}'(\myvec{u},\myvec{v}) = I_{2n} - G(\myvec{u},\myvec{v}),
\end{equation*}
where 
\begin{equation*}
  \label{mat:G}
  G(\myvec{u},\myvec{v}) = 
  \begin{bmatrix}
    G_{11}(\myvec{v}) & G_{12}(\myvec{u}) \\
    G_{21}(\myvec{v}) & G_{22}(\myvec{u})
  \end{bmatrix}
\end{equation*}
with 
\begin{align*}
  G_{11}(\myvec{v}) &= \diag(P\myvec{v}), \quad
  G_{12}(\myvec{u}) = [\myvec{u}\circ \myvec{p}_1,\myvec{u}\circ \myvec{p}_2,\ldots,\myvec{u}\circ \myvec{p}_n], \\
  G_{22}(\myvec{u}) &= \diag(\tilde{P}\myvec{u}), \quad 
  G_{21}(\myvec{v}) = [\myvec{v} \circ \tilde{\myvec{p}}_1,\myvec{v}\circ\tilde{\myvec{p}}_2,\ldots,\myvec{v}\circ\tilde{\myvec{p}}_n].
\end{align*}
It follows that the Jacobian of the fixed-point function $\myvec{g}$ is $G(\myvec{u},\myvec{v})$.
Furthermore, the Jacobian $\myvec{f}'$ is Lipschitz continuous with respect to the $\ell^\infty$ norm.
In fact, for any $\myvec{x}, \myvec{y} \in \RS^{2n}$, we have
$$
\|\myvec{f}'(\myvec{x}) - \myvec{f}'(\myvec{y})\|_{\infty}
  = \|G(\myvec{u},\myvec{v}) - G(\myvec{s},\myvec{t})\|_{\infty}
  \leq 2\max_{1\leq i \leq n} \left\{\sum_{j=1}^{n}p_{ij}, \sum_{j=1}^{n}\tilde{p}_{ij}\right\} \|\myvec{x} - \myvec{y}\|_{\infty}.
$$
According to Lemma 3 in \cite{Lu2005a}, it holds that
\begin{equation}
  \label{ineq:sum_pij}
  \sum_{j=1}^{n}p_{ij} < \frac{c(1-a)}{2} \quad \text{and} \quad
  \sum_{j=1}^{n}\tilde{p}_{ij} < \frac{c(1+a)}{2}.
\end{equation}
Then we can conclude that
$$
\|\myvec{f}'(\myvec{x}) - \myvec{f}'(\myvec{y})\|_{\infty} 
  \leq c(1+a) \|\myvec{x} - \myvec{y}\|_{\infty}.
$$
This means that the Jacobian $\myvec{f}'$ is Lipschitz continuous 
with Lipschitz constant $L = c(1+a)$.

Let $\myvec{x}^* \in \RS^{2n}$ be the minimal nonnegative solution of the nonlinear system \eqref{eq:f(u,v)=0}.
It was shown in \cite[Theorem 4.1]{BaiGL2008} that 
$$
\myvec{e} < \frac{2}{1 + \sqrt{1-4\phi}}\myvec{e} \leq \myvec{x}^*
 \leq \frac{2}{1 + \sqrt{1-4\Phi}}\myvec{e} \leq 2\myvec{e}, 
 \quad 0 < \phi < \Phi \leq \frac{1}{4},
$$
which improves upon the previous results obtained in \cite{Lu2005a}.
We observe that the fixed-point function $\myvec{g}$ defined by \eqref{eq:g(u,v)=0} is a contractive operator
with respect to the $\ell^\infty$ norm under some condition depending on the parameters $c$ and $a$. 
Indeed, for any $\myvec{x} = \mytran{[\mytran{\myvec{u}},\mytran{\myvec{v}}]}$,
$\myvec{y} = \mytran{[\mytran{\myvec{s}},\mytran{\myvec{t}}]} \in \RS^{2n}$, we have
\begin{align*}
  \|\myvec{g}(\myvec{x}) - \myvec{g}(\myvec{y})\|_{\infty}
  &= \left\|\begin{bmatrix}
    \myvec{u}\circ(P\myvec{v} - P\myvec{t}) + (\myvec{u} - \myvec{s})\circ(P\myvec{t}) \\
    \myvec{v}\circ(\tilde{P}\myvec{u} - \tilde{P}\myvec{s}) + (\myvec{v} - \myvec{t})\circ(\tilde{P}\myvec{s})
  \end{bmatrix}\right\|_{\infty} \\
  &\leq \left\|\begin{bmatrix}
    \myvec{u} \\
    \myvec{v}
  \end{bmatrix}\right\|_{\infty}
  \left\|\begin{bmatrix}
    \tilde{P} & \\
    & P
  \end{bmatrix}
  \begin{bmatrix}
    \myvec{u} - \myvec{s} \\
    \myvec{v} - \myvec{t}
  \end{bmatrix}\right\|_{\infty} 
  + \left\|\begin{bmatrix}
    \myvec{u} - \myvec{s} \\
    \myvec{v} - \myvec{t}
  \end{bmatrix}\right\|_{\infty}
  \left\|\begin{bmatrix}
    \tilde{P} & \\
    & P
  \end{bmatrix}
  \begin{bmatrix}
    \myvec{s} \\
    \myvec{t}
  \end{bmatrix}\right\|_{\infty}.
\end{align*}
We conclude from \eqref{ineq:sum_pij} that
\begin{equation*}
  \label{ineq:norm_g(x)-g(y)}
\|\myvec{g}(\myvec{x}) - \myvec{g}(\myvec{y})\|_{\infty}
  \leq \frac{c(1+a)}{2}(\|\myvec{x}\|_{\infty} + \|\myvec{y}\|_{\infty})\|\myvec{x} - \myvec{y}\|_{\infty}.
\end{equation*}
If 
$\myvec{x},\myvec{y} \in \ball_+(\myvec{0},r) 
  := \{\myvec{z} \in \RS^{2n} \mid \|\myvec{z}\|_\infty < r, \myvec{z} \geq \myvec{0}\}$
with $r = 2/\big(1+\sqrt{1-4\Phi}\big)$, then it follows that
$$
\|\myvec{g}(\myvec{x}) - \myvec{g}(\myvec{y})\|_{\infty}
  < \frac{c(1+a)}{2} \cdot \frac{4}{1 + \sqrt{1-4\Phi}} \|\myvec{x} - \myvec{y}\|_{\infty}
  = \frac{2c(1+a)}{1+\sqrt{1-4\Phi}}\|\myvec{x} - \myvec{y}\|_{\infty}.
$$
This means that the fixed-point function $\myvec{g}$ is contractive on the ball $\ball_+(\myvec{0},r)$
with contraction factor $\theta = 2c(1+a)/(1+\sqrt{1-4\Phi})$, provided that
$$
2c(1+a) < 1 + \sqrt{1-4\Phi}.
$$ 

Since the minimal positive solution of the nonlinear system \eqref{eq:f(u,v)=0},
or equivalently, the fixed-point problem \eqref{eq:g(u,v)=0}, is generally unavailable, 
Corollary \ref{cor:LocalConvLipCondAnderson(m)} cannot be directly applied.
Nonetheless, we can still verify the convergence results through numerical experiments
comparing Anderson acceleration with several established fixed-point iterative methods.
These experiments, presented in Subsection \ref{subsubsec:NumericalExperiments},
serve as an indirect validation of our theoretical results 
and further illustrate the robustness and efficiency of Anderson acceleration in various problem settings,
particularly in nearly singular and large-scale problems.

\begin{myremark}
  Although the Jacobian of the nonlinear operator $\myvec{f}$ defined by \eqref{eq:f(u,v)=0}
  is Lipschitz continuous,
  the sufficient condition $\theta = 2c(1+a)/(1+\sqrt{1-4\Phi}) < 1$
  obtained above for the contractivity of the fixed-point function $\myvec{g}$ defined by \eqref{eq:g(u,v)=0}
  becomes increasingly restrictive as the problem approaches the singular case.
  To see this, write $\beta := c(1+a)$.
  Since $0 < \Phi \leq 1/4$, we have $1 \leq 1 + \sqrt{1-4\Phi} < 2$,
  which gives $\beta < \theta \leq 2\beta$.
  Then, for $\beta \geq 1/2$, the condition $2\beta < 1 + \sqrt{1-4\Phi}$, 
  under which $\theta < 1$, is equivalent to $\Phi < \beta(1-\beta)$.
  In the test cases of Subsection \ref{subsubsec:NumericalExperiments}
  with $c = 1-a$, the condition becomes $\Phi < a^2(1-a^2)$.
  This means that the admissible range for $\Phi$ shrinks as $a$ approaches 0.
  Consequently, the contractivity of the fixed-point function $\myvec{g}$ 
  in the ball $\ball_+(\myvec{0},r)$  cannot be verified by the above condition in the nearly singular case.
  However, we can still observe the convergence of Anderson acceleration in numerical experiments
  even when this sufficient condition fails to hold.
  It should be emphasized that 
  the failure of this sufficient condition does not imply that $\myvec{g}$ is noncontractive.
  It only means that contractivity cannot be guaranteed by the present estimate.
  The numerical results in Subsection \ref{subsubsec:NumericalExperiments} 
  therefore illustrate the practical effectiveness of Anderson acceleration
  beyond the cases in which the sufficient local convergence conditions hold,
  rather than providing a direct verification of Corollary \ref{cor:LocalConvLipCondAnderson(m)}.
  The newly added Subsection \ref{subsec:CDR} considers a problem
  for which the assumptions of Theorem \ref{th:LocalConvAnderson(m)} can be verified explicitly,
  allowing the applicability of the theorem to be demonstrated directly.
\end{myremark}

\subsubsection{Implementation}

One of the main challenges in implementing the Anderson acceleration described
in Algorithm \ref{alg:AA} is determining the coefficients $\{\alpha_j^{(k)}\}_{j=0}^{m_k}$ 
by solving the constrained least squares problem \eqref{optim:alpha}.
Following the approach in \cite{FangS2009,WalkerN2011},
this problem can be equivalently reformulated as an unconstrained least squares problem,
which can be efficiently solved using the QR factorization.
To this end, define $\Delta{\myvec{f}_i} = \myvec{f}_{i+1} - \myvec{f}_i$ for any $i \geq 1$
and let 
$\mathcal{F}_k = (\Delta{\myvec{f}_{k-m_k}}, \ldots, \Delta{\myvec{f}_{k-1}}) \in \RS^{n\times m_k}$.
If we set $\alpha_0 = \gamma_0$,
$$
\alpha_i = \gamma_i - \gamma_{i-1}, \quad i = 1,2,\ldots,m_k-1,
$$
and $\alpha_{m_k} = 1 - \gamma_{m_k-1}$, 
then the constrained least squares problem \eqref{optim:alpha} is equivalent to the following unconstrained least squares problem:
\begin{equation}
  \label{optim:alpha_unconstrained}
  \min_{\bm{\gamma} = \mytran{(\gamma_0,\ldots,\gamma_{m_k-1})}} 
\|\myvec{f}_{k} - \mathcal{F}_k\bm{\gamma}\|_2.
\end{equation}
We denote the least squares solution by 
$\bm{\gamma}^{(k)} = \mytran{(\gamma_0^{(k)},\ldots,\gamma_{m_k-1}^{(k)})}$.
In addition, we set $\mathcal{G}_k = (\Delta{\myvec{g}_{k-m_k}},\ldots,\Delta{\myvec{g}_{k-1}})$
with $\Delta{\myvec{g}_i} = \myvec{g}(\myvec{x}_{i+1}) - \myvec{g}(\myvec{x}_i)$.
Then the updated iteration $\myvec{x}_{k+1}$ in Algorithm \ref{alg:AA} can be written as
$$
\myvec{x}_{k+1} = \myvec{g}(\myvec{x}_k) - \mathcal{G}_k \bm{\gamma}^{(k)}.
$$
For $\mathcal{F}_k \in \RS^{n \times m_k}$, since $m_k \ll n$,
we can compute the QR factorization of $\mathcal{F}_k$ via thin QR decomposition.
Let $\mathcal{F}_k = Q_kR_k$ be the thin QR factorization,
where $Q_k \in \RS^{n \times m_k}$ has orthogonal columns
and $R_k \in \RS^{m_k \times m_k}$ is an upper triangular matrix.
Then the unconstrained least squares problem \eqref{optim:alpha_unconstrained}
reduces to 
$$
\min_{\bm{\gamma} = \mytran{(\gamma_0,\ldots,\gamma_{m_k-1})}} 
\|\mytran{Q}_k\myvec{f}_{k} - R_k\bm{\gamma}\|_2.
$$
Therefore, the least squares solution $\bm{\gamma}^{(k)} \in \RS^{m_k}$ can be obtained by
solving the upper triangular system $R_k\bm{\gamma} = \mytran{Q}_k\myvec{f}_{k}$.

We point out that each matrix $\mathcal{F}_k \in \RS^{n \times m_k}$ 
in \eqref{optim:alpha_unconstrained} is obtained from $\mathcal{F}_{k-1}$ 
by appending the new column on the right and, 
if the resulting number of columns exceeds the depth $m$,
also deleting the first column on the left.
This means that the QR factorization of $\mathcal{F}_k$ can be efficiently updated 
from that of $\mathcal{F}_{k-1}$ using standard QR updating techniques \cite{GolubV1996},
with a computational cost of only $\mathcal{O}(m_kn)$.
It is worth noting that $\mathcal{F}_k$ does not need to be explicitly
constructed in the algorithm implementation.

To delete the first column on the left, 
we utilize MATLAB's \texttt{qrdelete} function for efficient QR factorization updating.
For appending a column on the right, we assume the QR factorization of $\mathcal{F}_{k-1} \in \RS^{n \times (m_k-1)}$
is given by $\mathcal{F}_{k-1} = Q_{k-1}R_{k-1}$,
with $Q_{k-1} \in \RS^{n \times (m_k-1)}$ and $R_{k-1} \in \RS^{(m_k-1)\times(m_k-1)}$.
Then the updated matrix $\mathcal{F}_k$ has the form
$$
Q_kR_k = \mathcal{F}_{k} 
  = [\mathcal{F}_{k-1}, \Delta{\myvec{f}_{k-1}}]
  = [Q_{k-1}R_{k-1}, \Delta{\myvec{f}_{k-1}}].
$$
If we set $Q_k = [Q_{k-1},\myvec{q}_{m_k}]$ 
and $R_k = \begin{bmatrix}
  R_{k-1} & \myvec{r}_{m_k} \\ 0 & r_{m_km_k}
\end{bmatrix}$, then the new column $\Delta{\myvec{f}_{k-1}}$ admits the decomposition
$$
\Delta{\myvec{f}_{k-1}} = Q_{k-1}\myvec{r}_{m_k} + r_{m_km_k}\myvec{q}_{m_k}
  = [Q_{k-1},\myvec{q}_{m_k}]\begin{bmatrix}
    \myvec{r}_{m_k} \\ r_{m_km_k}
  \end{bmatrix}.
$$
This means that updating the QR factorization from $\mathcal{F}_{k-1}$ to $\mathcal{F}_k$
requires only computing the QR factorization of the new column $\Delta{\myvec{f}_{k-1}}$
against the existing orthogonal basis $Q_{k-1}$, 
which can be efficiently accomplished using a single modified Gram-Schmidt sweep \cite{GolubV1996}.

The implementation of Anderson acceleration for solving the fixed-point problem \eqref{eq:g(u,v)=0}
is summarized in Algorithm \ref{alg:AA_Implementation}.
It is worth pointing out that the differentiability assumption 
is required for the convergence analysis, but not for the implementation of the algorithm.
In particular, neither Algorithm \ref{alg:AA} nor Algorithm \ref{alg:AA_Implementation}
requires the computation or approximation of the Jacobian of the nonlinear operator $\myvec{f}$.
At each iteration, the implementation of Anderson acceleration requires
only evaluations of the fixed-point function $\myvec{g}$ 
and the residual $\myvec{f}_k = \myvec{g}_k - \myvec{x}_k$,
together with residual differences and the solution of a small least-squares problem.
The explicit expression for the Jacobian $\myvec{f}'$ given in Subsection 4.1.1
is used only to verify the Lipschitz continuity assumption required by Corollary 3.2.
Therefore, no additional cost associated with evaluating 
or approximating the Jacobian is incurred in the numerical implementation.

\begin{algorithm}[t]
\caption{Anderson acceleration for solving the fixed-point problem \eqref{eq:g(u,v)=0}}
\label{alg:AA_Implementation}
\textit{Initialization}.
Given parameters $a \in [0,1)$ and $c \in (0, 1]$. 
Choose initial point 
$\myvec{x}_0 = [\myvec{u}^{\top}_0,\myvec{v}^{\top}_0]^{\top} = \myvec{0} \in \RS^{2n}$ 
and the depth $m \geq 1$.
\begin{itemize}[leftmargin=1em,itemindent=3.5em,parsep=0em,itemsep=0em,topsep=0em]
  \item[Step 1.]
  Form the matrices $P$ and $\tilde{P}$ as defined in \eqref{mat:P_Ptilde}.
  \item[Step 2.]
  Compute $[\myvec{u}^{\top}_1,\myvec{v}^{\top}_1]^{\top} = \myvec{g}_0 \triangleq \myvec{g}(\myvec{u}_0,\myvec{v}_0)$,
  where $\myvec{g}$ is defined by \eqref{eq:g(u,v)=0}.
  \item[Step 3.]
  Compute the initial residual 
  $\myvec{f}_0 = \myvec{g}_0 - \myvec{x}_0$.
\end{itemize}
\textit{Iterative process}.
For $k = 1,2,\ldots$ until convergence, do:
\begin{itemize}[leftmargin=1em,itemindent=3.5em,parsep=0em,itemsep=0em,topsep=0em]
  \item[Step 1.]
  Set $m_k = \min\{m,k\}$.
  \item[Step 2.]
  Compute $\myvec{g}_k = \myvec{g}(\myvec{u}_k,\myvec{v}_k)$
  and the residual $\myvec{f}_k = \myvec{g}_k - \myvec{x}_k$.
  \item[Step 3.]
  Set $\Delta{\myvec{g}_{k-1}} = \myvec{g}_k - \myvec{g}_{k-1}$
  and $\Delta{\myvec{f}_{k-1}} = \myvec{f}_k - \myvec{f}_{k-1}$.
  \item[Step 4.]
  Update the matrix $\mathcal{G}_k$:
  \begin{itemize}
    \item If $k = 1$, set $\mathcal{G}_k = \Delta{\myvec{g}_{k-1}}$.
    \item If $k \leq m$, set $\mathcal{G}_k = [\mathcal{G}_{k-1},\Delta{\myvec{g}_{k-1}}]$.
    \item If $k > m$, set $\mathcal{G}_k = [\mathcal{G}_{k-1}(:,2:m_k),\Delta{\myvec{g}_{k-1}}]$.
  \end{itemize}
  \item[Step 5.]
  Update the QR factorization $\mathcal{F}_k = Q_kR_k$:
  \begin{itemize}
    \item If $k = 1$, set $R_k = \|\Delta{\myvec{f}_{k-1}}\|_2$ 
    and $Q_k = \Delta{\myvec{f}_{k-1}}/R_k$.
    \item If $k > 1$ and $k > m$, 
    update the QR factorization by deleting the first column of $\mathcal{F}_{k-1} = Q_{k-1}R_{k-1}$
    using the MATLAB's \texttt{qrdelete} function: 
    $[Q_k,R_k] = \texttt{qrdelete}(Q_{k-1},R_{k-1},1)$.
    \item If $k > 1$, append the new column $\Delta{\myvec{f}_{k-1}}$ to $\mathcal{F}_{k-1}$
    by using a single modified Gram-Schmidt sweep.
  \end{itemize}
  \item[Step 6.]
  Solve the upper triangular system $R_k\bm{\gamma} = \mytran{Q}_k\myvec{f}_{k}$ 
  to obtain the least squares solution $\bm{\gamma}^{(k)}$.
  \item[Step 7.]
  Update the next iteration 
  $\myvec{x}_{k+1} = [\myvec{u}_{k+1}^{\top},\myvec{v}_{k+1}^{\top}]^{\top} 
    = \myvec{g}_k - \mathcal{G}_k\bm{\gamma}^{(k)}$.
\end{itemize}
\end{algorithm}

\subsubsection{Numerical experiments}
\label{subsubsec:NumericalExperiments}

This subsection presents numerical experiments that demonstrate 
the effectiveness of Anderson acceleration in reducing both the number of iterations 
and the overall execution time of the fixed-point iterative methods 
for computing the minimal nonnegative solution of the nonlinear system \eqref{eq:f(u,v)=0}, 
or equivalently, the fixed-point problem \eqref{eq:g(u,v)=0},
which stems from a nonsymmetric algebraic Riccati equation in neutron transport theory.

All experiments were conducted in MATLAB R2024b on a MacBook Pro 
equipped with an Apple M3 8-core CPU and 24 GB of RAM.
The algorithms tested are listed below, 
with abbreviations used in the corresponding tables and figures.
\begin{itemize}
  \item AA($m$) (for Anderson acceleration with depth $m$) 
  is our implementation of Algorithm \ref{alg:AA_Implementation}.
  \item FP is the algorithm from \cite{Lu2005a} by using the fixed-point iteration.
  \item MFP is the modified fixed-point iteration algorithm from \cite{BaoLW2006}.
  \item NBJ is the nonlinear block Jacobi iteration algorithm proposed in \cite{BaiGL2008}.
  \item NBGS is the nonlinear block Gauss-Seidel iteration algorithm from \cite{BaiGL2008}.
\end{itemize}
The initial point $\myvec{x}_0$ is set to $\myvec{0}$ for all the algorithms considered above.
Following Example 5.2 of \cite{GuoL2000b}, 
the constants $c_i$ and $\omega_i$ are obtained using a numerical quadrature on $[0,1]$, 
by dividing the interval into $n/4$ equal subintervals 
and applying four nodes Gauss-Legendre quadrature on each subinterval.
The stopping criterion for all algorithms considered above is given by
$$
\text{RES} := \max\left\{\frac{\|\myvec{u}_{k+1} - \myvec{u}_k\|_{\infty}}{\|\myvec{u}_{k+1}\|_{\infty}},
  \frac{\|\myvec{v}_{k+1} - \myvec{v}_k\|_{\infty}}{\|\myvec{v}_{k+1}\|_{\infty}}\right\}
  \leq n\cdot\texttt{eps},
$$
where $n$ is the matrix size from \eqref{eq:NARE} 
and $\texttt{eps} = 2^{-52} \approx 2.2204 \times 10^{-16}$ 
denotes the double-precision machine epsilon.
The CPU time for each iterative method is the wall-clock time 
for a complete execution of the MATLAB implementation,
measured using MATLAB's \texttt{tic/toc} commands over the entire run.
Each experiment is repeated 10 times, 
and the average runtimes are reported in the tables and figures below.
Since the iterative methods considered here have different computational costs per iteration, 
both the total CPU time and the number of iterations
are taken into account in evaluating their computational performance.
The number of iterations is denoted by IT.

We note that the fixed-point iterative methods proposed in \cite{Lu2005a,BaoLW2006,BaiGL2008}
exhibit linear convergence when the pair of parameters $(a,c) \neq (0,1)$.
A detailed theoretical analysis comparing the convergence rates of these methods
is provided in \cite{GuoL2010}.
The singular case $(a,c) = (0,1)$, known to be particularly challenging,
has been effectively solved by various Newton-type methods 
\cite{LinBW2008,LinB2008,HuangKH2010,LingLL2022,LiangL2026}
or alternative approaches \cite{BiniIP2008,MehrmannX2008}.
Such techniques are generally recommended when $(a,c)$ is close to $(0,1)$,
while fixed-point iterative methods are more suitable for the regular cases.
Our numerical results below demonstrate that AA performs efficiently in both cases.

We begin with the regular case $(a,c) = (0.5,0.5)$.
Figure \ref{fig:IterHistory_alphaC0.5} shows the iteration histories
for the problem sizes $n = 1024, 2048, 4096, 8192$.
It is observed that AA with various depths requires fewer iterations than FP, MFP and NBJ, 
although it still requires more iterations than NBGS.
As $(a,c)$ approaches the singular case $(0,1)$, 
AA becomes increasingly efficient and eventually outperforms NBGS in terms of iteration count,
as illustrated in Figures \ref{fig:IterHistory_alphaC10-1} and \ref{fig:IterHistory_alphaC10-3}.

\begin{figure}
  \centering
  \subfigure{\includegraphics[width=0.42\textwidth]{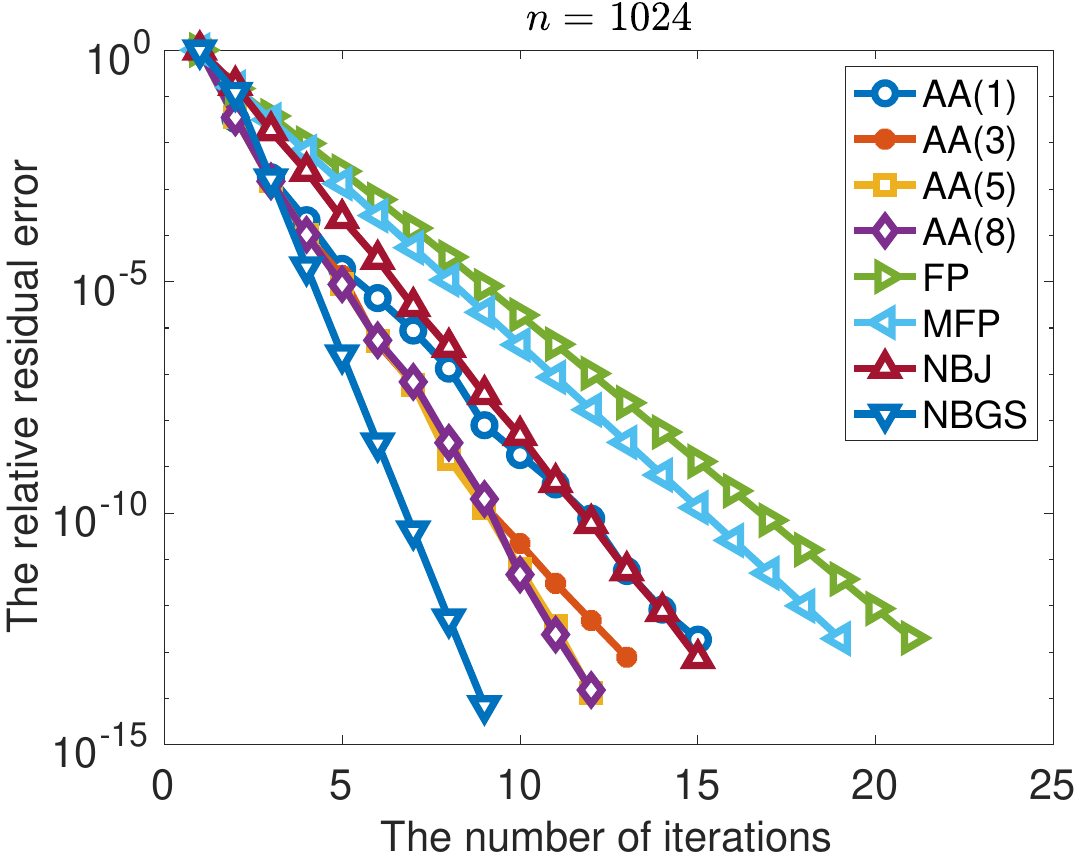}}\quad
  \subfigure{\includegraphics[width=0.42\textwidth]{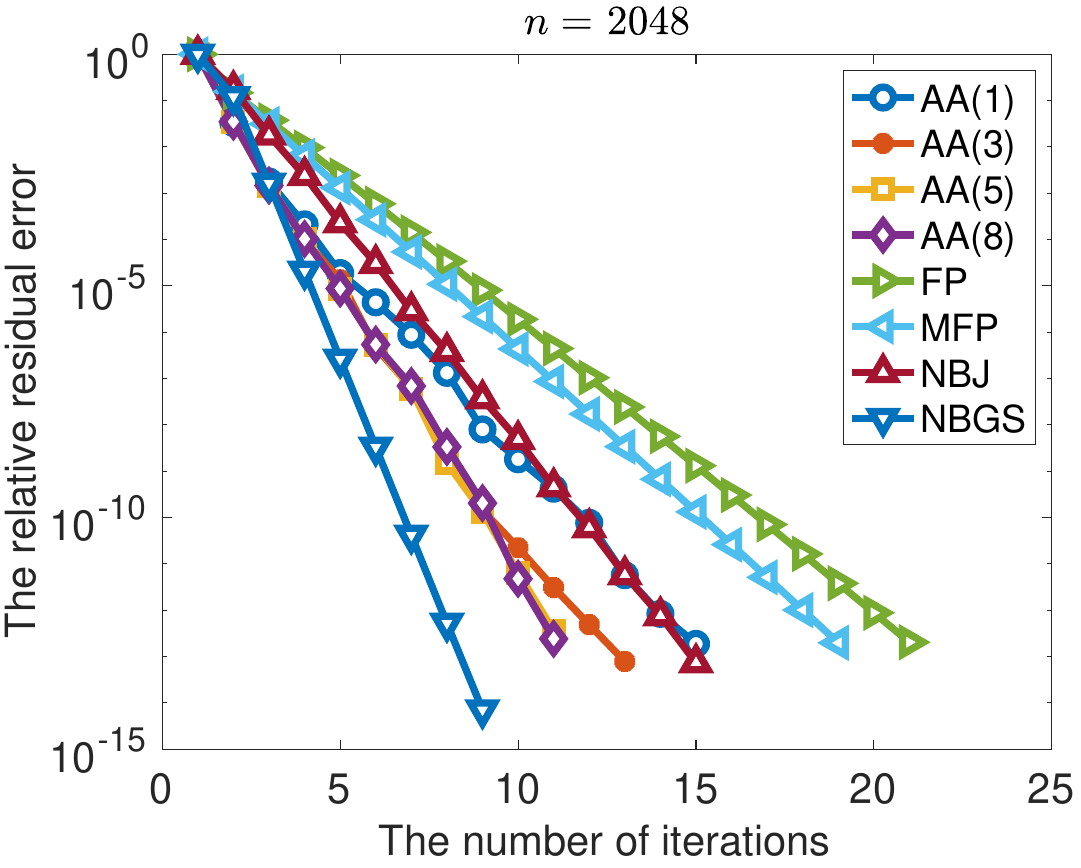}} \\
  \subfigure{\includegraphics[width=0.42\textwidth]{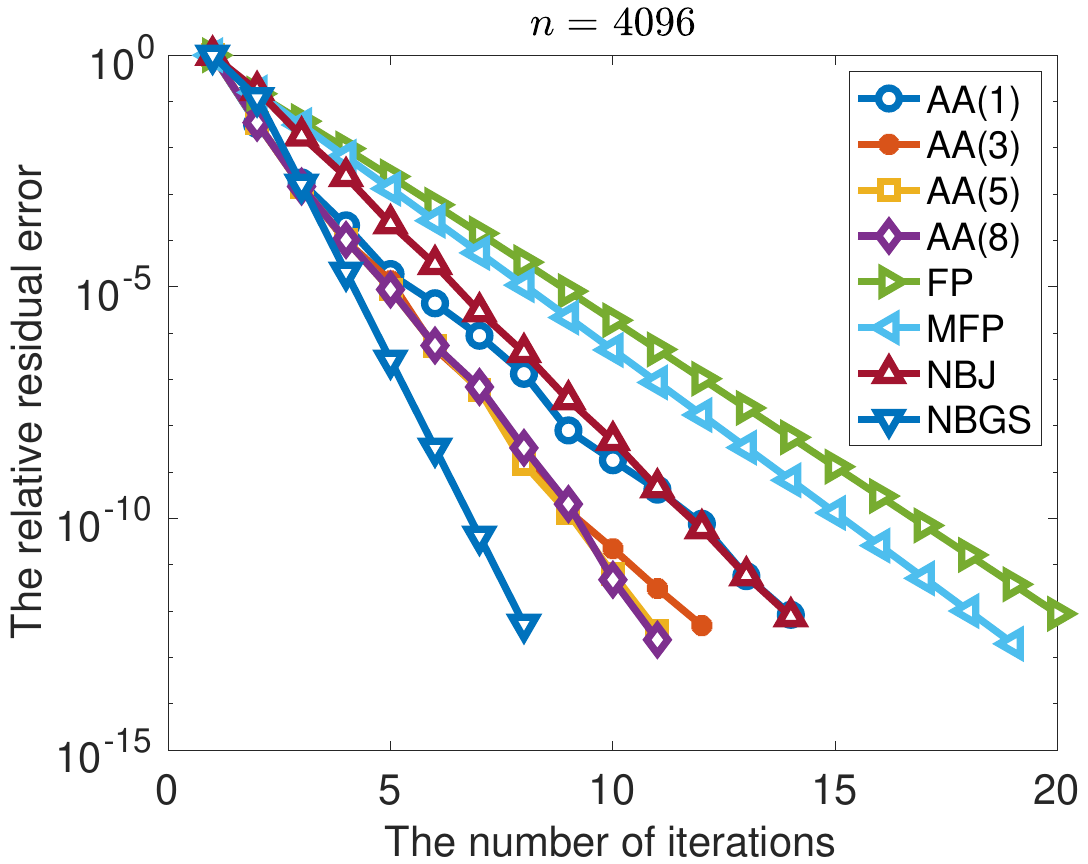}}\quad
  \subfigure{\includegraphics[width=0.42\textwidth]{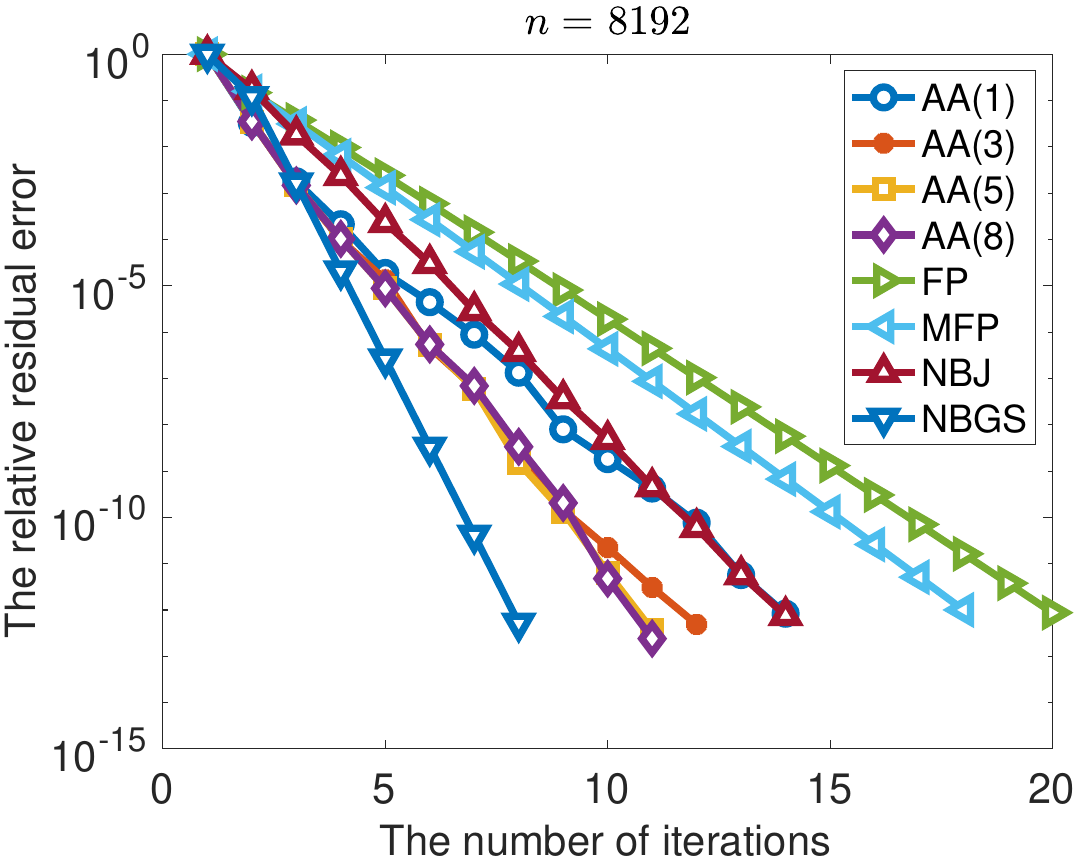}} \\
  \caption{Iteration histories for $(a, c) = (0.5,0.5)$ with various problem sizes $n = 1024, 2048, 4096, 8192$.}
  \label{fig:IterHistory_alphaC0.5}
\end{figure}

\begin{figure}
  \centering
  \subfigure{\includegraphics[width=0.42\textwidth]{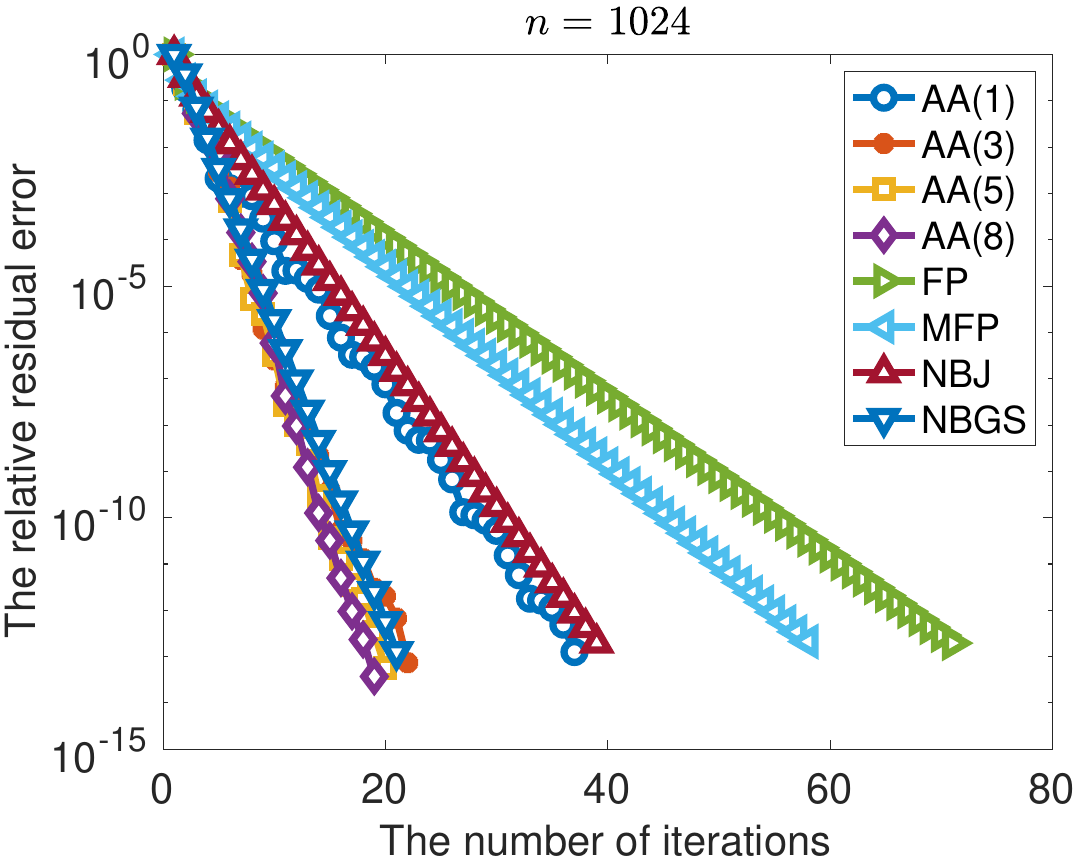}}\quad
  \subfigure{\includegraphics[width=0.42\textwidth]{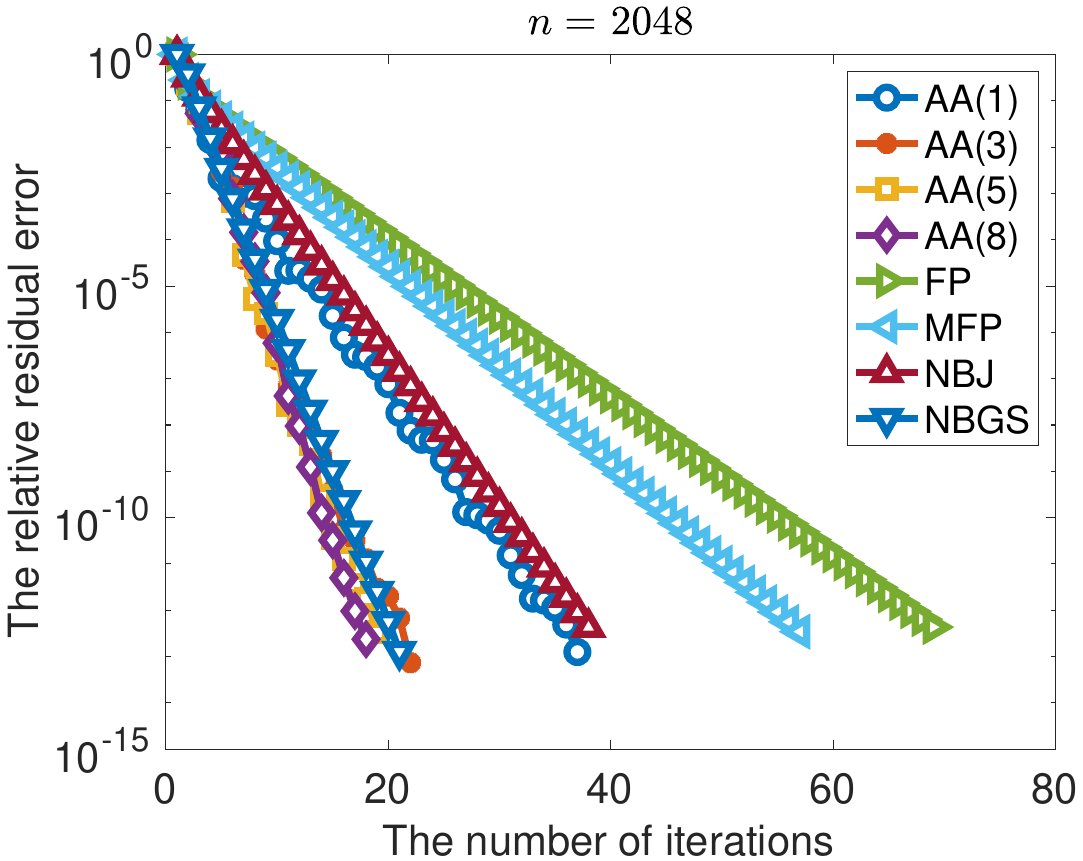}} \\
  \subfigure{\includegraphics[width=0.42\textwidth]{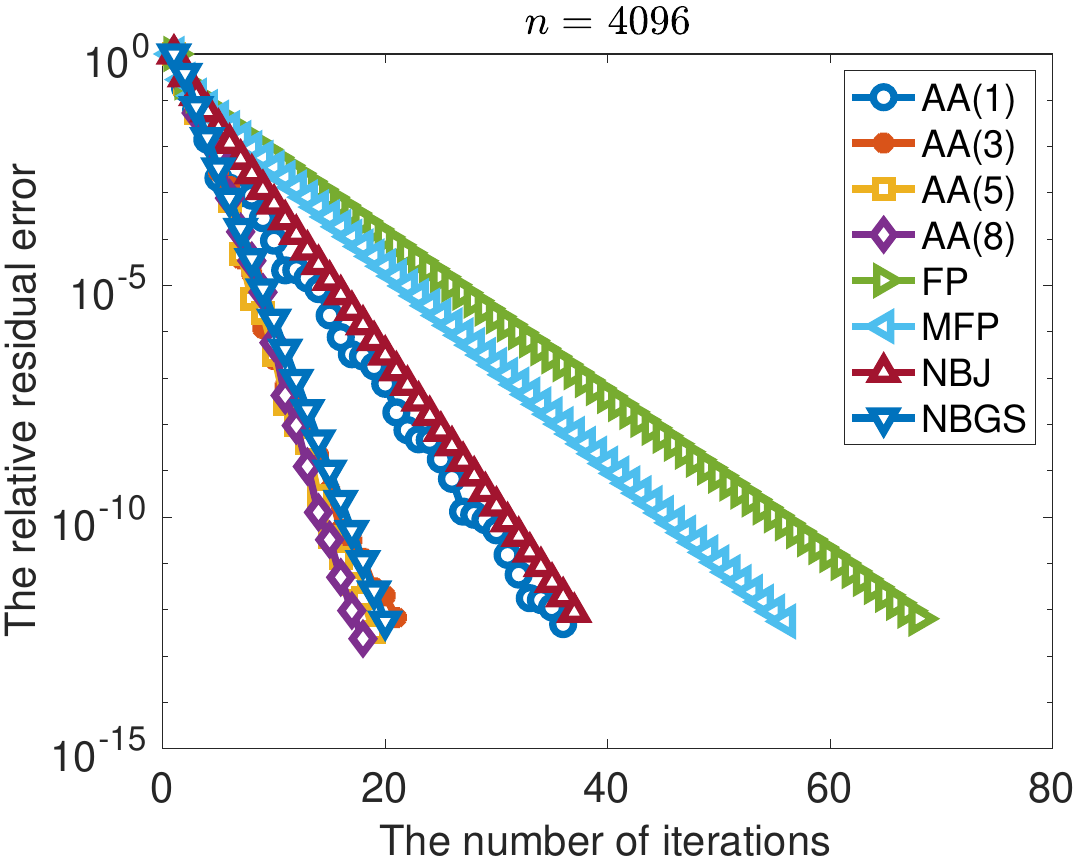}}\quad
  \subfigure{\includegraphics[width=0.42\textwidth]{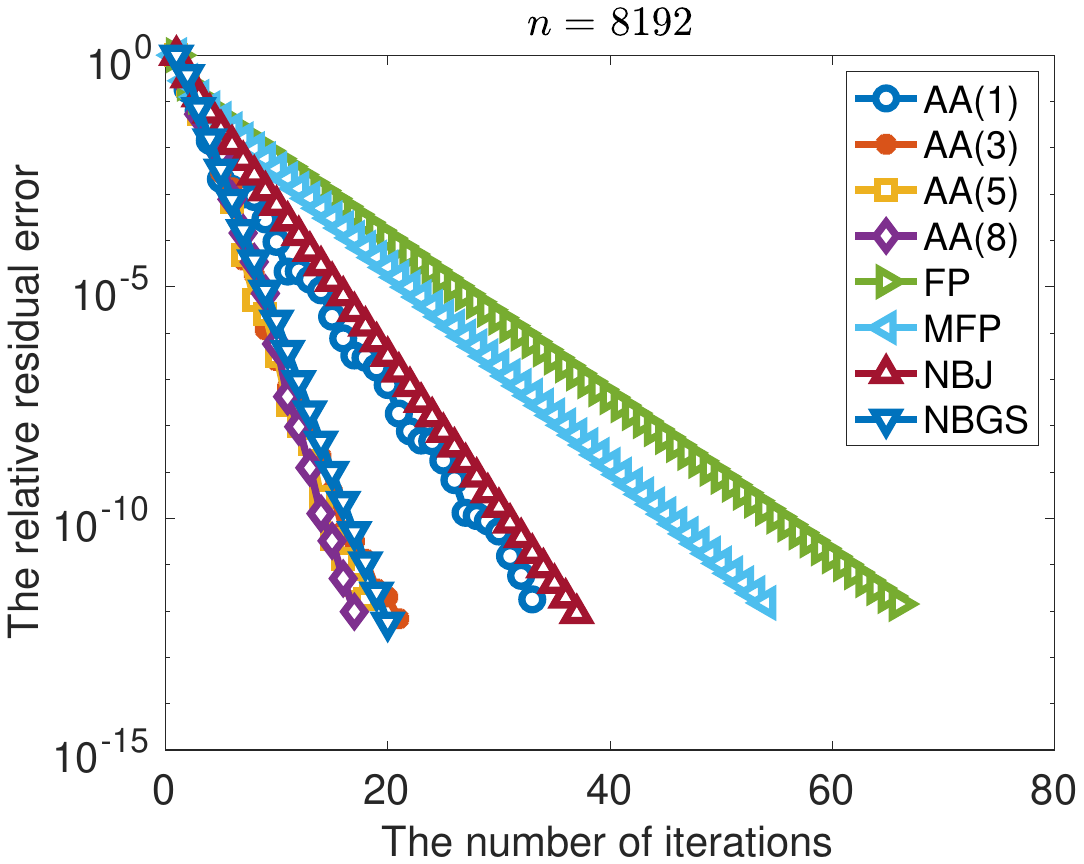}} \\
  \caption{Iteration histories for $(a, c) = (10^{-1},1-10^{-1})$ with various problem sizes $n = 1024, 2048, 4096, 8192$.}
  \label{fig:IterHistory_alphaC10-1}
\end{figure}

\begin{figure}
  \centering
  \subfigure{\includegraphics[width=0.42\textwidth]{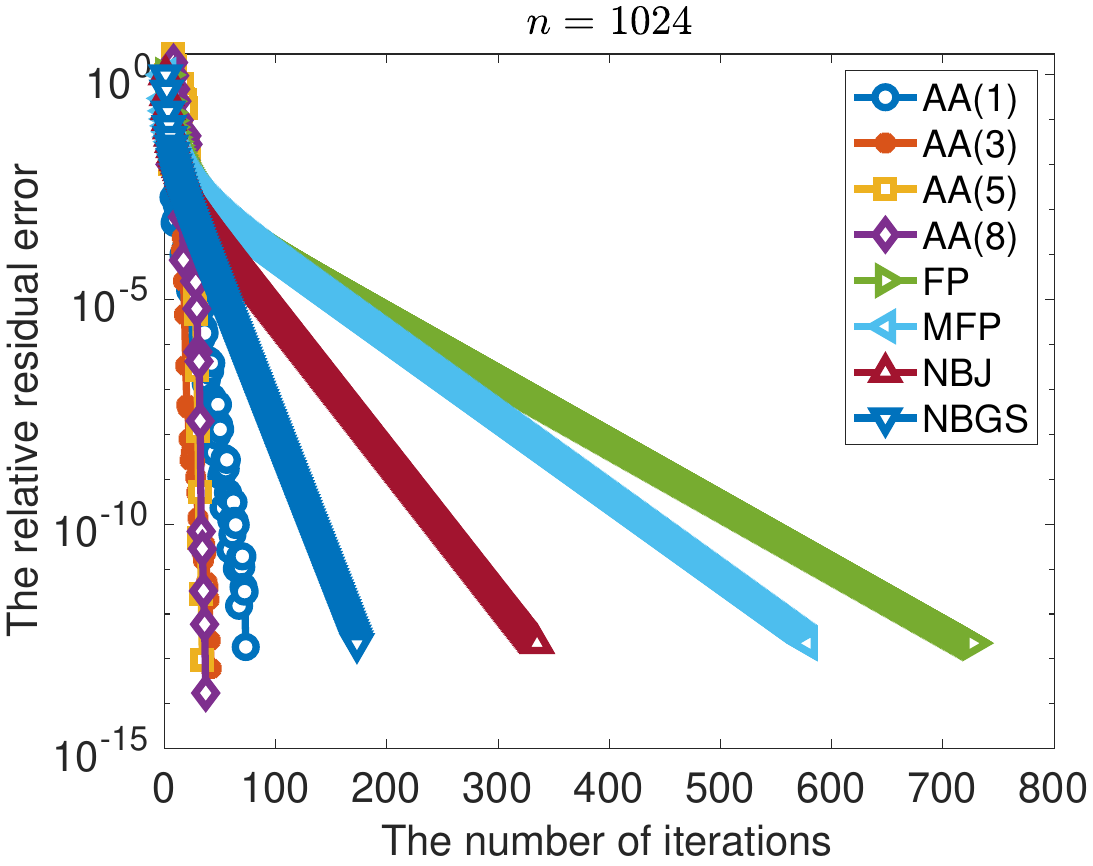}}\quad
  \subfigure{\includegraphics[width=0.42\textwidth]{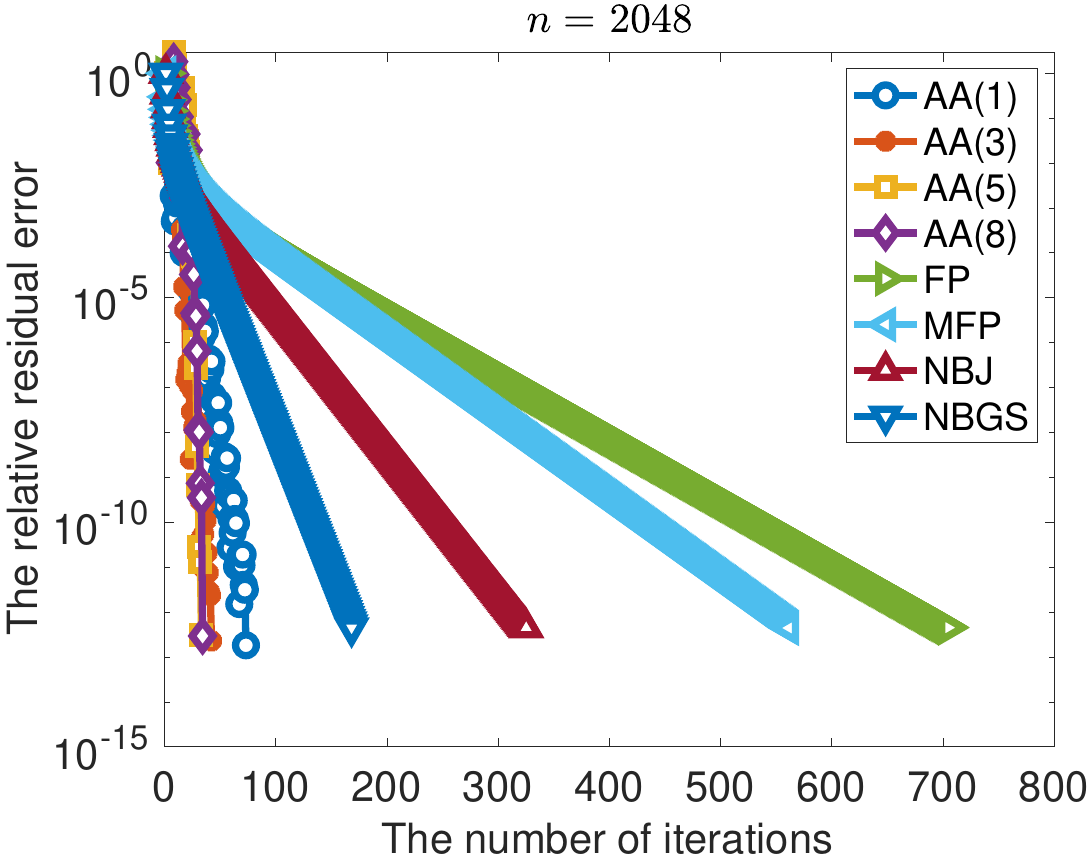}} \\
  \subfigure{\includegraphics[width=0.42\textwidth]{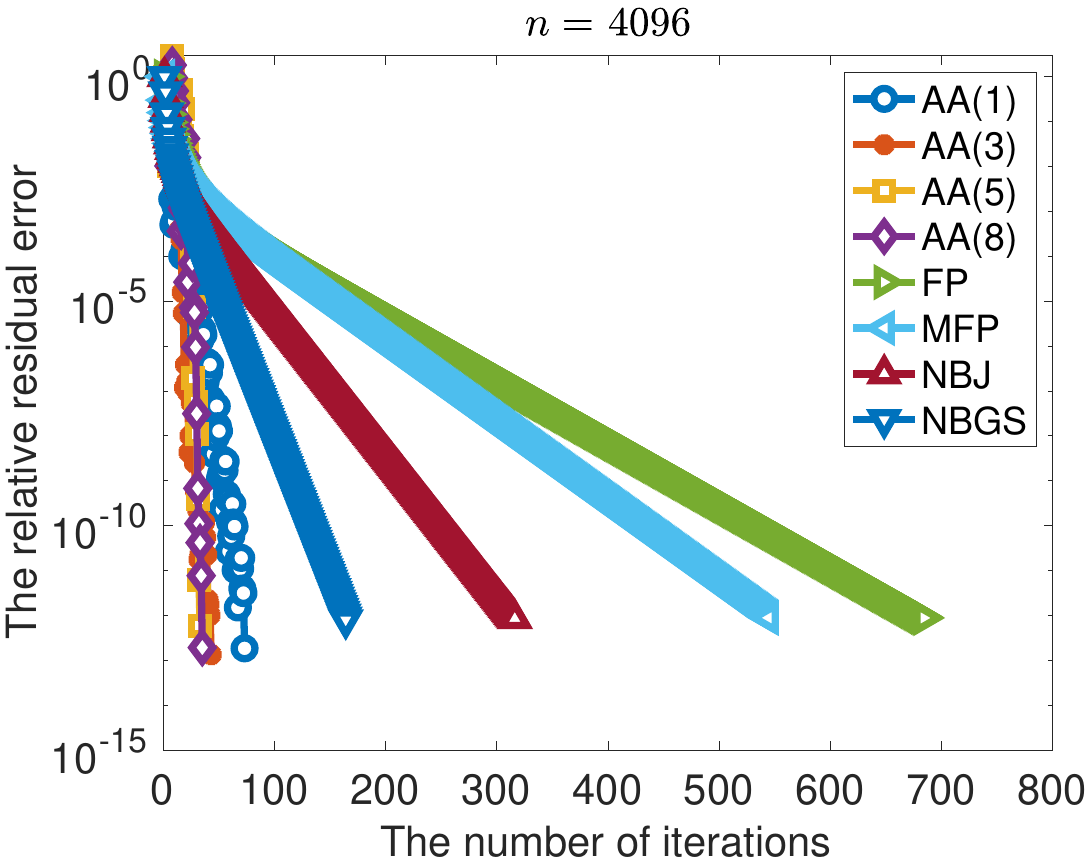}}\quad
  \subfigure{\includegraphics[width=0.42\textwidth]{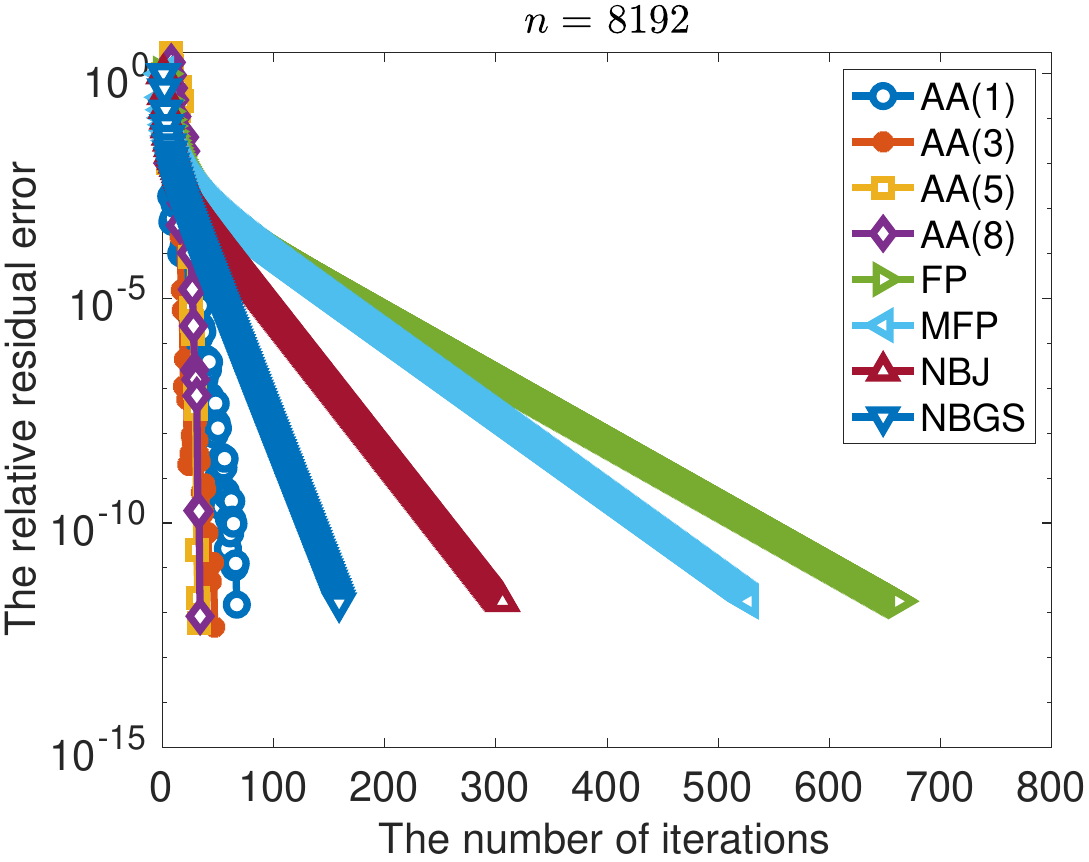}} \\
  \caption{Iteration histories for $(a, c) = (10^{-3},1-10^{-3})$ with various problem sizes $n = 1024, 2048, 4096, 8192$.}
  \label{fig:IterHistory_alphaC10-3}
\end{figure}

Tables \ref{tab:n1024}, \ref{tab:n2048}, \ref{tab:n4096} and \ref{tab:n8192}
report the overall numerical results for seven test cases 
with problem sizes $n = 1024, 2048, 4096$ and $8192$,
further illustrating the significant improvement in the performance of AA 
over other fixed-point iterative methods.
Specifically, AA requires fewer iterations and less computation time than 
FP, MFP and NBJ for each problem size across all seven test cases.
The only exceptions are the regular cases $(a,c) = (0.9,0.1)$ and $(0.1,0.9)$, 
where AA slightly underperforms NBGS in terms of iteration count and computation time.
In all other cases, AA demonstrates superior performance compared to 
the four fixed-point iterative methods,
with its advantages becoming particularly significant
as the pair of parameters $(a,c)$ approaches the singular case $(0,1)$.
In such cases, AA achieves significantly fewer iterations, reduced computation time,
and improved numerical accuracy.
For instance, compared to NBGS, even with depth $m = 1$, 
AA achieves a 720-fold reduction in the number of iterations
and a 360-fold reduction in CPU time
for the case $(a,c) = (10^{-9},1-10^{-9})$ when $n = 1024$,
and the corresponding reductions are 970-fold and 160-fold, respectively,
for the same case when $n = 8192$.

\begin{sidewaystable}
\centering
\caption{Numerical results for $n = 1024$}
\label{tab:n1024}
\begin{tabular}{cccccccccc}
  \toprule
  $(a,c)$ & Item & AA(1) & AA(3) & AA(5) & AA(8) & FP & MFP & NBJ & NBGS \\
  \midrule
  $(0.9,0.1)$ & IT & 7 & 6 & 6 & 6 & 9 & 8 & 7 & 5 \\
  & CPU & 0.0119 & 0.0111 & 0.0127 & 0.0100 & 0.0111 & 0.0117 & 0.0117 & 0.0107 \\
  & RES & 3.7180e-15 & 8.5951e-14 & 6.2769e-14 & 6.2769e-14 & 6.3425e-15 & 1.7322e-13 & 6.2987e-14 & 4.4191e-14 \\
  $(0.1,0.9)$ & IT & 37 & 22 & 20 & 19 & 71 & 58 & 39 & 21 \\
  & CPU & 0.0159 & 0.0141 & 0.0177 & 0.0171 & 0.0203 & 0.0163 & 0.0158 & 0.0139 \\
  & RES & 1.2577e-13 & 7.3995e-14 & 5.6501e-14 & 3.7234e-14 & 1.9562e-13 & 2.1107e-13 & 1.9645e-13 & 1.2801e-13 \\
  $(10^{-2},1-10^{-2})$ & IT & 70 & 29 & 25 & 23 & 242 & 194 & 117 & 61 \\
  & CPU & 0.0196 & 0.0154 & 0.0162 & 0.0173 & 0.0453 & 0.0303 & 0.0244 & 0.0195 \\
  & RES & 1.1239e-13 & 7.4448e-14 & 2.0664e-13 & 1.0577e-13 & 2.2173e-13 & 2.1277e-13 & 1.9614e-13 & 1.7487e-13 \\
  $(10^{-4},1-10^{-4})$ & IT & 119 & 42 & 34 & 34 & 2100 & 1667 & 955 & 494 \\
  & CPU & 0.0293 & 0.0176 & 0.0171 & 0.0169 & 0.2723 & 0.2225 & 0.1309 & 0.0620 \\
  & RES & 1.2653e-13 & 1.4410e-14 & 5.3285e-14 & 3.4963e-14 & 2.2656e-13 & 2.2689e-13 & 2.2176e-13 & 2.1726e-13 \\
  $(10^{-6},1-10^{-6})$ & IT & 114 & 57 & 41 & 42 & 16528 & 13143 & 7531 & 3915 \\
  & CPU & 0.0274 & 0.0207 & 0.0191 & 0.0185 & 2.5279 & 1.8203 & 1.1512 & 0.4395 \\
  & RES & 6.2976e-14 & 1.0726e-13 & 1.1131e-14 & 9.1490e-16 & 2.2722e-13 & 2.2676e-13 & 2.2554e-13 & 2.2646e-13 \\
  $(10^{-8},1-10^{-8})$ & IT & 108 & 53 & 41 & 48 & 119319 & 95406 & 55398 & 29168 \\
  & CPU & 0.0282 & 0.0237 & 0.0173 & 0.0260 & 18.1731 & 13.2879 & 6.8167 & 3.2711 \\
  & RES & 1.5058e-13 & 7.3612e-14 & 1.7471e-13 & 1.9487e-13 & 2.2733e-13 & 2.2733e-13 & 2.2687e-13 & 2.2626e-13 \\
  $(10^{-9},1-10^{-9})$ & IT & 106 & 62 & 49 & 52 & 304534 & 244626 & 143488 & 76421 \\
  & CPU & 0.0283 & 0.0283 & 0.0188 & 0.0241 & 43.7372 & 36.4205 & 16.3861 & 10.2075 \\
  & RES & 6.7664e-14 & 1.0402e-13 & 2.0956e-13 & 2.7799e-14 & 2.2715e-13 & 2.2730e-13 & 2.2730e-13 & 2.2730e-13 \\
  \bottomrule
\end{tabular}
\end{sidewaystable}

\begin{sidewaystable}
\centering
\caption{Numerical results for $n = 2048$}
\label{tab:n2048}
\begin{tabular}{cccccccccc}
  \toprule
  $(a,c)$ & Item & AA(1) & AA(3) & AA(5) & AA(8) & FP & MFP & NBJ & NBGS \\
  \midrule
  $(0.9,0.1)$ & IT & 6 & 6 & 6 & 6 & 8 & 8 & 7 & 5 \\
  & CPU & 0.0593 & 0.0554 & 0.0554 & 0.0554 & 0.0583 & 0.0608 & 0.0589 & 0.0556 \\
  & RES & 3.7858e-13 & 8.6170e-14 & 6.2769e-14 & 6.2769e-14 & 3.8689e-13 & 1.7321e-13 & 6.2987e-14 & 4.4191e-16 \\
  $(0.1,0.9)$ & IT & 37 & 22 & 19 & 18 & 69 & 57 & 38 & 21 \\
  & CPU & 0.1070 & 0.0836 & 0.0780 & 0.0768 & 0.1528 & 0.1362 & 0.1054 & 0.0800 \\ 
  & RES & 1.2564e-13 & 7.3519e-14 & 3.3556e-13 & 2.3628e-13 & 4.3260e-13 & 3.4456e-13 & 4.2847e-13 & 1.2863e-13 \\
  $(10^{-2},1-10^{-2})$ & IT & 69 & 28 & 25 & 23 & 236 & 189 & 114 & 59 \\
  & CPU & 0.1595 & 0.0932 & 0.0883 & 0.0853 & 0.4168 & 0.3377 & 0.2213 & 0.1389 \\ 
  & RES & 3.0851e-13 & 4.4862e-13 & 2.0830e-13 & 1.0773e-13 & 4.1767e-13 & 4.1541e-13 & 3.8832e-13 & 4.3851e-13 \\
  $(10^{-4},1-10^{-4})$ & IT & 115 & 42 & 33 & 34 & 2031 & 1613 & 925 & 478 \\
  & CPU & 0.2363 & 0.1179 & 0.1024 & 0.1048 & 3.2430 & 2.5565 & 1.4444 & 0.7820 \\ 
  & RES & 2.4404e-13 & 3.2872e-13 & 4.3675e-14 & 3.7916e-14 & 4.5448e-13 & 4.5219e-13 & 4.4473e-13 & 4.4831e-13 \\
  $(10^{-6},1-10^{-6})$ & IT & 100 & 45 & 38 & 41 & 15833 & 12598 & 7230 & 3766 \\
  & CPU & 0.2118 & 0.1365 & 0.1111 & 0.1173 & 24.9623 & 19.9647 & 11.0088 & 5.8172 \\ 
  & RES & 2.7612e-13 & 4.2402e-13 & 4.0557e-14 & 1.5095e-14 & 4.5472e-13 & 4.5456e-13 & 4.5410e-13 & 4.5196e-13 \\
  $(10^{-8},1-10^{-8})$ & IT & 109 & 43 & 39 & 52 & 112384 & 89976 & 52336 & 27666 \\
  & CPU & 0.2277 & 0.1305 & 0.1127 & 0.1381 & 174.8984 & 142.5888 & 79.3217 & 42.3485 \\ 
  & RES & 1.8417e-13 & 2.6724e-13 & 1.9501e-13 & 1.5734e-13 & 4.5462e-13 & 4.5462e-13 & 4.5386e-13 & 4.5401e-13 \\
  $(10^{-9},1-10^{-9})$ & IT & 108 & 55 & 55 & 50 & 282652 & 227488 & 134008 & 71720 \\
  & CPU & 0.2248 & 0.1393 & 0.1401 & 0.1329 & 433.1914 & 353.6441 & 204.5411 & 109.9595 \\ 
  & RES & 3.9098e-14 & 3.4700e-13 & 1.0996e-14 & 1.0370e-13 & 4.5472e-13 & 4.5441e-13 & 4.5472e-13 & 4.5426e-13 \\
  \bottomrule
\end{tabular}
\end{sidewaystable}

\begin{sidewaystable}
\centering
\caption{Numerical results for $n = 4096$}
\label{tab:n4096}
\begin{tabular}{cccccccccc}
  \toprule
  $(a,c)$ & Item & AA(1) & AA(3) & AA(5) & AA(8) & FP & MFP & NBJ & NBGS \\
  \midrule
  $(0.9,0.1)$ & IT & 6 & 6 & 6 & 6 & 8 & 8 & 7 & 5 \\
  & CPU & 0.2382 & 0.2259 & 0.2301 & 0.2347 & 0.2487 & 0.2503 & 0.2455 & 0.2290 \\
  & RES & 3.7858e-13 & 8.6389e-14 & 6.2987e-14 & 6.2987e-14 & 3.8711e-13 & 1.7321e-13 & 6.2768e-14 & 4.4191e-16 \\
  $(0.1,0.9)$ & IT & 36 & 21 & 19 & 18 & 68 & 56 & 37 & 20 \\
  & CPU & 0.4272 & 0.3255 & 0.3156 & 0.3157 & 0.6352 & 0.5604 & 0.4404 & 0.3205 \\ 
  & RES & 4.9122e-13 & 6.7939e-13 & 3.3733e-13 & 2.3734e-13 & 6.4263e-13 & 5.6174e-13 & 8.8138e-13 & 5.7073e-13 \\
  $(10^{-2},1-10^{-2})$ & IT & 68 & 26 & 24 & 23 & 229 & 184 & 111 & 58 \\
  & CPU & 0.6622 & 0.3572 & 0.3515 & 0.3493 & 1.6709 & 1.3866 & 0.9160 & 0.5644 \\ 
  & RES & 7.3474e-13 & 6.2237e-13 & 5.6842e-13 & 1.0826e-13 & 8.7396e-13 & 8.1203e-13 & 7.7411e-13 & 6.9077e-13 \\
  $(10^{-4},1-10^{-4})$ & IT & 112 & 37 & 32 & 32 & 1962 & 1559 & 895 & 463 \\
  & CPU & 0.9254 & 0.4330 & 0.4048 & 0.4094 & 12.8604 & 10.3231 & 5.9628 & 3.1721 \\ 
  & RES & 7.0958e-13 & 8.2890e-13 & 1.3132e-13 & 4.6285e-13 & 9.0652e-13 & 9.0419e-13 & 8.8695e-13 & 9.0046e-13 \\
  $(10^{-6},1-10^{-6})$ & IT & 100 & 52 & 38 & 40 & 15140 & 12057 & 6928 & 3615 \\
  & CPU & 0.8389 & 0.5318 & 0.4469 & 0.4637 & 97.8450 & 77.7533 & 45.1426 & 23.4178 \\ 
  & RES & 4.7920e-13 & 2.1132e-13 & 7.0744e-14 & 1.3569e-14 & 9.0909e-13 & 9.0817e-13 & 9.0878e-13 & 9.0725e-13 \\
  $(10^{-8},1-10^{-8})$ & IT & 104 & 52 & 41 & 50 & 105468 & 84556 & 49331 & 26165 \\
  & CPU & 0.8636 & 0.5297 & 0.4672 & 0.5346 & 678.0004 & 543.9350 & 319.1383 & 168.3982 \\ 
  & RES & 7.3068e-13 & 8.3467e-13 & 6.6746e-13 & 1.2801e-13 & 9.0905e-13 & 9.0874e-13 & 9.0920e-13 & 9.0737e-13 \\
  $(10^{-9},1-10^{-9})$ & IT & 107 & 56 & 49 & 60 & 260734 & 210351 & 124431 & 66969 \\
  & CPU & 0.8816 & 0.5588 & 0.5203 & 0.5971 & 1683.4272 & 1353.8946 & 801.2046 & 561.6521 \\ 
  & RES & 2.2160e-13 & 6.7625e-13 & 4.9518e-13 & 3.8368e-13 & 9.0909e-13 & 9.0833e-13  & 9.0940e-13 & 9.0864e-13 \\
  \bottomrule
\end{tabular}
\end{sidewaystable}

\begin{sidewaystable}
\centering
\caption{Numerical results for $n = 8192$}
\label{tab:n8192}
\begin{tabular}{cccccccccc}
  \toprule
  $(a,c)$ & Item & AA(1) & AA(3) & AA(5) & AA(8) & FP & MFP & NBJ & NBGS \\
  \midrule
  $(0.9,0.1)$ & IT & 6 & 6 & 6 & 6 & 8 & 8 & 7 & 5 \\
  & CPU & 1.0840 & 1.0987 & 1.0951 & 1.0877 & 1.1754 & 1.1855 & 1.1403 & 1.0552 \\
  & RES & 3.7858e-13 & 8.6389e-14 & 6.3206e-14 & 6.3206e-14 & 3.8711e-13 & 1.7321e-13 & 6.2768e-14 & 4.4191e-16 \\
  $(0.1,0.9)$ & IT & 33 & 21 & 18 & 17 & 66 & 54 & 37 & 20 \\
  & CPU & 2.0959 & 1.6639 & 1.5476 & 1.5151 & 3.3465 & 3.3578 & 2.5705 & 1.6243 \\ 
  & RES & 1.7853e-12 & 6.8092e-13 & 1.7769e-12 & 9.6635e-13 & 1.4162e-12 & 1.4974e-12 & 8.8421e-13 & 5.7260e-13 \\
  $(10^{-2},1-10^{-2})$ & IT & 63 & 25 & 22 & 21 & 223 & 178 & 108 & 56 \\
  & CPU & 3.1793 & 1.8143 & 1.6897 & 1.6514 & 9.1696 & 9.0883 & 5.7975 & 2.9814 \\ 
  & RES & 1.8084e-12 & 1.7286e-12 & 1.3659e-12 & 1.2474e-12 & 1.6488e-12 & 1.8117e-12 & 1.5294e-12 & 1.7226e-12 \\
  $(10^{-4},1-10^{-4})$ & IT & 109 & 34 & 31 & 32 & 1893 & 1505 & 864 & 448 \\
  & CPU & 4.9374 & 2.1441 & 2.0285 & 2.0497 & 71.3948 & 70.9941 & 41.1794 & 20.8786 \\ 
  & RES & 9.8466e-13 & 6.5734e-13 & 4.6889e-13 & 1.9392e-13 & 1.8175e-12 & 1.8035e-12 & 1.8150e-12 & 1.7896e-12 \\
  $(10^{-6},1-10^{-6})$ & IT & 100 & 49 & 37 & 39 & 14448 & 11514 & 6628 & 3466 \\
  & CPU & 4.5839 & 2.6971 & 2.2662 & 2.3174 & 184.0034 & 164.5215 & 95.1652 & 50.3798 \\ 
  & RES & 6.5132e-13 & 1.3125e-12 & 6.6824e-13 & 1.0122e-12 & 1.8180e-12 & 1.8157e-12 & 1.8181e-12 & 1.8102e-12 \\
  $(10^{-8},1-10^{-8})$ & IT & 87 & 49 & 42 & 46 & 98540 & 79120 & 46324 & 24660 \\
  & CPU & 4.1214 & 2.6954 & 2.4590 & 2.6037 & 990.0935 & 746.2888 & 386.8558 & 233.7374 \\ 
  & RES & 1.4841e-12 & 7.6747e-13 & 1.0326e-12 & 8.0241e-13 & 1.8188e-12 & 1.8174e-12 & 1.8174e-12 & 1.8178e-12 \\
  $(10^{-9},1-10^{-9})$ & IT & 64 & 49 & 44 & 49 & 238855 & 193179 & 114927 & 62217 \\
  & CPU & 3.2395 & 2.7019 & 2.5261 & 2.6625 & 2146.4811 & 1724.1584 & 895.6699 & 517.4697 \\ 
  & RES & 1.6310e-12 & 1.7158e-12 & 6.5997e-13 & 1.5136e-12 & 1.8188e-12 & 1.8185e-12 & 1.8183e-12 & 1.8189e-12 \\
  \bottomrule
\end{tabular}
\end{sidewaystable}

It is worth noting that, for each test case, 
the number of iterations tends to stagnate as the problem size $n$ grows.
Furthermore, for each fixed problem size, 
as the pair of parameters $(a,c)$ approaches the singular case $(0,1)$,
the number of iterations and computation time required by AA increase only slightly,
whereas those of the other four fixed-point iterative methods increase significantly.
This advantage of AA is further illustrated in 
Figures \ref{fig:IterTimeHistory_alphaC10-3}, \ref{fig:IterTimeHistory_alphaC10-5}
and \ref{fig:IterTimeHistory_alphaC10-7},
which depict the iteration histories 
in terms of both the number of iterations and the elapsed time
for the nearly singular cases 
$(a,c) = (10^{-3},1-10^{-3})$, $(10^{-5},1-10^{-5})$ and $(10^{-7},1-10^{-7})$,
respectively, with problem sizes $n = 4096$ and $8192$.

\begin{figure}
  \centering
  \subfigure{\includegraphics[width=0.48\textwidth]{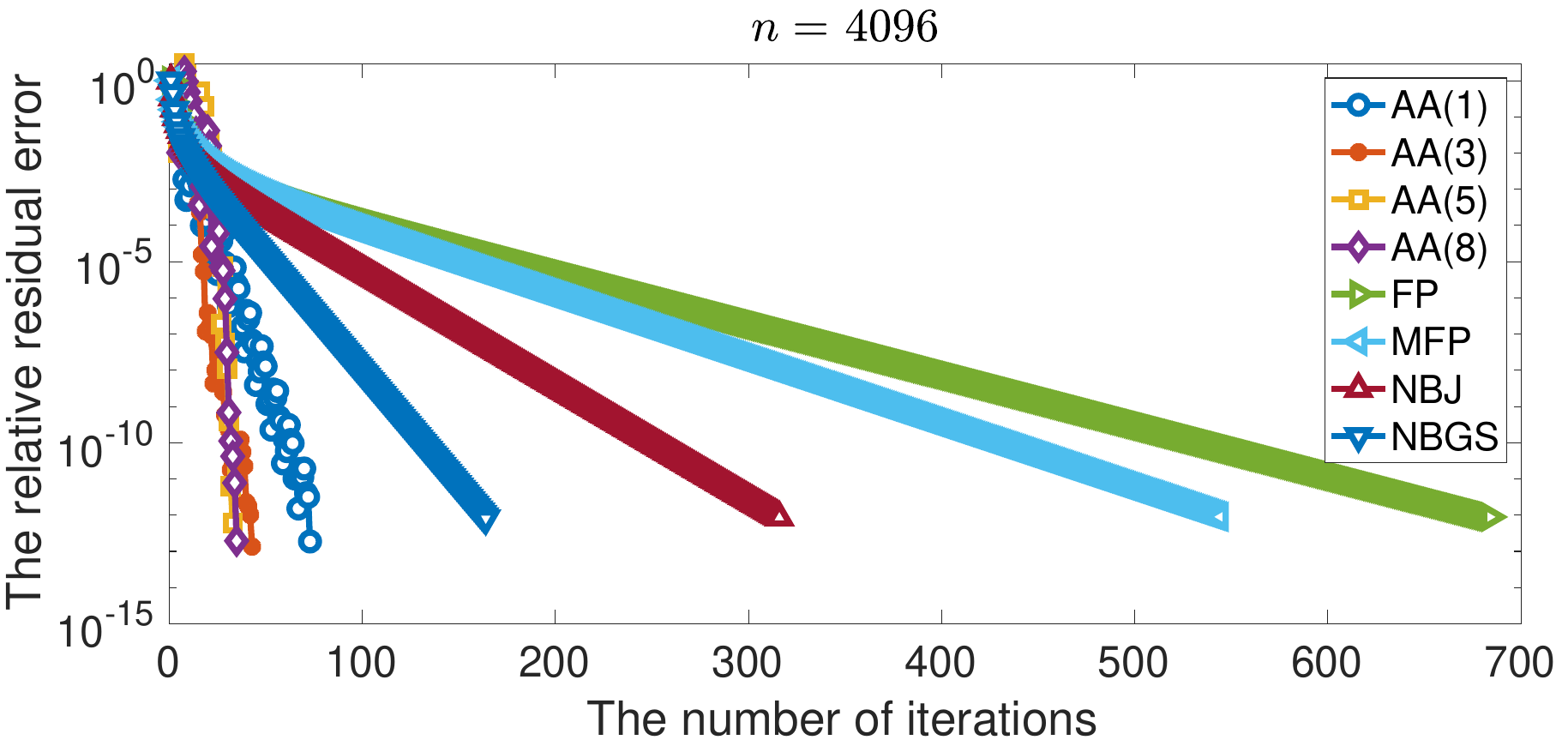}}\quad
  \subfigure{\includegraphics[width=0.48\textwidth]{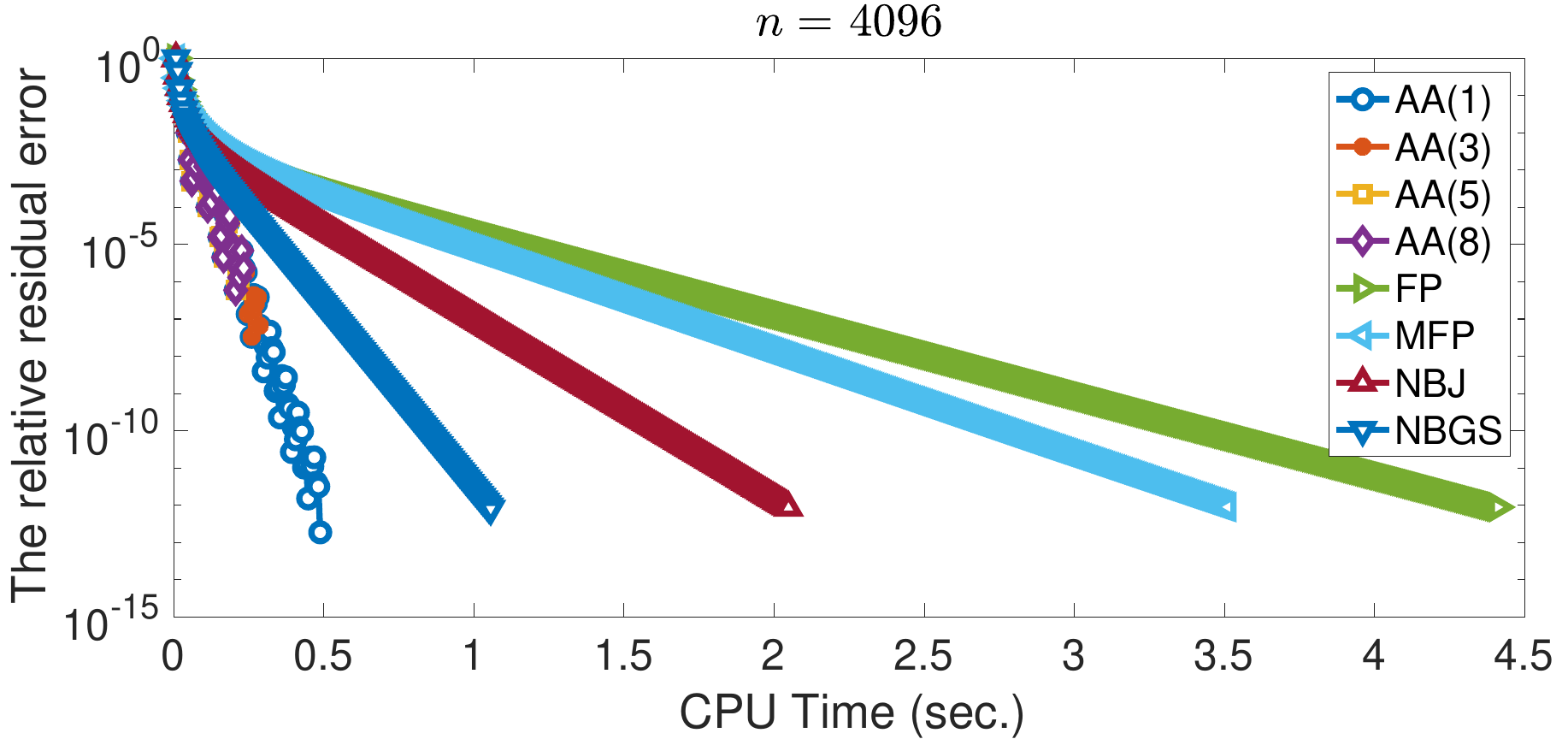}} \\
  \subfigure{\includegraphics[width=0.48\textwidth]{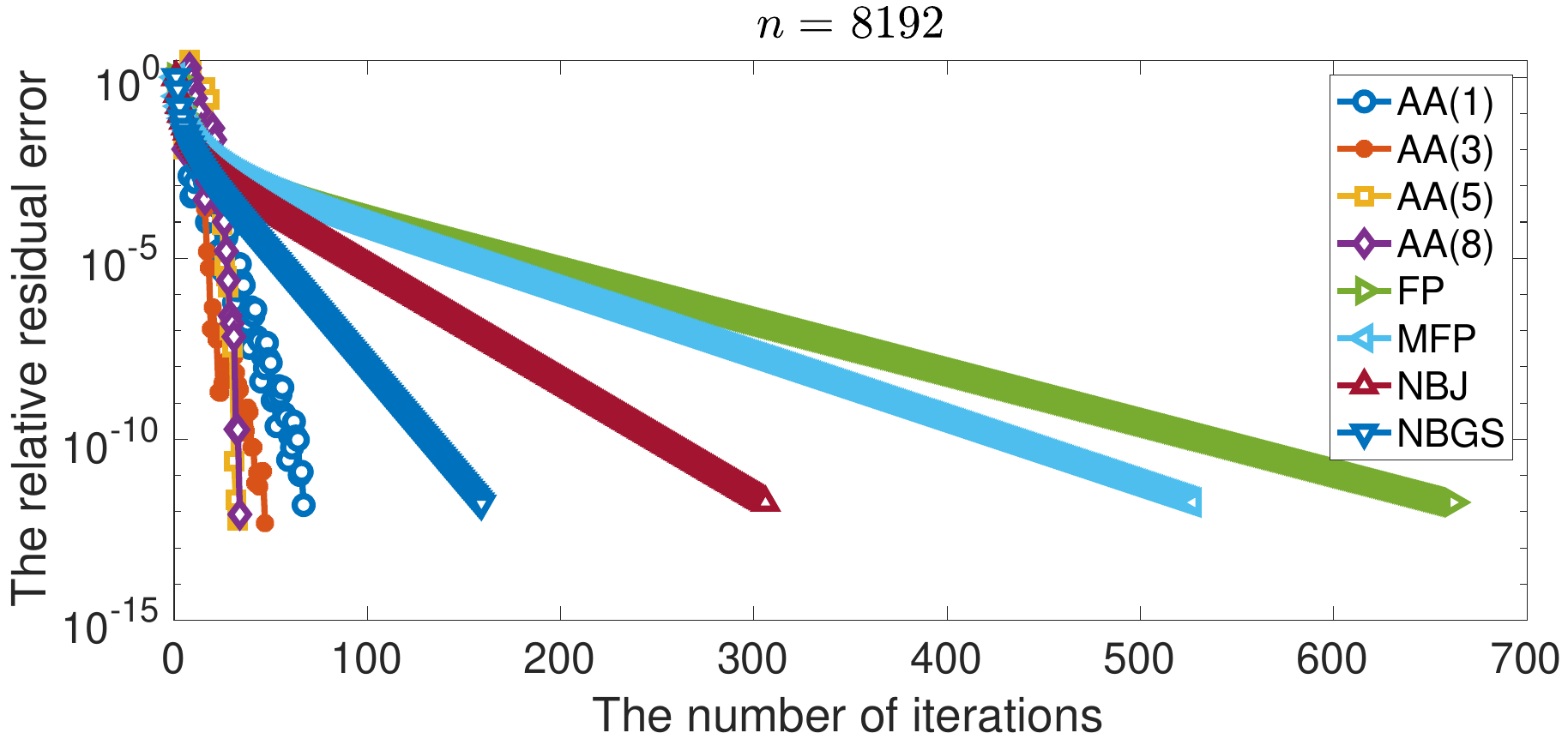}}\quad
  \subfigure{\includegraphics[width=0.48\textwidth]{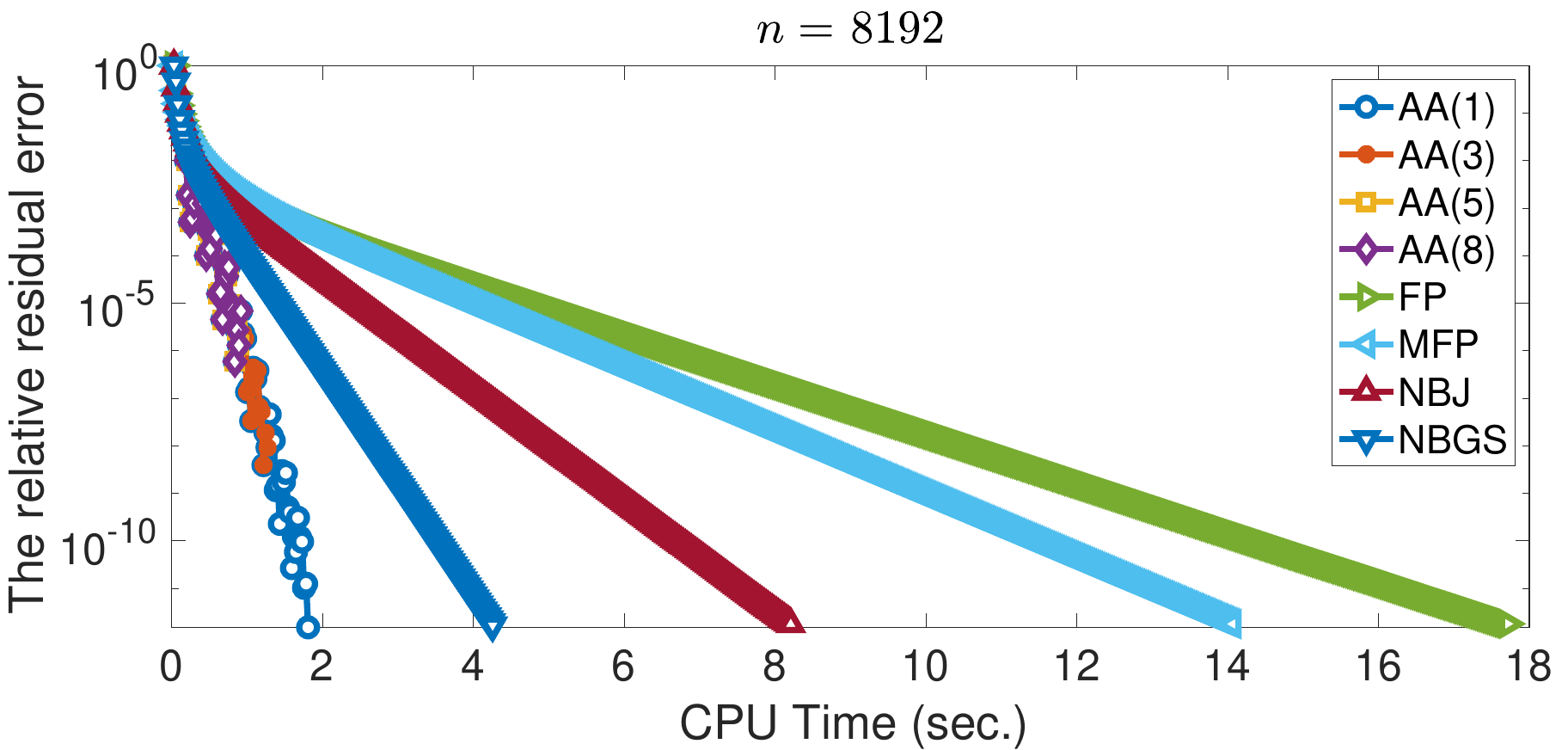}} \\
  \caption{Comparison of Anderson acceleration and other fixed-point methods
  for $(a, c) = (10^{-3},1-10^{-3})$ with problem sizes $n = 4096, 8192$.
  Left: number of iterations. Right: elapsed time.}
  \label{fig:IterTimeHistory_alphaC10-3}
\end{figure}

\begin{figure}
  \centering
  \subfigure{\includegraphics[width=0.48\textwidth]{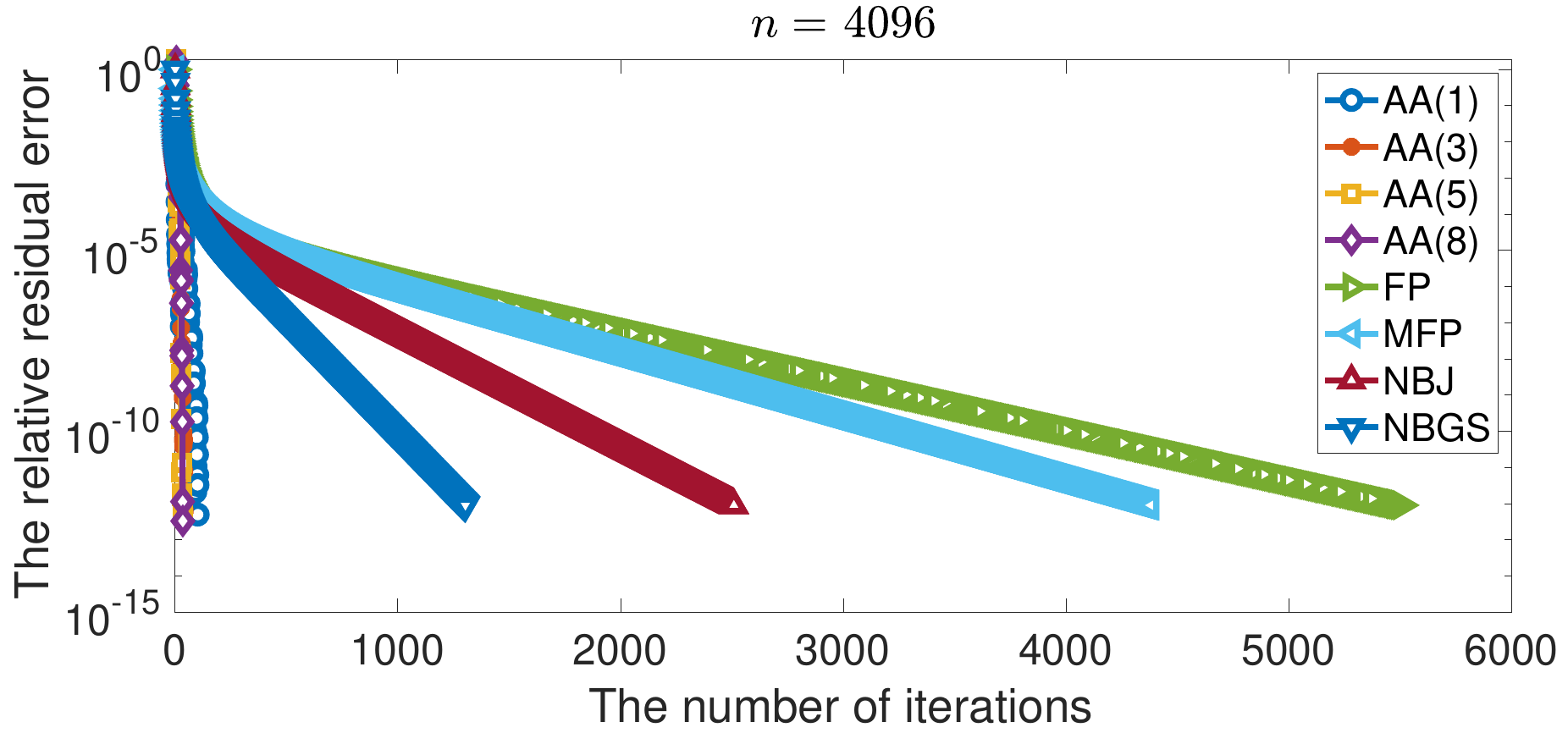}}\quad
  \subfigure{\includegraphics[width=0.48\textwidth]{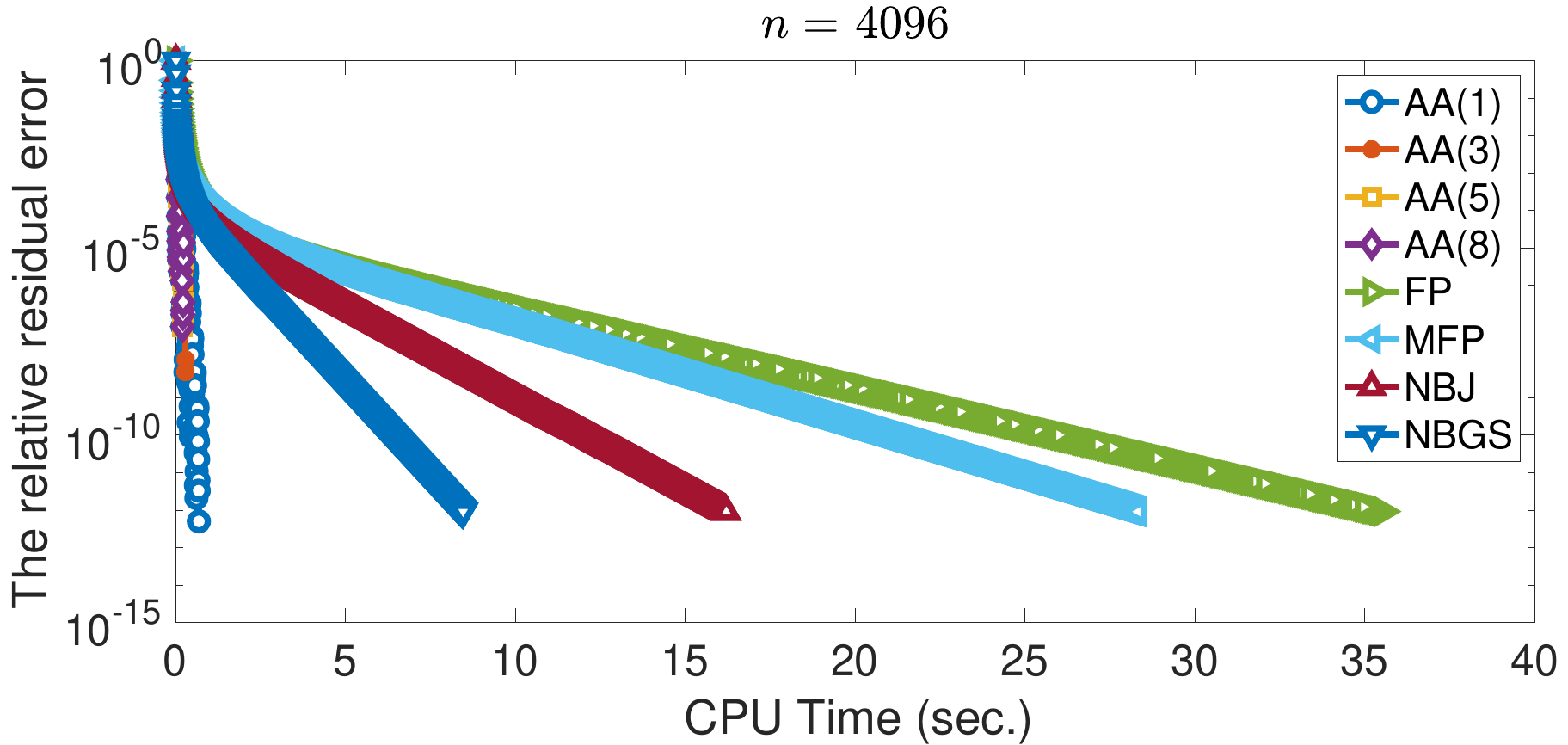}} \\
  \subfigure{\includegraphics[width=0.48\textwidth]{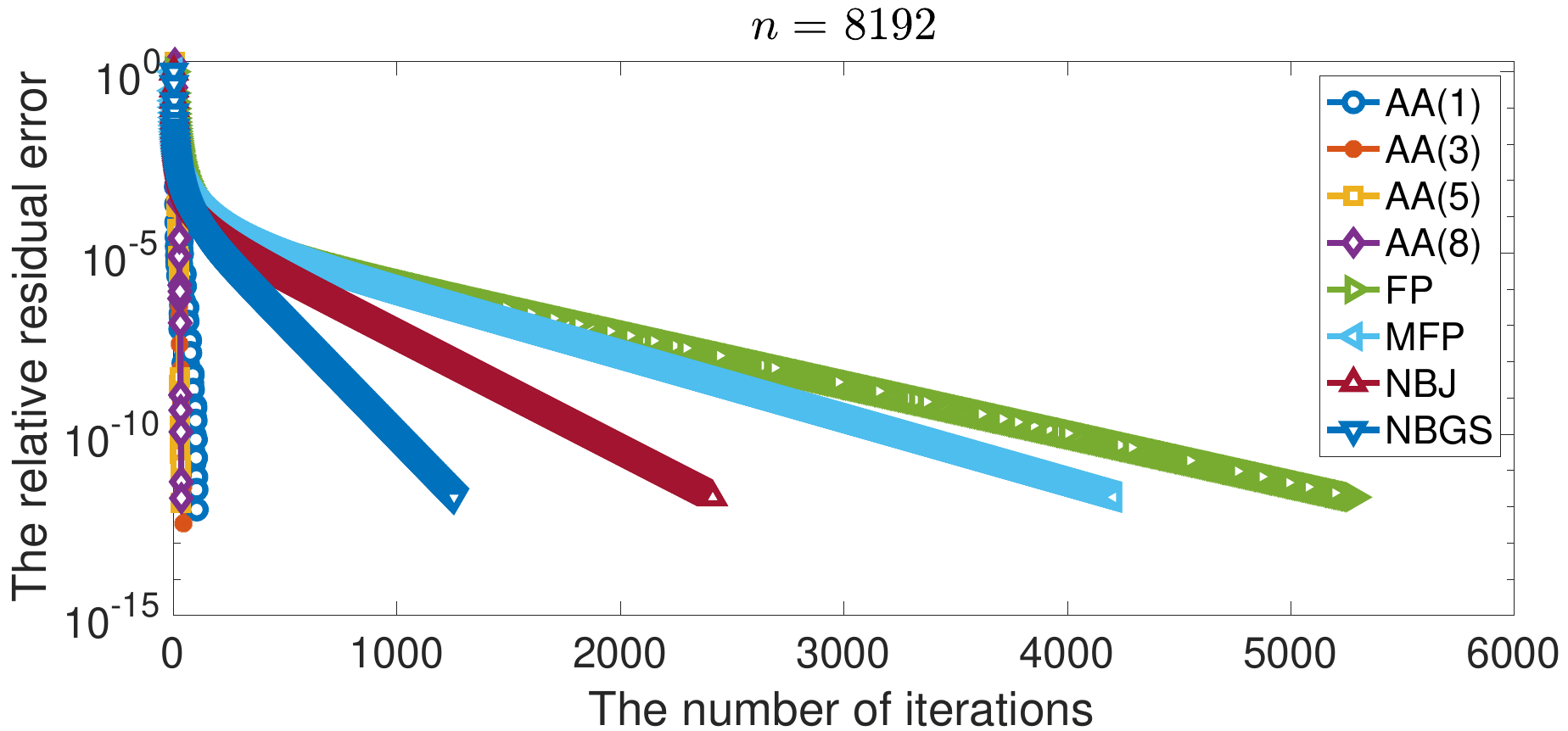}}\quad
  \subfigure{\includegraphics[width=0.48\textwidth]{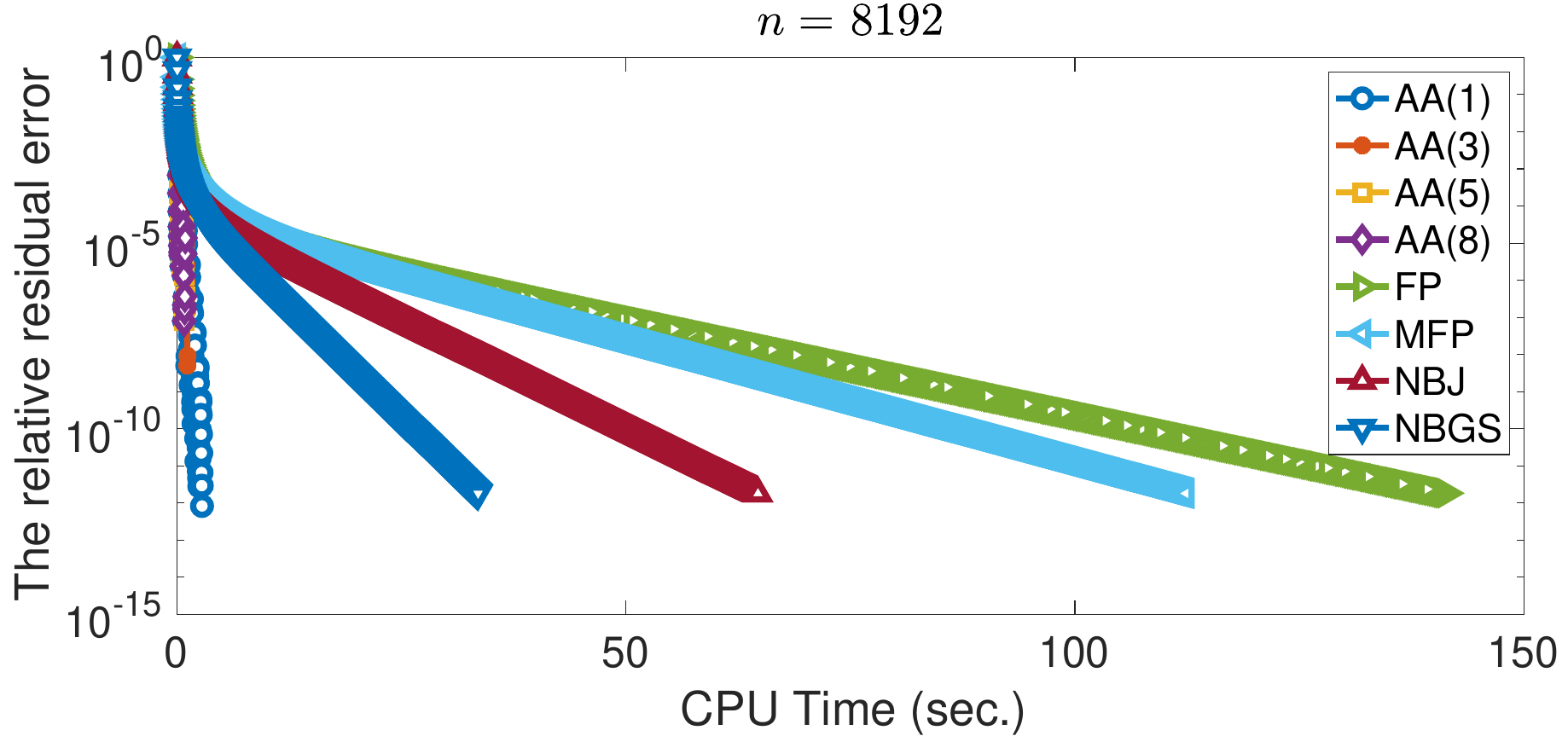}} \\
  \caption{Comparison of Anderson acceleration and other fixed-point methods
  for $(a, c) = (10^{-5},1-10^{-5})$ with problem sizes $n = 4096, 8192$.
  Left: number of iterations. Right: elapsed time.}
  \label{fig:IterTimeHistory_alphaC10-5}
\end{figure}

\begin{figure}
  \centering
  \subfigure{\includegraphics[width=0.48\textwidth]{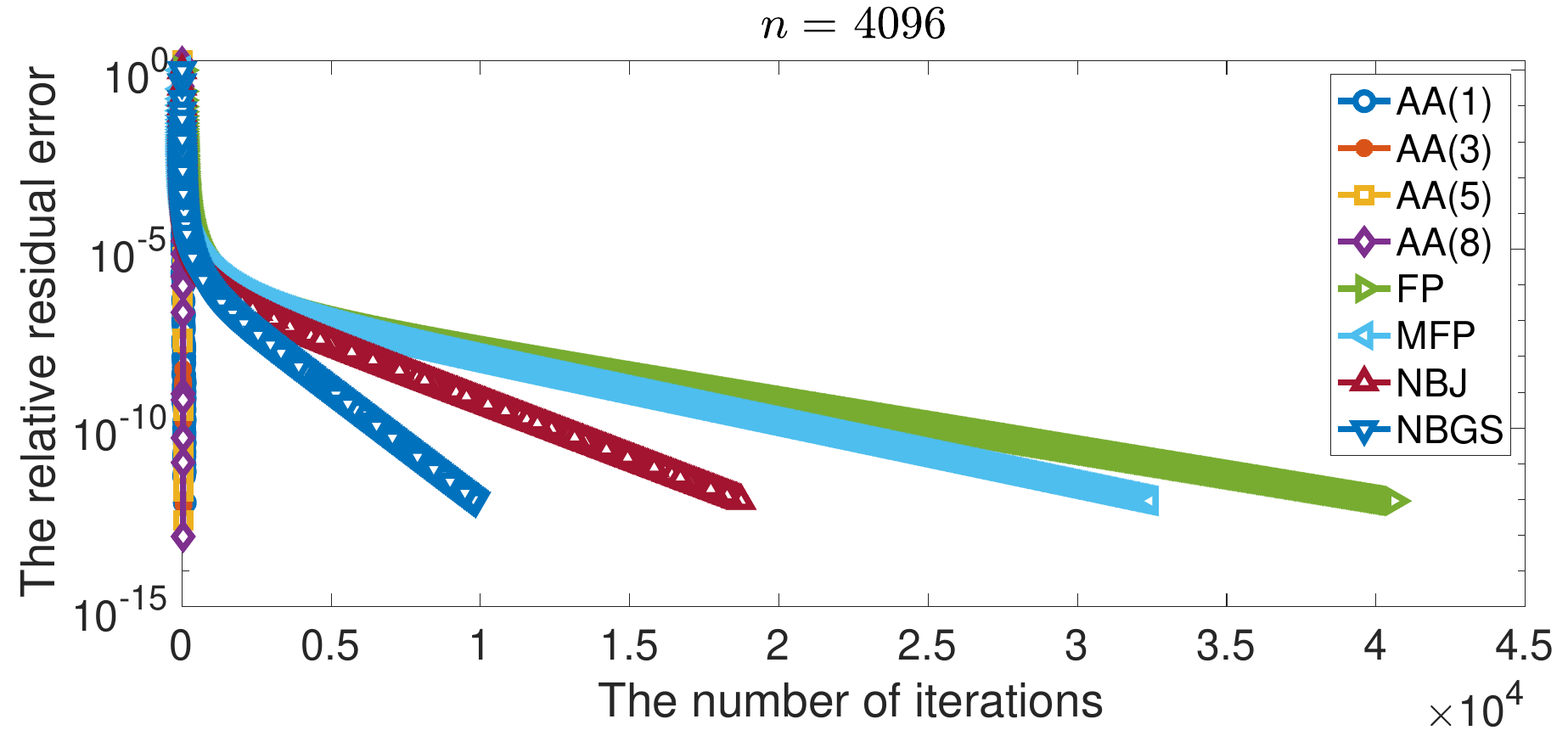}}\quad
  \subfigure{\includegraphics[width=0.48\textwidth]{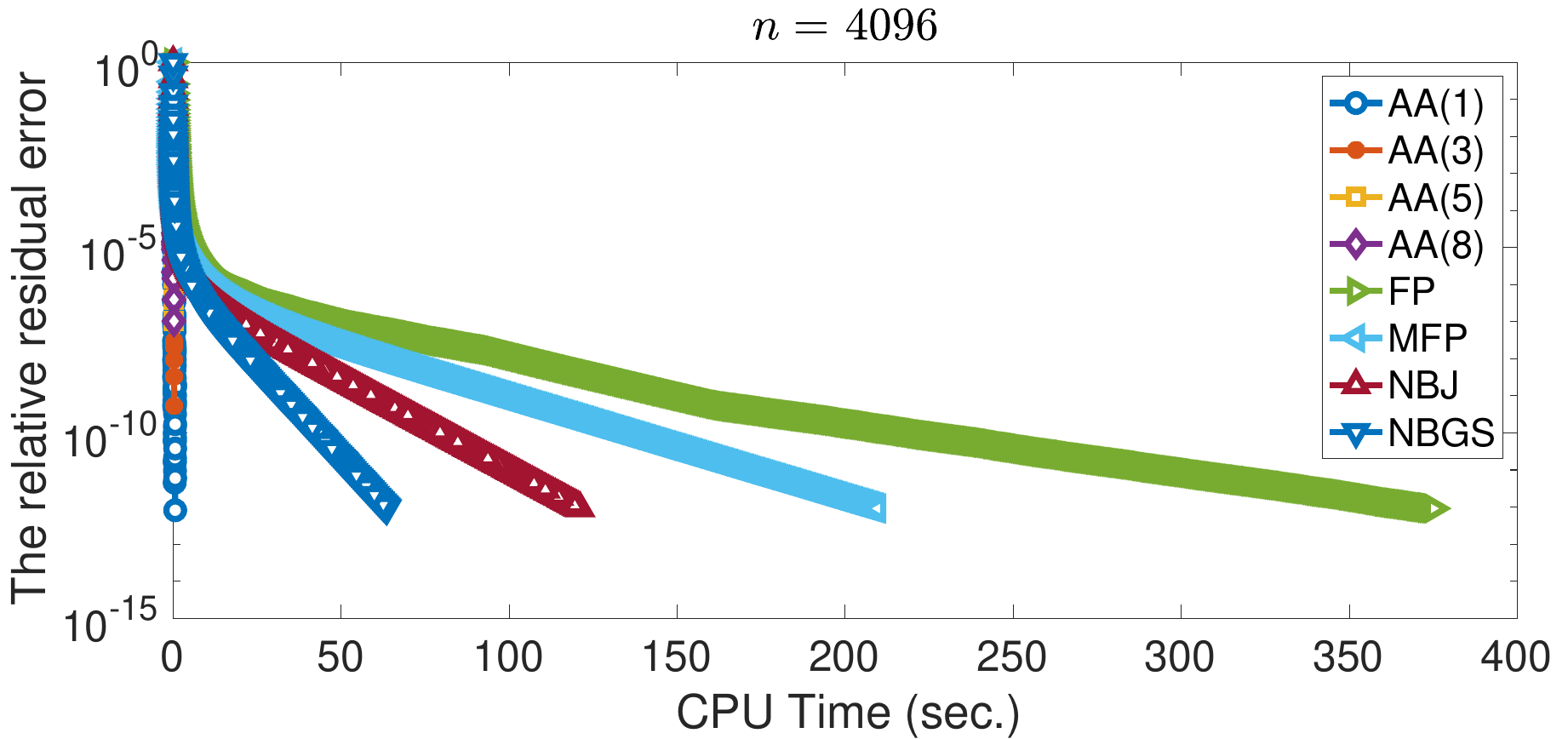}} \\
  \subfigure{\includegraphics[width=0.48\textwidth]{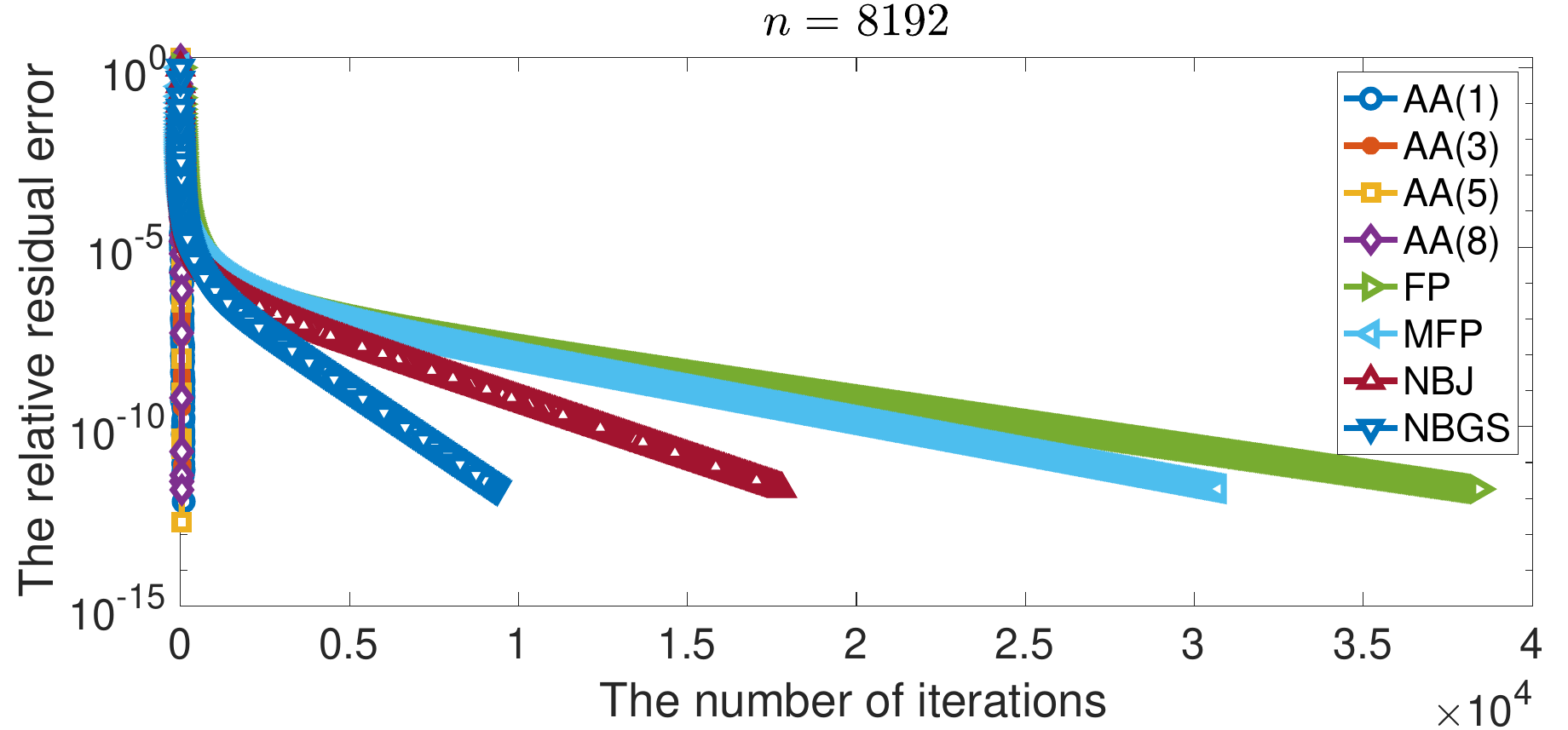}}\quad
  \subfigure{\includegraphics[width=0.48\textwidth]{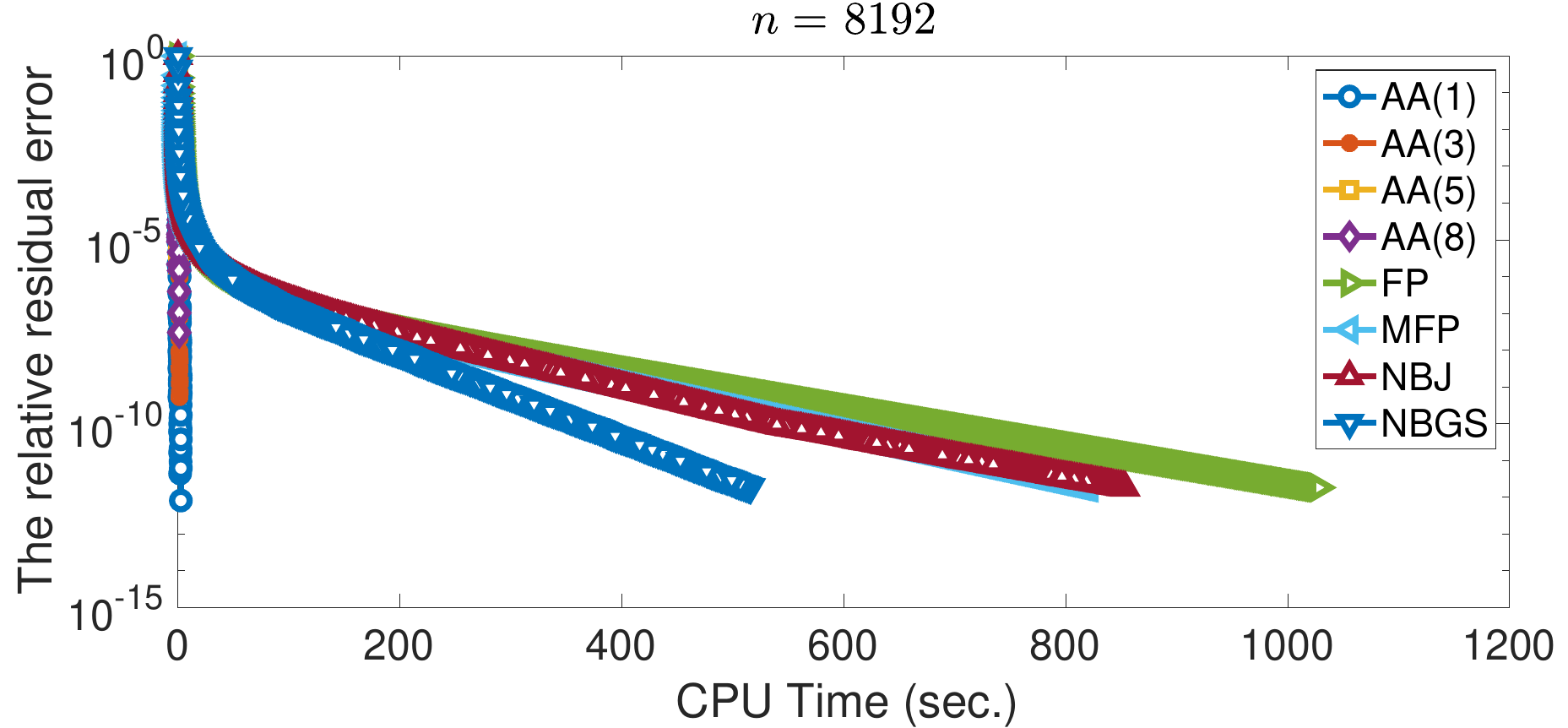}} \\
  \caption{Comparison of Anderson acceleration and other fixed-point methods
  for $(a, c) = (10^{-7},1-10^{-7})$ with problem sizes $n = 4096, 8192$.
  Left: number of iterations. Right: elapsed time.}
  \label{fig:IterTimeHistory_alphaC10-7}
\end{figure}

To conclude, compared to the other four fixed-point iterative methods,
we have found that AA for the NARE,
with varying depths $m$, yields comparable performance in the regular cases.
In addition, for nearly singular cases and large-scale problems,
AA achieves a significant reduction in both the number of iterations and execution time,
while exhibiting better desired accuracy in most cases.

\subsection{Convection-diffusion-reaction equation}
\label{subsec:CDR}

This subsection concerns a nontrivial example for which 
the Jacobian of the nonlinear operator $\myvec{f}$ is H\"older continuous with exponent $\nu\in(0,1)$, 
but fails to be Lipschitz continuous in any neighborhood of the solution.
Consequently, none of the local convergence results for Anderson acceleration established in 
\cite{TothK2015,ChenK2019,EvansPRX2020,PollockR2021} is applicable to this problem, 
since they all require the Jacobian to be Lipschitz continuous.
Throughout this subsection, $x$ and $y$ denote the spatial variables,
whereas boldfaced letters $\myvec{x}, \myvec{y}$ continue to denote vectors in $\RS^n$.

We consider the following steady semilinear convection-diffusion-reaction problem
\begin{equation}
\label{eq:CDR}
\left\{
\begin{aligned}
  -(u_{xx} + u_{yy}) + w_1u_x + w_2u_y &= \lambda |u|^\nu u + \psi(x,y), && (x,y) \in \Omega,\\
  u(x,y) &= 0, && (x,y) \in \partial\Omega,
\end{aligned}
\right.
\end{equation}
where $\Omega = (0,1)\times(0,1)$ and $\partial\Omega$ denotes its boundary.
The positive constants $w_1$ and $w_2$ are the convection coefficients,
$\lambda > 0$ is the coefficient of the nonlinear reaction term, 
$\nu \in (0,1]$ is the exponent of the nonlinearity,
and $\psi(x,y)$ is a given source term.
Problems of this type belong to a broad class of nonlinear convection-diffusion-reaction models
that have been investigated extensively
(see, for example, \cite{BarrenecheaJK2024,GonzalezTaberneroLCG2024}, and the references therein).
For $\lambda = 1, \psi \equiv 0$ and $u \geq 0$, 
the problem \eqref{eq:CDR} reduces to the benchmark problem considered in
\cite{BertacciniGS2007,ChenLW2014,ZhongCG2018,LingLL2022}.

Discretizing the diffusion term $-(u_{xx} + u_{yy})$ in \eqref{eq:CDR} 
by the standard five-point finite-difference scheme 
and the convection term $w_1u_x + w_2u_y$ by second-order central differences
on the uniform grid $(x_i, y_j) = (ih,jh), 1 \leq i,j \leq N$
with mesh size $h = 1/(N+1)$, 
we obtain a nonlinear system of equations of dimension $n=N^{2}$ as follows:
\begin{equation}
\label{eq:CDR_discrete}
K\myvec{u} - \lambda h^{2}\bm{\varphi}(\myvec{u}) - h^{2}\bm{\psi} = \myvec{0}, 
\end{equation}
where 
$$
\myvec{u} = (\myvec{u}_1^\top, \myvec{u}_2^\top, \ldots, \myvec{u}_N^\top)^\top, \quad
\myvec{u}_j = (u_{1j}, u_{2j}, \ldots, u_{Nj})^\top, \quad j = 1,2,\ldots,N,
$$
with $u_{ij}$ denoting the approximation to $u(x_{i},y_{j})$, 
and $\bm{\psi} \in \mathbb{R}^{n}$ collecting the values of the source term $\psi$,
ordered consistently with $\myvec{u}$.
The nonlinear term is defined componentwise by
$$
\bm{\varphi}(\myvec{u}) 
  := |\myvec{u}|^{\nu} \circ \myvec{u}
  = \bigl(|u_{1}|^{\nu}u_{1},\ldots,|u_{n}|^{\nu}u_{n}\bigr)^{\top}, \quad
  |\myvec{u}|^{\nu} := \bigl(|u_{1}|^{\nu},\ldots,|u_{n}|^{\nu}\bigr)^{\top}.
$$
The coefficient matrix $K\in\mathbb{R}^{n\times n}$ is given by
$$
K = I_N \otimes K_1 + K_2 \otimes I_N, \quad
K_j = \tridiag\left(-1-\frac{w_{j}h}{2}, 2, -1+\frac{w_{j}h}{2}\right)
    \in \RS^{N \times N}, \quad j = 1,2.
$$
Here $\circ$ denotes the Hadamard product, $\otimes$ denotes the Kronecker product,
and $\tridiag(a,b,c)$ denotes the tridiagonal matrix with $a,b,c$ on its subdiagonal, 
diagonal, and superdiagonal, respectively.

Throughout this subsection, 
we assume the mesh P\'eclet condition $h\cdot\max\{w_1,w_2\} < 2$, under which $K$ is a nonsingular M-matrix. 
Its symmetric part is 
$S := (K + K^\top)/2 = I \otimes A + A \otimes I$ with $A = \tridiag(-1,2,-1)$.
Since $\|K^{-1}\|_2 = 1/\sigma_{\min}(K)$ and $\sigma_{\min}(K) \geq \lambda_{\min}(S)$,
where $\sigma_{\min}(K)$ denotes the smallest singular value of $K$, it follows that
\begin{equation}
  \label{cons:C_K}
  C_K := h^{2}\|K^{-1}\|_{2} 
  \leq \frac{h^{2}}{\lambda_{\min}(S)} 
  = \frac{h^{2}}{8\sin^{2}(\pi h/2)}
  \leq \frac{1}{8},
\end{equation}
where the last inequality follows from the fact that $\sin(\pi h/2) \geq h$ for $h \in (0,1]$.
Moreover, $h^{2}/\big(8\sin^{2}(\pi h/2)\big)$ decreases monotonically to $1/(2\pi^{2})$ as $h \to 0$.
Therefore, the quantity $C_{K}$ is bounded independently of the mesh size.

Multiplying \eqref{eq:CDR_discrete} by $K^{-1}$ yields the fixed-point function
\begin{equation}
  \label{eq:g(x)_CDR}
  \myvec{g}(\myvec{u})
  = K^{-1}\big(\lambda h^{2}\bm{\varphi}(\myvec{u}) + h^{2}\bm{\psi}\big).
\end{equation}
Since $\bm{\varphi}$ is continuously differentiable with
$\bm{\varphi}'(\myvec{u}) = (1+\nu)\diag\big(|\myvec{u}|^{\nu}\big)$, 
the Jacobian of $\myvec{g}$ is given by 
\begin{equation}
  \label{eq:g'(x)_CDR}
  \myvec{g}'(\myvec{u})
  = (1+\nu)\lambda h^{2}K^{-1}\diag\big(|\myvec{u}|^{\nu}\big).
\end{equation}
The associated nonlinear operator is
\begin{equation}
  \label{eq:f(x)=0_CDR}
  \myvec{f}(\myvec{u}) = \myvec{g}(\myvec{u}) - \myvec{u}
    = -K^{-1}\big[K\myvec{u} - \lambda h^{2}\bm{\varphi}(\myvec{u}) - h^{2}\bm{\psi}\big],
\end{equation}
whose Jacobian is $\myvec{f}'(\myvec{u}) = \myvec{g}'(\myvec{u}) - I_n$.
Since $K$ is nonsingular, 
solving \eqref{eq:CDR_discrete} is equivalent to finding a fixed point of $\myvec{g}$.
That is, $\myvec{u}^{*}$ is a solution of \eqref{eq:CDR_discrete} if and only if
$\myvec{g}(\myvec{u}^{*}) = \myvec{u}^{*}$,
or equivalently, $\myvec{f}(\myvec{u}^{*}) = \myvec{0}$.

\subsubsection{H\"older continuity and local contractivity}
\label{subsec:HolderContractivity}

In this subsection, we verify the assumptions of Theorem \ref{th:LocalConvAnderson(m)}
for the discrete problem \eqref{eq:CDR_discrete} and show,
in addition, that the H\"older exponent cannot be taken to be $1$.
We begin by recalling some elementary facts needed to establish
the H\"older continuity of the Jacobian of the nonlinear operator $\myvec{f}$ 
defined by \eqref{eq:f(x)=0_CDR}.
For $\nu \in (0,1]$ and $a,b \in \RS$, we have
$$
\big| |a|^{\nu} - |b|^{\nu} \big| \leq |a-b|^{\nu}.
$$
Moreover, for any $\myvec{v}\in \RS^n$, 
we have $\|\myvec{v}\|_\infty \leq \|\myvec{v}\|_2$ 
and $\|\diag(\myvec{v})\|_2 = \|\myvec{v}\|_\infty$.

For any $\myvec{u}, \myvec{v} \in \RS^n$, where $n = N^2$, 
it follows from \eqref{eq:g'(x)_CDR} and \eqref{eq:f(x)=0_CDR} that
$$
\myvec{f}'(\myvec{u}) - \myvec{f}'(\myvec{v})
  = \myvec{g}'(\myvec{u}) - \myvec{g}'(\myvec{v})
  = (1+\nu)\lambda h^{2}K^{-1}\diag\big(|\myvec{u}|^{\nu} - |\myvec{v}|^{\nu}\big).
$$
Then, we have
\begin{align*}
  \|\myvec{f}'(\myvec{u}) - \myvec{f}'(\myvec{v})\|_2
  &\leq (1+\nu)\lambda h^{2}\|K^{-1}\|_2\|\diag\big(|\myvec{u}|^{\nu} - |\myvec{v}|^{\nu}\big)\|_2 \\
  &= (1+\nu)\lambda h^{2}\|K^{-1}\|_2\||\myvec{u}|^{\nu} - |\myvec{v}|^{\nu}\|_\infty \\
  &\leq (1+\nu)\lambda h^{2}\|K^{-1}\|_2\|\myvec{u} - \myvec{v}\|_\infty^{\nu} \\
  &\leq (1+\nu)\lambda h^{2}\|K^{-1}\|_2\|\myvec{u} - \myvec{v}\|_2^{\nu}.
\end{align*}
Set $H_\nu := (1+\nu)\lambda h^{2}\|K^{-1}\|_2 = (1+\nu)\lambda C_K$.
We conclude that the Jacobian $\myvec{f}'$ is H\"older continuous with exponent $\nu \in (0,1]$ 
and H\"older constant $H_\nu$ with respect to the $\ell^2$ norm.
Moreover, \eqref{cons:C_K} gives $H_\nu \leq (1+\nu)\lambda/8$.
Hence, the H\"older constant $H_\nu$ is bounded independently of the mesh size.

We next consider a concrete example for which the local contractivity
of the fixed-point function $\myvec{g}$ defined by \eqref{eq:g(x)_CDR}
can be verified explicitly.
Let 
$$
u^*(x,y) = \sin(2\pi x)\sin(\pi y), \quad (x,y) \in \Omega.
$$
Set $\myvec{u}^* = (u^*(x_i, y_j))_{i,j=1}^N$.
Rather than sampling a continuous source term,
we prescribe $\bm{\psi}$ at the discrete level by
$\bm{\psi} = h^{-2}K\myvec{u}^* - \lambda\bm{\varphi}(\myvec{u}^*)$.
Then, $\myvec{u}^*$ is an exact solution of \eqref{eq:CDR_discrete}
and $\|\myvec{u}^*\|_\infty \leq 1$.
For any $\myvec{v} \in \ball(\myvec{u}^*,\rho)$ with $\rho > 0$, 
we have 
$$
\|\myvec{v}\|_\infty \leq \|\myvec{u}^*\|_\infty + \|\myvec{v} - \myvec{u}^*\|_\infty
  < 1 + \rho.
$$
It follows from \eqref{eq:g'(x)_CDR} that
$$
\|\myvec{g}'(\myvec{v})\|_2
  \leq (1+\nu)\lambda h^{2}\|K^{-1}\|_2\|\myvec{v}\|_\infty^{\nu}
  < (1+\nu)\lambda h^{2}\|K^{-1}\|_2(1 + \rho)^{\nu}
  = H_\nu (1 + \rho)^{\nu} =: \theta_\rho.
$$
Therefore, if $H_\nu < 1$, 
then $\myvec{g}$ is contractive in $\ball(\myvec{u}^*,\rho)$ 
for any $\rho < H_\nu^{-1/\nu} - 1$, for which $\theta_\rho < 1$.
To satisfy in addition the condition $\theta_\rho M_\alpha \leq 1$ 
required in Theorem \ref{th:LocalConvAnderson(m)},
it suffices to choose $\rho$ such that $0 < \rho < (H_\nu M_\alpha)^{-1/\nu} - 1$
provided that $H_\nu M_\alpha < 1$.
In particular, we have $\|\myvec{g}'(\myvec{u}^*)\|_2 \leq H_\nu < 1$.
Since $\myvec{f}'(\myvec{u}^*) = \myvec{g}'(\myvec{u}^*) - I$,
the Banach lemma implies that $\myvec{f}'(\myvec{u}^*)$ is nonsingular and 
$$
\|\myvec{f}'(\myvec{u}^*)^{-1}\|_2 \leq \frac{1}{1 - H_\nu}.
$$
Moreover, 
$\|\myvec{f}'(\myvec{u}^*)\|_2 = \|I - \myvec{g}'(\myvec{u}^*)\|_2 \leq 1 + H_\nu$.
Consequently, we have
$$
\kappa_2(\myvec{f}'(\myvec{u}^*)) 
  = \|\myvec{f}'(\myvec{u}^*)\|_2\|\myvec{f}'(\myvec{u}^*)^{-1}\|_2
  \leq \frac{1 + H_\nu}{1 - H_\nu}.
$$
Substituting these bounds into \eqref{cons:r_nu} 
and into the definition of $\tilde{r}_\nu$ in \eqref{cons:r_nu_tilde} yields
$$
r_\nu \geq \left(\frac{1-H_\nu}{H_\nu}\right)^{1/\nu} \quad \text{and} \quad
\tilde{r}_\nu \geq \left(\frac{(1-\theta_\rho)\nu}{M_\alpha(1+M_{\alpha}^{\nu})}\right)^{1/\nu}
  \frac{\nu(1-H_\nu)}{(2+\nu)(1+H_\nu)}\left(\frac{1-H_\nu}{H_\nu}\right)^{1/\nu}.
$$
This means that the assumptions (ii) and (iii) of Theorem \ref{th:LocalConvAnderson(m)} 
are satisfied in $\ball(\myvec{u}^*,\hat{r}_\nu)$ with $\hat{r}_\nu = \min\{r_\nu,\rho\}$,
and by Remark \ref{remark:ConvergenceRadius}, every initial guess $\myvec{u}_0 \in \ball(\myvec{u}^*,r)$
with $r = \min\{\hat{r}_\nu/M_\alpha,\tilde{r}_\nu\}$ satisfies $\zeta < 1 - \tau$.
Since $H_\nu \leq (1+\nu)\lambda/8 =: \bar{H}_\nu$,
if $\bar{H}_\nu M_\alpha < 1$, then $\rho > 0$ can be chosen independently of the mesh size
such that $\bar{H}_\nu M_\alpha (1 + \rho)^\nu < 1$. 
This implies that $\theta_\rho$ is uniformly bounded away from $1$,
and $r_\nu$ and $\tilde{r}_\nu$ admit positive lower bounds independent of the mesh size.
Therefore, the radius $r$ is uniformly bounded away from zero 
and the convergence ball does not degenerate under mesh refinement.

In Subsection \ref{subsec:NARE},
the minimal positive solution is not available in closed form, 
and the constants in Corollary \ref{cor:LocalConvLipCondAnderson(m)} therefore cannot be evaluated.
In the present setting, by contrast, 
the assumptions (ii) and (iii) of Theorem \ref{th:LocalConvAnderson(m)} can be verified explicitly,
and the resulting radius admits a positive lower bound independent of the mesh size.

We proceed to show that the H\"older estimate established above is sharp.
The following proposition asserts that, 
for $\nu \in (0,1)$, the Jacobian $\myvec{f}'$ is not Lipschitz continuous
in any neighborhood of the solution $\myvec{u}^*$.
The H\"older exponent in Theorem \ref{th:LocalConvAnderson(m)} therefore cannot be taken to be $1$,
and the classical Lipschitz setting is excluded in this example.

\begin{myproposition}
  \label{prop:CDR_noLipschitz}
  Let $\nu \in (0,1)$ and $N$ be odd.
  Then, for any $\delta > 0$, 
  the Jacobian of the nonlinear operator $\myvec{f}$ defined by \eqref{eq:f(x)=0_CDR} 
  is not Lipschitz continuous in $\ball(\myvec{u}^*, \delta)$.
\end{myproposition}

\begin{proof}
  Set $i_0 = (N + 1)/2$. 
  The assumption that $N$ is odd gives $x_{i_0} = 1/2$. Then
  $$
  u^*(x_{i_0},y_j) = \sin(\pi)\sin(\pi y_j) = 0, \quad j = 1,2,\ldots,N.
  $$
  Since $0 < y_j < 1$, we have $\sin(\pi y_j) > 0$.
  This implies that $u^*(x_i,y_j) = 0$ for $1 \leq i,j \leq N$ if and only if $i = i_0$.
  Thus, the solution $\myvec{u}^*$ has exactly $N$ zero components.
  Let $\ell_0$ be the index of any such zero component.
  Given $\delta > 0$, we define
  $$
  \myvec{u}_t := \myvec{u}^* + t\myvec{e}_{\ell_0}, \quad 0 < t < \delta.
  $$
  Clearly, $\myvec{u}_t \in \ball(\myvec{u}^*, \delta)$.
  Noting that the $\ell_0$th component of $\myvec{u}^*$ is zero, 
  we know that the matrices $\diag(|\myvec{u}_t|^{\nu})$ 
  and $\diag(|\myvec{u}^*|^{\nu})$ differ only in the $\ell_0$th diagonal entry.
  More precisely,
  $$
  \diag(|\myvec{u}_t|^{\nu}) - \diag(|\myvec{u}^*|^{\nu})
  = \diag(t^{\nu}\myvec{e}_{\ell_0})
  = t^{\nu}\myvec{e}_{\ell_0}\myvec{e}_{\ell_0}^{\top}.
  $$
  It follows that 
  $$
  \myvec{f}'(\myvec{u}_t) - \myvec{f}'(\myvec{u}^*)
  = (1+\nu)\lambda h^{2}K^{-1}\diag(|\myvec{u}_t|^{\nu} - |\myvec{u}^*|^{\nu})
  = (1+\nu)\lambda h^{2}t^{\nu}K^{-1}\myvec{e}_{\ell_0}\myvec{e}_{\ell_0}^{\top}.
  $$
  Then, we have
  $$
  \|\myvec{f}'(\myvec{u}_t) - \myvec{f}'(\myvec{u}^*)\|_2
  = (1+\nu)\lambda h^{2}t^{\nu}\|(K^{-1}\myvec{e}_{\ell_0})\myvec{e}_{\ell_0}^{\top}\|_2
  = (1+\nu)\lambda h^{2}t^{\nu}\|K^{-1}\myvec{e}_{\ell_0}\|_2,
  $$
  which allows us to get
  $$
  \frac{\|\myvec{f}'(\myvec{u}_t) - \myvec{f}'(\myvec{u}^*)\|_2}{\|\myvec{u}_t - \myvec{u}^*\|_2}
  = (1+\nu)\lambda h^{2}t^{\nu}\frac{\|K^{-1}\myvec{e}_{\ell_0}\|_2}{t}
  = (1+\nu)\lambda h^{2}t^{\nu-1}\|K^{-1}\myvec{e}_{\ell_0}\|_2.
  $$
  Write $K = (k_{ij})$. Since $K$ is a nonsingular M-matrix, 
  we have $k_{ii} > 0$, $k_{ij} \leq 0$ for $i \neq j$, and $K^{-1} \geq 0$.
  See \cite[Chapter 6]{BermanP1994} for details.
  Taking the $\ell_0$th diagonal entry of $KK^{-1} = I$ yields
  $$
  1 = k_{\ell_0\ell_0}(K^{-1})_{\ell_0\ell_0} + \sum_{j \neq \ell_0} k_{\ell_0 j}(K^{-1})_{j\ell_0}
  \leq k_{\ell_0\ell_0}(K^{-1})_{\ell_0\ell_0} = 4(K^{-1})_{\ell_0\ell_0},
  $$
  and so $\|K^{-1}\myvec{e}_{\ell_0}\|_2 \geq (K^{-1})_{\ell_0\ell_0} \geq 1/4$.
  This leads to
  $$
  \frac{\|\myvec{f}'(\myvec{u}_t) - \myvec{f}'(\myvec{u}^*)\|_2}{\|\myvec{u}_t - \myvec{u}^*\|_2}
  = (1+\nu)\lambda h^{2}\|K^{-1}\myvec{e}_{\ell_0}\|_2 t^{\nu-1}
  \geq \frac{(1+\nu)\lambda h^{2}}{4} t^{\nu-1}.
  $$
  Since $0 < \nu < 1$, the right-hand side tends to infinity as $t \to 0^+$.
  Therefore, there exists no finite Lipschitz constant for $\myvec{f}'$ in $\ball(\myvec{u}^*,\delta)$.
  The arbitrariness of $\delta > 0$ completes the proof.
\end{proof}

\begin{myremark}
  The assumption that $N$ is odd in the above proposition is made only to render the conclusion exact.
  Regardless of the parity of $N$, however, 
  no Lipschitz estimate with a mesh-independent constant is available.
  Indeed, since $h = 1/(N+1)$, 
  when $N$ is odd, the line $x = 1/2$ contains interior grid points.
  Thus, the solution $\myvec{u}^*$ has zero components at these grid points.
  The loss of Lipschitz continuity of the Jacobian $\myvec{f}'$ 
  is due to the scalar function $t \mapsto |t|^\nu$,
  which is H\"older continuous with exponent $\nu \in (0,1)$ on $\RS$
  but not Lipschitz continuous at $t = 0$.
  When $N$ is even, all components of the solution $\myvec{u}^*$ are nonzero.
  In this case, 
  the Jacobian $\myvec{f}'$ is Lipschitz continuous in a sufficiently small neighborhood of $\myvec{u}^*$.
  However, a bound for the Lipschitz constant obtained from the derivative of 
  the above scalar function contains the factor
  $\big(\min\limits_{1\leq i \leq n} |u^*_i|\big)^{\nu-1}$, 
  which deteriorates under mesh refinement because $\min{|u^*_i|} \to 0$.
  Thus, the parity assumption distinguishes between exact failure 
  and asymptotic deterioration of Lipschitz continuity.
  By contrast, the H\"older estimate remains valid for every $N$,
  with a H\"older constant $H_\nu = (1+\nu)\lambda C_K$
  that is uniformly bounded independently of the mesh size.
  Figure \ref{fig:CDR} (b) provides a concrete illustration of this observation.
\end{myremark}

\subsubsection{Numerical illustration}

In this subsection, we provide numerical results to verify 
the assumptions and conclusion of Theorem \ref{th:LocalConvAnderson(m)}
on the discrete problem \eqref{eq:CDR_discrete}
for which the quantities involved in the theorem can be evaluated explicitly.
Computational efficiency is not the focus here,
since the acceleration effect has already been demonstrated in Subsection \ref{subsec:NARE}.

We set the parameters in the problem \eqref{eq:CDR} to be $w_1 = w_2 = 1$ and $\lambda = 1.5$.
The initial guess is chosen to be 
$\myvec{u}_0 = \myvec{u}^* + \delta\myvec{d}/\|\myvec{d}\|_2$, 
where $\myvec{d} \in \RS^n$, with $n = N^2$, is a fixed random vector 
and $\delta > 0$ is a small perturbation parameter.
Then $\|\myvec{u}_0 - \myvec{u}^*\|_2 = \delta$ exactly and independently of the mesh size.
The iteration is terminated when the residual norm satisfies
$\|\myvec{f}(\myvec{u}_k)\|_2 \leq 10^{-13}\|\myvec{u}^*\|_2$,
and the observed R-factor, denoted by $R_{\text{obs}}$, is computed as
$(\|\myvec{f}(\myvec{u}_k)\|_2 / \|\myvec{f}(\myvec{u}_0)\|_2)^{1/k}$
at termination.
In addition, the linear systems involving $K$ are solved directly
by a two-dimensional discrete sine transform with complexity $\mathcal{O}(n\log n)$ \cite{GolubV1996}.

Table \ref{tab:CDR_data} lists the quantities involved in 
Theorem \ref{th:LocalConvAnderson(m)} and Remark \ref{remark:ConvergenceRadius} 
for $\nu = 0.5$ on four grids,
where $M_\alpha = 1.15$ is a representative value of the runs reported below.
Since $\tilde{r}_\nu$ decreases with $\rho$ through $\theta_\rho$,
whereas $\rho/M_\alpha$ increases with $\rho$,
the largest guaranteed radius is obtained at the unique value of $\rho$ 
satisfying $\rho = M_\alpha\tilde{r}_\nu(\rho)$, which we compute by bisection. 
The values of $\rho$ reported in Table \ref{tab:CDR_data} are obtained in this way.
We note that the quantities $C_K$, $H_\nu$, $\theta_\rho$, $r_\nu$, $\rho$ and $\tilde{r}_\nu$
remain essentially constant as $n$ increases by a factor of $64$.
This numerical evidence confirms the mesh independence established in \eqref{cons:C_K}
and in the discussion of Subsection \ref{subsec:HolderContractivity}.
In particular, the guaranteed convergence ball does not degenerate under mesh refinement.
We also observe that $\rho$ and $\tilde{r}_\nu$ are both more than 
two orders of magnitude smaller than $r_\nu$ and in fact satisfy $\rho = M_\alpha\tilde{r}_\nu$.
This means that the condition $\zeta < 1 - \tau$, 
as reflected in $\tilde{r}_\nu$,
provides the dominant restriction on the guaranteed convergence radius,
whereas the H\"older neighborhood $r_\nu$ required in assumption (ii) is much less restrictive.

\begin{table}[h]
\centering
\caption{Theoretical constants of Theorem~3.1 ($\nu=0.5$, $M_\alpha = 1.15$).}
\label{tab:CDR_data}
\begin{tabular}{cccccccc}
\toprule
$N$ & $C_K$ & $H_\nu$ & $\theta_\rho$ & $r_\nu$ & $\hat{r}_\nu(=\rho)$ & $\tilde{r}_\nu$ & $r$ \\
\midrule
31  & 0.050146 & 0.1128 & 0.1323 & 61.8259 & 0.3757 & 0.3267 & 0.3267 \\
63  & 0.050113 & 0.1128 & 0.1323 & 61.9176 & 0.3763 & 0.3272 & 0.3272 \\
127 & 0.050105 & 0.1127 & 0.1323 & 61.9406 & 0.3765 & 0.3274 & 0.3274 \\
255 & 0.050103 & 0.1127 & 0.1323 & 61.9463 & 0.3765 & 0.3274 & 0.3274 \\
\bottomrule
\end{tabular}
\end{table}

Table \ref{tab:CDR_AAdepth} reports the numerical results for varying Anderson depths
$m = 0,1,3,5$ and perturbation parameters $\delta = 10^{-1}, 10^{-2}$ on the grid with $N = 127$.
In both cases, $\delta < r$, so that $\myvec{u}_0 \in \ball(\myvec{u}^*,r)$.
Assumptions (ii) and (iii) of Theorem \ref{th:LocalConvAnderson(m)}
have been verified in Subsection \ref{subsec:HolderContractivity}.
The remaining two conditions are also met throughout 
Tables \ref{tab:CDR_AAdepth}, \ref{tab:CDR_HolderExp} and \ref{tab:CDR_Lipschitz},
where $\theta_\rho \leq 0.2271$ and $M_\alpha \leq 1.15$, so that $\theta_\rho M_\alpha \leq 0.2612 < 1$.
Moreover, the condition $\zeta < 1 - \tau$ holds in each case, 
which confirms that the factor $q$ in Theorem \ref{th:LocalConvAnderson(m)} 
provides a valid upper bound for the observed R-factor $R_{\text{obs}}$.
As expected, the observed R-factor remains below the theoretical convergence factor $q$ throughout.
Anderson acceleration reduces the observed R-factor by approximately one half
at depth $m = 1$ and to approximately one third at depth $m = 3$.
Further increases in the depth $m$ yield only limited improvements in the observed R-factor.
The values of $\bar\eta$ are close to $1$ in all cases,
and increasing the depth $m$ has only a minor effect on $\bar\eta$.
If we take $\bar\eta$ to be the median of $\eta_k$, 
following the approach used in \cite{PollockRX2019}, 
the resulting values are concentrated around $0.5$ and decrease as $m$ increases.
However, the corresponding values of $q$ vary only slightly.
This suggests that, for this example, the influence of $\eta_k$ on $q$ is rather limited.
The increase in $M_\alpha$ is also modest, 
from $1$ to about $1.11$ as $m$ increases from $0$ to $5$.

Two additional observations in Table \ref{tab:CDR_AAdepth} are worth noting.
First, although the bound is valid, it is rather conservative, 
with the observed R-factor $R_{\text{obs}}$ being one order of magnitude smaller 
than the theoretical convergence factor $q$.
This is due in part to the factor
$$
M_\alpha(1 + M_\alpha^\nu) \kappa_2^\nu(\myvec{f}'(\myvec{u}^*))\|\myvec{f}'(\myvec{u}^*)^{-1}\|_2
$$
appearing in \eqref{cons:zeta}.
Second, the theoretical convergence factor $q$ becomes less favorable as $m$ increases,
whereas the observed R-factor $R_{\text{obs}}$ improves.
A similar behavior occurs for the estimate $\theta^{1/(m+1)}$ in \cite{ChenK2019},
which approaches $1$ as $m$ increases.
Obtaining a sharper theoretical estimate that better reflects the observed dependence 
of the R-factor on the Anderson depth remains an open question.

\begin{table}[h]
\centering
\caption{Convergence factors for different Anderson depths ($N=127$, $\nu=0.5$).}
\label{tab:CDR_AAdepth}
\begin{tabular}{cccccccccc}
\toprule
$\delta$ & $m$ & $R_{\text{obs}}$ & $\bar\eta$ & $M_\alpha$ & $\theta_\rho$ & $r$ & $\zeta$ & $1-\tau$ & $q$ \\
\midrule
$10^{-1}$ & 0 & 0.0327 & 1.0000 & 1.0000 & 0.1362 & 0.4606 & 0.4025 & 0.8638 & 0.5387 \\
          & 1 & 0.0168 & 0.9942 & 1.0900 & 0.1337 & 0.3733 & 0.4484 & 0.8671 & 0.7394 \\
          & 3 & 0.0102 & 0.9942 & 1.1130 & 0.1331 & 0.3547 & 0.4603 & 0.8676 & 0.8589 \\
          & 5 & 0.0098 & 0.9942 & 1.1130 & 0.1331 & 0.3547 & 0.4603 & 0.8676 & 0.9022 \\
\addlinespace
$10^{-2}$ & 0 & 0.0282 & 1.0000 & 1.0000 & 0.1362 & 0.4606 & 0.1273 & 0.8638 & 0.2635 \\
          & 1 & 0.0141 & 0.9941 & 1.0893 & 0.1337 & 0.3739 & 0.1417 & 0.8670 & 0.4487 \\
          & 3 & 0.0100 & 0.9941 & 1.1133 & 0.1331 & 0.3545 & 0.1456 & 0.8676 & 0.6537 \\
          & 5 & 0.0100 & 0.9941 & 1.1133 & 0.1331 & 0.3545 & 0.1456 & 0.8676 & 0.7492 \\
\bottomrule
\end{tabular}
\end{table}

Figure \ref{fig:CDR}(a) displays the residual histories corresponding to the first block of Table \ref{tab:CDR_AAdepth}. 
All four curves decrease monotonically, 
and the separation between the fixed-point iteration and the accelerated iterations 
becomes increasingly apparent after the first few steps.
The curve for $m=1$ lies between the fixed-point curve and those for $m=3,5$, 
while the latter two remain close
and are nearly indistinguishable in the last few iterations.
This behavior is in agreement with the observed R-factors $R_{\rm obs}$ reported in Table \ref{tab:CDR_AAdepth}.
Figure \ref{fig:CDR}(b) illustrates Proposition \ref{prop:CDR_noLipschitz}
by plotting the quotient 
$\|\myvec{f}'(\myvec{u}^* + t\myvec{e}_{\ell_0}) - \myvec{f}'(\myvec{u}^*)\|/t^p$
along the direction $\myvec{e}_{\ell_0}$ corresponding to a vanishing component of the solution $\myvec{u}^*$.
The plotted values follow from the closed-form expression obtained in the proof of Proposition \ref{prop:CDR_noLipschitz}. 
For $p=1$ and $\nu=1$, the quotient is independent of $t$,
consistent with the Lipschitz continuity established in Subsection \ref{subsec:HolderContractivity}. 
For $p=1$ and $0 < \nu <1$, the quotient grows like $t^{\nu-1}$
and therefore becomes unbounded as $t\to 0^{+}$, with a steeper growth as $\nu$ decreases.
This shows that no finite Lipschitz constant exists in any neighborhood of $\myvec{u}^*$.
In contrast, the black curve corresponding to $p = \nu = 0.5$ shows the H\"older quotient,
which is independent of $t$ and bounded by the H\"older constant $H_\nu$ established in Subsection \ref{subsec:HolderContractivity}.

\begin{figure}
  \centering
  \subfigure{\includegraphics[width=0.48\textwidth]{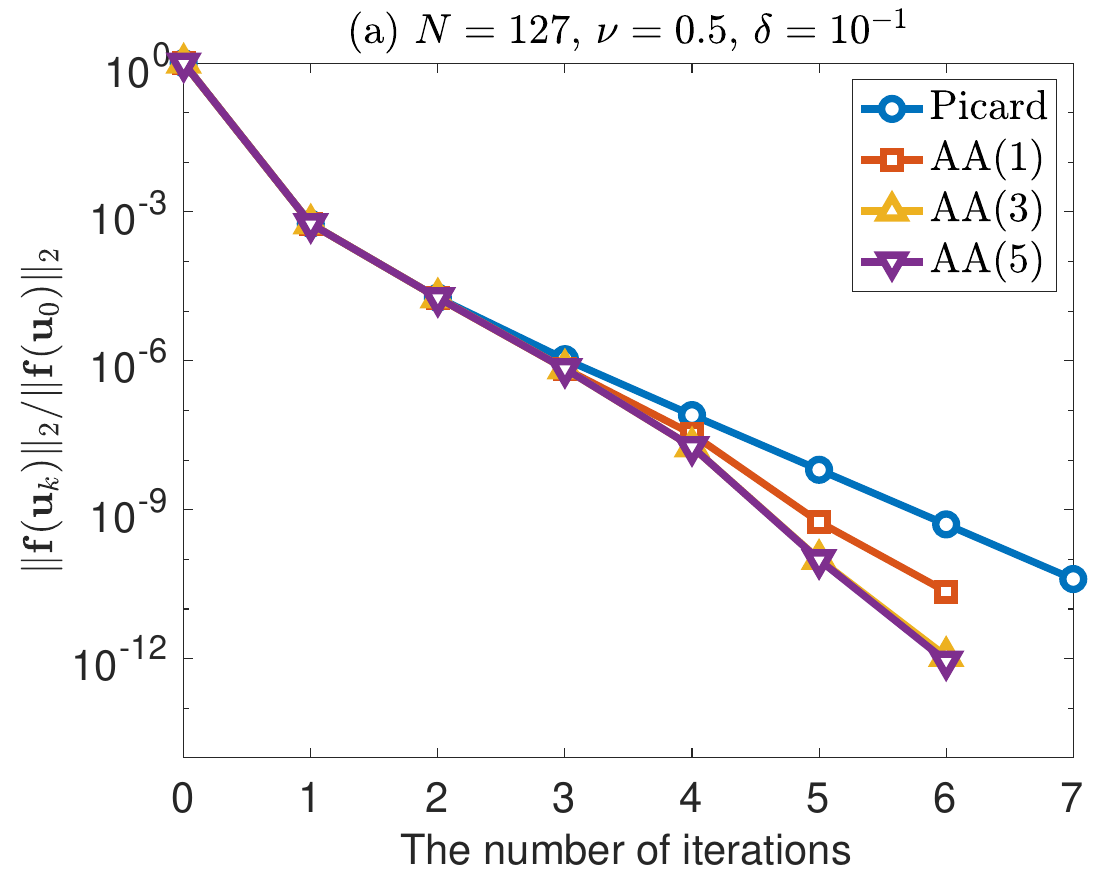}}\quad
  \subfigure{\includegraphics[width=0.48\textwidth]{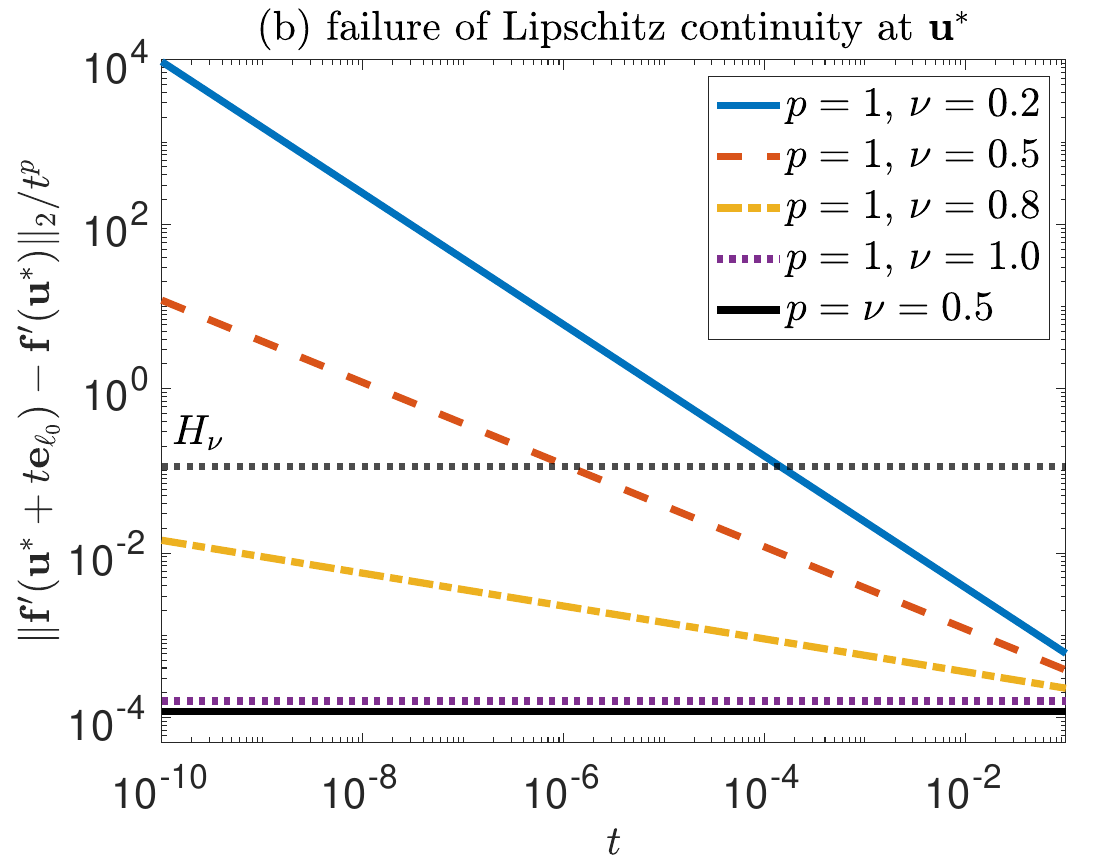}} \\
  \caption{(a) Convergence histories for $N=127$, $\nu=0.5$, $\lambda=1.5$
and $\delta=10^{-1}$. (b) The difference quotient of $\myvec{f}'$ along
$\myvec{e}_{\ell_0}$ from
Proposition 4.1. For $p=1$ and $\nu<1$ it grows like $t^{\nu-1}$ 
 and no finite Lipschitz constant exists in any neighborhood of $\myvec{u}^{*}$; 
 for $p=1$ and $\nu=1$ it is constant.
 The case $p=\nu=0.5$ gives the corresponding H\"older quotient, which is bounded by the H\"older constant $H_\nu$.}
  \label{fig:CDR}
\end{figure}

Table \ref{tab:CDR_HolderExp} examines the effect of varying the H\"older exponent $\nu$ 
with fixed Anderson depth $m = 3$ and perturbation parameter $\delta = 10^{-2}$.
The observed R-factor $R_{\text{obs}}$ is remarkably insensitive to $\nu$.
In contrast, the theoretical convergence factor $q$ becomes substantially less favorable as $\nu$ decreases.
The guaranteed convergence radius $r$ decreases from $0.4393$ at $\nu = 0.8$ to $0.0252$ at $\nu = 0.2$,
while $\zeta$ increases from $0.0394$ to $0.7557$
and the theoretical convergence factor $q$ increases from $0.5005$ to $0.9557$.
Consequently, the condition $\zeta < 1 - \tau$ is nearly violated at $\nu = 0.2$.
This deterioration can be explained directly from the definition of $\zeta$ given by \eqref{cons:zeta}.
At $\nu = 0.2$, the factor $\nu^{-(1+\nu)}$ is approximately $6.9$,
while $\delta^{\nu}$ increases from $10^{-2}$ at $\nu = 1$ to approximately $0.40$ at $\nu = 0.2$.
Thus, Theorem \ref{th:LocalConvAnderson(m)} predicts that 
both the guaranteed convergence radius and the theoretical convergence factor
become less favorable as the H\"older exponent $\nu$ decreases,
even though the observed convergence remains essentially unchanged.
This observation suggests that the theoretical convergence factor $q$
may be overly pessimistic for small H\"older exponents $\nu$.
Nevertheless, for $\nu < 1$, the existing Lipschitz-based convergence results 
in \cite{TothK2015,ChenK2019} are not applicable,
since Proposition \ref{prop:CDR_noLipschitz} shows that
the required Lipschitz continuity fails in any neighborhood of the solution $\myvec{u}^*$.

\begin{table}[h]
\centering
\caption{Convergence factors for different H\"older exponents
($N=127$, $m=3$, $\delta=10^{-2}$).}
\label{tab:CDR_HolderExp}
\begin{tabular}{ccccccccccc}
\toprule
$\nu$ & FP & AA & $R_{\mathrm{obs}}$ & $H_\nu$ & $M_\alpha$ & $\theta_\rho$ & $r$ & $\zeta$ & $1-\tau$ & $q$ \\
\midrule
0.2 & 7 & 5 & 0.0091 & 0.0902 & 1.1294 & 0.0907 & 0.0252  & 0.7557 & 0.9099 & 0.9557 \\
0.4 & 6 & 5 & 0.0097 & 0.1052 & 1.1234 & 0.1167 & 0.2631  & 0.2388 & 0.8840 & 0.7300 \\
0.6 & 6 & 5 & 0.0103 & 0.1203 & 1.1055 & 0.1506 & 0.4115  & 0.0913 & 0.8502 & 0.5913 \\
0.8 & 6 & 5 & 0.0110 & 0.1353 & 1.1283 & 0.1867 & 0.4393  & 0.0394 & 0.8142 & 0.5005 \\
1.0 & 6 & 5 & 0.0115 & 0.1503 & 1.1403 & 0.2260 & 0.4414  & 0.0175 & 0.7750 & 0.4362 \\
\bottomrule
\end{tabular}
\end{table}

For $\nu = 1$, the Jacobian of the nonlinear operator $\myvec{f}$ 
defined by \eqref{eq:f(x)=0_CDR} is Lipschitz continuous.
In this case, the local convergence results of \cite{TothK2015,ChenK2019} are applicable.
For $N = 127$ and $\delta = 10^{-2}$,
Table \ref{tab:CDR_Lipschitz} compares the convergence factor $q$
obtained from Corollary \ref{cor:LocalConvLipCondAnderson(m)} 
with the factor $\theta_\rho^{1/(m+1)}$ from \cite{ChenK2019} 
and the admissible R-factor interval $(\theta_\rho,1)$ from \cite{TothK2015}.
We have in this case $H_\nu = H_1 = 0.1503$.
For every depth considered, the factor $q$ is smaller than $\theta_\rho^{1/(m+1)}$.
This reduction is approximately $40\%$ at depth $m = 1$ and $37\%$ at depth $m = 3$,
and remains noticeable for larger $m$.
In contrast, the result in \cite{TothK2015} only guarantees
an R-factor in the depth-independent interval $(\theta_\rho,1)$,
whereas Corollary \ref{cor:LocalConvLipCondAnderson(m)} provides
an explicit factor $q$ that incorporates the effects of the depth $m$,
the optimization gain $\eta_k$ and the condition number $\kappa_2(\myvec{f}'(\myvec{u}^*))$.
The observed factors $R_{\text{obs}}$ are substantially smaller than
the corresponding theoretical factors $q$,
indicating that the theoretical estimates remain conservative.
These results provide a quantitative illustration of the refinement stated in Remark \ref{remark:ConvergenceFactor_Lip}.

\begin{table}[h]
\centering
\caption{Comparison of convergence factors ($\nu=1$, $N=127$, $\delta=10^{-2}$).}
\label{tab:CDR_Lipschitz}
\begin{tabular}{cccccccc c}
\toprule
$m$ & $\bar\eta$ & $M_\alpha$ & $\zeta$ & $1-\tau$ & $q$ & $\theta_\rho^{1/(m+1)}$ & $(\theta_\rho,1)$ & $R_{\text{obs}}$ \\
\midrule
1 & 0.9958 & 1.1068 & 0.0168 & 0.7739 & 0.2849 & 0.4765 & $(0.2271,1)$ & 0.0178 \\
3 & 0.9958 & 1.1403 & 0.0175 & 0.7750 & 0.4362 & 0.6895 & $(0.2260,1)$ & 0.0115 \\
5 & 0.9958 & 1.1413 & 0.0176 & 0.7750 & 0.5559 & 0.7804 & $(0.2259,1)$ & 0.0114 \\
8 & 0.9958 & 1.1413 & 0.0176 & 0.7750 & 0.6680 & 0.8477 & $(0.2259,1)$ & 0.0114 \\
10 & 0.9958 & 1.1413 & 0.0176 & 0.7750 & 0.7166 & 0.8735 & $(0.2259,1)$ & 0.0114 \\
\bottomrule
\end{tabular}
\end{table}

Theorem \ref{th:LocalConvAnderson(m=1)} provides a refined residual bound for Anderson acceleration with depth $m=1$.
We conclude this subsection by verifying this bound numerically on the discrete problem \eqref{eq:CDR_discrete}.
Write $s_k: = \sqrt{1-\eta_k^2}$ and set
\[
c_k^{(1)} := \theta_\rho\left(1 + \frac{1+\theta_\rho}{1-\theta_\rho}s_k\right),\quad
c_k^{(2)} := \frac{H_\nu}{1+\nu}\left[\left(1+\frac{\theta_\rho}{1-\theta_\rho}s_k\right)^{1+\nu}
 + \left(\frac{s_k}{1-\theta_\rho}\right)^{1+\nu}\frac{1}{|\alpha_k|^{\nu}}\right].
\]
The bound \eqref{ineq:norm_f(xk+1)_m=1} then reads
$\|\myvec{f}(\myvec{u}_{k+1})\|_2 
  \leq c_k^{(1)}\|\myvec{f}(\myvec{u}_k)\|_2 + c_k^{(2)}\|\myvec{f}(\myvec{u}_k)\|_2^{1+\nu}$.
We compare the predicted one-step factor
$\hat\varrho_k := c_k^{(1)} + c_k^{(2)}\|\myvec{f}(\myvec{u}_k)\|_2^{\nu}$
with the observed one
$\varrho_k := \|\myvec{f}(\myvec{u}_{k+1})\|_2/\|\myvec{f}(\myvec{u}_k)\|_2$.
Table \ref{tab:CDR_m=1} reports these quantities for $N=127$, $\nu=0.2$ 
and the perturbation parameter $\delta = 10^{-2}$ used in Table \ref{tab:CDR_AAdepth}. 
The contraction factor $\theta_\rho$ and the H\"older constant $H_\nu$ 
are those of the corresponding run.

The bound is satisfied at every reported iteration and $\hat{\varrho}_k < 0.19$ throughout,
indicating that \eqref{ineq:norm_f(xk+1)_m=1} provides a nontrivial one-step contraction estimate.
The predicted factors $\hat\varrho_k$ are nevertheless larger than the observed factors $\varrho_k$,
showing that the bound is quantitatively conservative.
As the residual decreases, 
the higher-order contribution $c_k^{(2)}\|\mathbf f(\mathbf u_k)\|_2^{\nu}$ 
decreases steadily from $1.0668\times 10^{-2}$ to $2.7345\times 10^{-3}$,
and becomes increasingly small relative to the linear contribution $c_k^{(1)}$.
Thus, the predicted factor is increasingly dominated by the linear term.
For fixed $\theta_\rho$, $c_k^{(1)}$ increases as $\eta_k$ decreases,
whereas the observed factors $\varrho_k$ show no corresponding monotone behavior.
This suggests that the gap between the predicted and observed factors is mainly associated with the linear coefficient.
Although $c_k^{(2)}$ contains the term $|\alpha_k|^{-\nu}$,
it occurs only in combination with $s_k^{1+\nu}$.
By \eqref{eq:norm_f(xk+1)-f(xk)}, we have
$$
s_k^{1+\nu}|\alpha_k|^{-\nu}
  = |\alpha_k|
  \left(\frac{\|\myvec{f}(\myvec{u}_k)-\myvec{f}(\myvec{u}_{k-1})\|_2}{\|\myvec{f}(\myvec{u}_k)\|_2}\right)^{1+\nu}.
$$
Thus, the inverse power of $|\alpha_k|$ does not by itself cause 
the higher-order contribution to become large
when $|\alpha_k|$ is small, as illustrated by the first row of Table \ref{tab:CDR_m=1}.

\begin{table}[h]
\centering
\caption{The residual bound \eqref{ineq:norm_f(xk+1)_m=1} for the depth $m=1$
($N=127$, $\nu=0.2$, $\delta=10^{-2}$).}
\label{tab:CDR_m=1}
\begin{tabular}{cccccccc}
\toprule
$k$ & $\|\myvec{f}(\myvec{u}_k)\|_2$ & $\eta_k$ & $|\alpha_k|$
 & $c_k^{(1)}$ & $c_k^{(2)}\|\myvec{f}(\myvec{u}_k)\|_2^{\nu}$ & $\hat\varrho_k$ & $\varrho_k$ \\
\midrule
1 & 5.5429e-06 & 0.9932 & 6.4385e-05 & 0.1035 & 1.0668e-02 & 0.1141 & 3.8128e-02 \\
2 & 2.1134e-07 & 0.6741 & 2.8990e-02 & 0.1713 & 9.2840e-03 & 0.1806 & 4.5335e-02 \\
3 & 9.5812e-09 & 0.5605 & 3.9021e-02 & 0.1810 & 5.2604e-03 & 0.1863 & 2.3919e-02 \\
4 & 2.2918e-10 & 0.5413 & 2.0526e-02 & 0.1824 & 2.7345e-03 & 0.1852 & 1.7274e-02 \\
\bottomrule
\end{tabular}
\end{table}

\section{Conclusions}
\label{sec:Conclusions}

In this paper, we have presented a new local convergence analysis
for Anderson acceleration applied to nonlinear equations under the assumptions that 
the first derivative of the nonlinear operator is H\"{o}lder continuous,
and that the associated fixed-point operator is contractive.
The main results are encapsulated in Theorem \ref{th:LocalConvAnderson(m)},
which shows that Anderson acceleration is R-linearly convergent.
This convergence factor is explicitly characterized by the unique root of equation \eqref{eq:q}.
In particular, when H\"{o}lder continuity reduces to Lipschitz continuity,
the resulting convergence factor is not determined only by the contraction factor $\theta$.
It also depends on the optimization gain $\eta_k$ defined in \eqref{cons:eta_k}
and the condition number $\kappa(\myvec{f}'(\myvec{x}^*))$.
Consequently, our analysis yields a more refined characterization of convergence behavior
compared to previous studies that rely exclusively on the contraction factor $\theta$.
Moreover, we obtained a refined residual bound of Anderson acceleration
for depth $m=1$.
We further demonstrated the applicability and efficiency of Anderson acceleration 
by computing an approximation to the minimal positive solution of 
the special nonlinear equation \eqref{eq:f(u,v)=0},
which is derived from NARE \eqref{eq:NARE} arising from neutron transport theory.
The numerical results confirm that Anderson acceleration
performs efficiently in both regular and nearly singular cases,
especially for large-scale problems.
We also presented a semilinear convection-diffusion-reaction equation 
whose discrete Jacobian is H\"older continuous with a mesh-independent H\"older constant,
but is not Lipschitz continuous in any neighborhood of the solution.
For this problem, the hypotheses of Theorem \ref{th:LocalConvAnderson(m)}
can be verified explicitly, and the guaranteed convergence radius does not degenerate
under mesh refinement, whereas the existing Lipschitz-based results are not applicable.
The numerical results show that, in the Lipschitz case, 
the theoretical convergence factor $q$ obtained from our analysis
is smaller than the corresponding factor given by existing results.
One goal of our future work is to explore whether Anderson acceleration
can be used to improve the convergence of Newton's method
for special singular nonlinear equations.

\section*{Acknowledgments}

This work was supported by the Fujian Province Natural Science Foundation of China (Grant No. 2022J01896), 
the Education Research Projects for Young Teachers of Fujian Provincial Education Department (Grant No. JAT220197), 
Fujian Key Laboratory of Granular Computing and Applications, 
and Fujian Key Laboratory of Data Science and Statistics.


\begin{thebibliography}{10}

\bibitem{AnJW2017}
{\sc H.~An, X.~Jia, and H.~F. Walker}, {\em Anderson acceleration and
  application to the three-temperature energy equations}, J. Comput. Phys., 347
  (2017), pp.~1--19.

\bibitem{Anderson1965}
{\sc D.~G. Anderson}, {\em Iterative procedures for nonlinear integral
  equations}, J. ACM, 12 (1965), pp.~547--560.

\bibitem{Anderson2019}
{\sc D.~G. Anderson}, {\em {Comments on “Anderson acceleration, mixing and
  extrapolation”}}, Numer. Algorithms, 80 (2019), pp.~135--234.

\bibitem{Argyros1992}
{\sc I.~K. Argyros}, {\em {Remarks on the convergence of Newton's method under
  H{\"o}lder continuity conditions}}, Tamkang J. Math., 23 (1992),
  pp.~269--277.

\bibitem{Bach2021}
{\sc F.~Bach}, {\em {On the effectiveness of Richardson extrapolation in data
  science}}, SIAM J. Math. Data Sci., 3 (2021), pp.~1251--1277.

\bibitem{BaiGL2008}
{\sc Z.~Bai, Y.~Gao, and L.~Lu}, {\em {Fast iterative schemes for nonsymmetric
  algebraic Riccati equations arising from transport theory}}, SIAM J. Sci.
  Comput., 30 (2008), pp.~804--818.

\bibitem{BaoLW2006}
{\sc L.~Bao, Y.~Lin, and Y.~Wei}, {\em {A modified simple iterative method for
  nonsymmetric algebraic Riccati equations arising in transport theory}}, Appl.
  Math. Comput., 181 (2006), pp.~1499--1504.

\bibitem{BarreTDA2022}
{\sc M.~Barr{\'e}, A.~Taylor, and A.~d'Aspremont}, {\em Convergence of a
  constrained vector extrapolation scheme}, SIAM J. Math. Data Sci., 4 (2022),
  pp.~979--1002.

\bibitem{BarrenecheaJK2024}
{\sc G.~R. Barrenechea, V.~John, and P.~Knobloch}, {\em Finite element methods
  respecting the discrete maximum principle for convection-diffusion
  equations}, SIAM Rev., 66 (2024), pp.~3--88.

\bibitem{BermanP1994}
{\sc A.~Berman and R.~J. Plemmons}, {\em {Nonnegative Matrices in the
  Mathematical Sciences}}, SIAM, Philadelphia, 1994.

\bibitem{BertacciniGS2007}
{\sc D.~Bertaccini, G.~H. Golub, and S.~Serra-Capizzano}, {\em Spectral
  analysis of a preconditioned iterative method for the convection-diffusion
  equation}, SIAM J. Matrix Anal. Appl., 29 (2007), pp.~260--278.

\bibitem{BianC2022}
{\sc W.~Bian and X.~Chen}, {\em Anderson acceleration for nonsmooth fixed point
  problems}, SIAM J. Numer. Anal., 60 (2022), pp.~2565--2591.

\bibitem{BianCK2021}
{\sc W.~Bian, X.~Chen, and C.~T. Kelley}, {\em Anderson acceleration for a
  class of nonsmooth fixed-point problems}, SIAM J. Sci. Comput., 43 (2021),
  pp.~S1--S20.

\bibitem{BiniIP2008}
{\sc D.~A. Bini, B.~Iannazzo, and F.~Poloni}, {\em {A fast Newton's method for
  a nonsymmetric algebraic Riccati equation}}, SIAM J. Matrix Anal. Appl., 30
  (2008), pp.~276--290.

\bibitem{BollapragadaSD2023}
{\sc R.~Bollapragada, D.~Scieur, and A.~d'Aspremont}, {\em Nonlinear
  acceleration of momentum and primal-dual algorithms}, Math. Program., 198
  (2023), pp.~325--362.

\bibitem{BrezinskiR2019}
{\sc C.~Brezinski and M.~Redivo-Zaglia}, {\em {The genesis and early
  developments of Aitken's process, Shanks' transformation, the
  $\varepsilon$--algorithm, and related fixed point methods}}, Numer.
  Algorithms, 80 (2019), pp.~11--133.

\bibitem{BrezinskiRS2018}
{\sc C.~Brezinski, M.~Redivo-Zaglia, and Y.~Saad}, {\em {Shanks sequence
  transformations and Anderson acceleration}}, SIAM Rev., 60 (2018),
  pp.~646--669.

\bibitem{CartisGT2017}
{\sc C.~Cartis, N.~I. Gould, and P.~L. Toint}, {\em {Worst-case evaluation
  complexity of regularization methods for smooth unconstrained optimization
  using H{\"o}lder continuous gradients}}, Optim. Methods Softw., 32 (2017),
  pp.~1273--1298.

\bibitem{CartisGT2019}
{\sc C.~Cartis, N.~I. Gould, and P.~L. Toint}, {\em Universal regularization
  methods: varying the power, the smoothness and the accuracy}, SIAM J. Optim.,
  29 (2019), pp.~595--615.

\bibitem{ChenLW2014}
{\sc M.~Chen, R.~Lin, and Q.~Wu}, {\em {Convergence analysis of the modified
  Newton--HSS method under the H{\"o}lder continuous condition}}, J. Comput.
  Appl. Math., 264 (2014), pp.~115--130.

\bibitem{ChenK2019}
{\sc X.~Chen and C.~T. Kelley}, {\em {Convergence of the EDIIS algorithm for
  nonlinear equations}}, SIAM J. Sci. Comput., 41 (2019), pp.~A365--A379.

\bibitem{Catinas2021}
{\sc E.~C\u{a}tina\c{s}}, {\em How many steps still left to $x^*$?}, SIAM Rev.,
  63 (2021), pp.~585--624.

\bibitem{DeSterckH2021}
{\sc H.~De~Sterck and Y.~He}, {\em {On the asymptotic linear convergence speed
  of Anderson acceleration, Nesterov acceleration, and nonlinear GMRES}}, SIAM
  J. Sci. Comput., 43 (2021), pp.~S21--S46.

\bibitem{DevolderGN2014}
{\sc O.~Devolder, F.~Glineur, and Y.~Nesterov}, {\em First-order methods of
  smooth convex optimization with inexact oracle}, Math. Program., 146 (2014),
  pp.~37--75.

\bibitem{EvansPRX2020}
{\sc C.~Evans, S.~Pollock, L.~G. Rebholz, and M.~Xiao}, {\em {A proof that
  Anderson acceleration improves the convergence rate in linearly converging
  fixed-point methods (but not in those converging quadratically)}}, SIAM J.
  Numer. Anal., 58 (2020), pp.~788--810.

\bibitem{FangS2009}
{\sc H.-r. Fang and Y.~Saad}, {\em Two classes of multisecant methods for
  nonlinear acceleration}, Numer. Linear Algebra Appl., 16 (2009),
  pp.~197--221.

\bibitem{Ferreira2009}
{\sc O.~P. Ferreira}, {\em Local convergence of {Newton's} method in {Banach}
  space from the viewpoint of the majorant principle}, IMA J. Numer. Anal., 29
  (2009), pp.~746--759.

\bibitem{FuZB2020}
{\sc A.~Fu, J.~Zhang, and S.~Boyd}, {\em {Anderson accelerated
  Douglas--Rachford splitting}}, SIAM J. Sci. Comput., 42 (2020),
  pp.~A3560--A3583.

\bibitem{GhadimiLZ2019}
{\sc S.~Ghadimi, G.~Lan, and H.~Zhang}, {\em Generalized uniformly optimal
  methods for nonlinear programming}, J. Sci. Comput., 79 (2019),
  pp.~1854--1881.

\bibitem{GolubV1996}
{\sc G.~H. Golub and C.~F. Van~Loan}, {\em {Matrix Computations, Fourth
  Edition}}, The Johns Hopkins University Press, Baltimore and London, 2013.

\bibitem{GonzalezTaberneroLCG2024}
{\sc V.~Gonz{\'a}lez-Tabernero, J.~G. L{\'o}pez-Salas, M.~J. Castro-D{\'\i}az,
  and J.~A. Garc{\'\i}a-Rodr{\'\i}guez}, {\em {Boundary treatment for
  high-order IMEX Runge--Kutta local discontinuous Galerkin schemes for
  multidimensional nonlinear parabolic PDEs}}, SIAM J. Sci. Comput., 46 (2024),
  pp.~A3282--A3304.

\bibitem{GrapigliaN2017}
{\sc G.~N. Grapiglia and Y.~Nesterov}, {\em {Regularized Newton methods for
  minimizing functions with H\"{o}lder continuous Hessians}}, SIAM J. Optim.,
  27 (2017), pp.~478--506.

\bibitem{GrapigliaN2019}
{\sc G.~N. Grapiglia and Y.~Nesterov}, {\em {Accelerated regularized Newton
  methods for minimizing composite convex functions}}, SIAM J. Optim., 29
  (2019), pp.~77--99.

\bibitem{GrapigliaN2020}
{\sc G.~N. Grapiglia and Y.~Nesterov}, {\em {Tensor methods for minimizing
  convex functions with H\"{o}lder continuous higher-order derivatives}}, SIAM
  J. Optim., 30 (2020), pp.~2750--2779.

\bibitem{Guo2001}
{\sc C.-H. Guo}, {\em {Nonsymmetric algebraic Riccati equations and
  Wiener--Hopf factorization for M-matrices}}, SIAM J. Matrix Anal. Appl., 23
  (2001), pp.~225--242.

\bibitem{GuoL2000b}
{\sc C.-H. Guo and A.~J. Laub}, {\em {On the iterative solution of a class of
  nonsymmetric algebraic Riccati equations}}, SIAM J. Matrix Anal. Appl., 22
  (2000), pp.~376--391.

\bibitem{GuoL2010}
{\sc C.-H. Guo and W.-W. Lin}, {\em {Convergence rates of some iterative
  methods for nonsymmetric algebraic Riccati equations arising in transport
  theory}}, Linear Algebra Appl., 432 (2010), pp.~283--291.

\bibitem{Hernandez2001}
{\sc M.~A. Hern{\'a}ndez}, {\em {The Newton method for operators with
  H{\"o}lder continuous first derivative}}, J. Optim. Theory Appl., 109 (2001),
  pp.~631--648.

\bibitem{HighamS2016}
{\sc N.~J. Higham and N.~Strabi{\'c}}, {\em Anderson acceleration of the
  alternating projections method for computing the nearest correlation matrix},
  Numer. Algorithms, 72 (2016), pp.~1021--1042.

\bibitem{HuangM2020}
{\sc B.~Huang and C.~Ma}, {\em {Some accelerated iterative algorithms for
  solving nonsymmetric algebraic Riccati equations arising in transport
  theory}}, Int. J. Comput. Math., 97 (2020), pp.~1819--1839.

\bibitem{HuangM2014}
{\sc N.~Huang and C.~Ma}, {\em {Some predictor--corrector-type iterative
  schemes for solving nonsymmetric algebraic Riccati equations arising in
  transport theory}}, Numer. Linear Algebra Appl., 21 (2014), pp.~761--780.

\bibitem{Huang2002}
{\sc Z.~Huang}, {\em On {Newton's} method under {H\"{o}lder} continuous
  derivative}, J. Math. Anal. Appl., 270 (2002), pp.~332--339.

\bibitem{Huang2004}
{\sc Z.~Huang}, {\em The convergence ball of {Newton's} method and the
  uniqueness ball of equations under {H\"{o}lder}-type continuous derivatives},
  Comput. Math. Appl., 47 (2004), pp.~247--251.

\bibitem{HuangKH2010}
{\sc Z.~Huang, X.~Kong, and W.~Hu}, {\em {The King--Werner method for solving
  nonsymmetric algebraic Riccati equation}}, Appl. Math. Comput., 216 (2010),
  pp.~1790--1804.

\bibitem{Jay2001}
{\sc L.~O. Jay}, {\em {A note on Q-order of convergence}}, BIT Numer. Math., 41
  (2001), pp.~422--429.

\bibitem{Juang1995}
{\sc J.~Juang}, {\em {Existence of algebraic matrix Riccati equations arising
  in transport theory}}, Linear Algebra Appl., 230 (1995), pp.~89--100.

\bibitem{JuangL1998}
{\sc J.~Juang and W.-W. Lin}, {\em {Nonsymmetric algebraic Riccati equations
  and Hamiltonian-like matrices}}, SIAM J. Matrix Anal. Appl., 20 (1998),
  pp.~228--243.

\bibitem{Kelley2018}
{\sc C.~T. Kelley}, {\em Numerical methods for nonlinear equations}, Acta
  Numer., 27 (2018), pp.~207--287.

\bibitem{LiS2008}
{\sc C.~Li and W.~Shen}, {\em Local convergence of inexact methods under the
  {H\"{o}lder} condition}, J. Comput. Appl. Math., 222 (2008), pp.~544--560.

\bibitem{LiangL2026}
{\sc J.~Liang and Y.~Ling}, {\em {On the convergence of two-step modified
  Newton method for nonsymmetric algebraic Riccati equations from transport
  theory}}, Numer. Algorithms, 102 (2026), pp.~809--845.

\bibitem{Lin2008}
{\sc Y.~Lin}, {\em {A class of iterative methods for solving nonsymmetric
  algebraic Riccati equations arising in transport theory}}, Comput. Math.
  Appl., 56 (2008), pp.~3046--3051.

\bibitem{LinB2008}
{\sc Y.~Lin and L.~Bao}, {\em {Convergence analysis of the Newton--Shamanskii
  method for a nonsymmetric algebraic Riccati equation}}, Numer. Linear Algebra
  Appl., 15 (2008), pp.~535--546.

\bibitem{LinBW2008}
{\sc Y.~Lin, L.~Bao, and Y.~Wei}, {\em {A modified Newton method for solving
  non-symmetric algebraic Riccati equations arising in transport theory}}, IMA
  J. Numer. Anal., 28 (2008), pp.~215--224.

\bibitem{LinBW2011}
{\sc Y.~Lin, L.~Bao, and Q.~Wu}, {\em {On the convergence rate of an iterative
  method for solving nonsymmetric algebraic Riccati equations}}, Comput. Math.
  Appl., 62 (2011), pp.~4178--4184.

\bibitem{LingLL2022}
{\sc Y.~Ling, J.~Liang, and W.~Lin}, {\em {On semilocal convergence analysis
  for two-step Newton method under generalized Lipschitz conditions in Banach
  spaces}}, Numer. Algorithms, 90 (2022), pp.~577--606.

\bibitem{LingX2017}
{\sc Y.~Ling and X.~Xu}, {\em {On one-parameter family of Newton-like
  iterations for solving nonsymmetric algebraic Riccati equation from transport
  theory}}, J. Nonlinear Convex Anal., 18 (2017), pp.~1833--1848.

\bibitem{LipnikovSV2013}
{\sc K.~Lipnikov, D.~Svyatskiy, and Y.~Vassilevski}, {\em Anderson acceleration
  for nonlinear finite volume scheme for advection-diffusion problems}, SIAM J.
  Sci. Comput., 35 (2013), pp.~A1120--A1136.

\bibitem{Lu2005b}
{\sc L.-Z. Lu}, {\em {Newton iterations for a non-symmetric algebraic Riccati
  equation}}, Numer. Linear Algebra Appl., 12 (2005), pp.~191--200.

\bibitem{Lu2005a}
{\sc L.-Z. Lu}, {\em {Solution form and simple iteration of a nonsymmetric
  algebraic Riccati equation arising in transport theory}}, SIAM J. Matrix
  Anal. Appl., 26 (2005), pp.~679--685.

\bibitem{MaiJ2020}
{\sc V.~Mai and M.~Johansson}, {\em Anderson acceleration of proximal gradient
  methods}, in International Conference on Machine Learning, PMLR, 2020,
  pp.~6620--6629.

\bibitem{MarumoT2025}
{\sc N.~Marumo and A.~Takeda}, {\em {Universal heavy-ball method for nonconvex
  optimization under H{\"o}lder continuous Hessians}}, Math. Program., 212
  (2025), pp.~147--175.

\bibitem{MehrmannX2008}
{\sc V.~Mehrmann and H.~Xu}, {\em {Explicit solutions for a Riccati equation
  from transport theory}}, SIAM J. Matrix Anal. Appl., 30 (2008),
  pp.~1339--1357.

\bibitem{Nesterov2015}
{\sc Y.~Nesterov}, {\em Universal gradient methods for convex optimization
  problems}, Math. Program., 152 (2015), pp.~381--404.

\bibitem{OrtegaR1970}
{\sc J.~M. Ortega and W.~C. Rheinboldt}, {\em {Iterative Solution of Nonlinear
  Equations in Several Variables}}, Academic Press, New York, 1970.

\bibitem{OuyangLM2024}
{\sc W.~Ouyang, Y.~Liu, and A.~Milzarek}, {\em {Descent properties of an
  Anderson accelerated gradient method with restarting}}, SIAM J. Optim., 34
  (2024), pp.~336--365.

\bibitem{PasiniYRS2021}
{\sc M.~L. Pasini, J.~Yin, V.~Reshniak, and M.~Stoyanov}, {\em {Stable Anderson
  acceleration for deep learning}}, arXiv preprint arXiv:2110.14813,  (2021).

\bibitem{PasiniYRS2022}
{\sc M.~L. Pasini, J.~Yin, V.~Reshniak, and M.~K. Stoyanov}, {\em Anderson
  acceleration for distributed training of deep learning models}, in
  SoutheastCon 2022, IEEE, 2022, pp.~289--295.

\bibitem{PollockR2021}
{\sc S.~Pollock and L.~G. Rebholz}, {\em Anderson acceleration for contractive
  and noncontractive operators}, IMA J. Numer. Anal., 41 (2021),
  pp.~2841--2872.

\bibitem{PollockR2023}
{\sc S.~Pollock and L.~G. Rebholz}, {\em {Filtering for Anderson
  acceleration}}, SIAM J. Sci. Comput., 45 (2023), pp.~A1571--A1590.

\bibitem{PollockRTX2025}
{\sc S.~Pollock, L.~G. Rebholz, X.~Tu, and M.~Xiao}, {\em {Analysis of the
  Picard-Newton iteration for the Navier-Stokes equations: global stability and
  quadratic convergence}}, J. Sci. Comput., 104 (2025).
\newblock Article number 25, 23 pp.

\bibitem{PollockRX2019}
{\sc S.~Pollock, L.~G. Rebholz, and M.~Xiao}, {\em {Anderson-accelerated
  convergence of Picard iterations for incompressible Navier--Stokes
  equations}}, SIAM J. Numer. Anal., 57 (2019), pp.~615--637.

\bibitem{Potra1989}
{\sc F.~A. Potra}, {\em {On Q-order and R-order of convergence}}, J. Optim.
  Theory Appl., 63 (1989), pp.~415--431.

\bibitem{Pulay1980}
{\sc P.~Pulay}, {\em {Convergence acceleration of iterative sequences. The case
  of SCF iteration}}, Chem. Phys. Lett., 73 (1980), pp.~393--398.

\bibitem{RebholzVX2021}
{\sc L.~G. Rebholz, D.~Vargun, and M.~Xiao}, {\em {Enabling convergence of the
  iterated penalty Picard iteration with $O(1)$ penalty parameter for
  incompressible Navier--Stokes via Anderson acceleration}}, Comput. Methods
  Appl. Mech. Eng., 387 (2021).
\newblock Article number 114178, 17 pp.

\bibitem{RebholzX2023}
{\sc L.~G. Rebholz and M.~Xiao}, {\em {The effect of Anderson acceleration on
  superlinear and sublinear convergence}}, J. Sci. Comput., 96 (2023).
\newblock Article number 34, 23 pp.

\bibitem{Rokne1972}
{\sc J.~Rokne}, {\em {Newton's} method under mild differentiability conditions
  with error analysis}, Numer. Math., 18 (1972), pp.~401--412.

\bibitem{Saad2025}
{\sc Y.~Saad}, {\em Acceleration methods for fixed point iterations}, Acta
  Numer., 34 (2025), pp.~805--890.

\bibitem{ScieurAB2020}
{\sc D.~Scieur, A.~d'Aspremont, and F.~Bach}, {\em Regularized nonlinear
  acceleration}, Math. Program., 179 (2020), pp.~47--83.

\bibitem{ShenL2008}
{\sc W.~Shen and C.~Li}, {\em Convergence criterion of inexact methods for
  operators with {H\"{o}lder} continuous derivatives}, Taiwan. J. Math., 12
  (2008), pp.~1865--1882.

\bibitem{TothK2015}
{\sc A.~Toth and C.~T. Kelley}, {\em {Convergence analysis for Anderson
  acceleration}}, SIAM J. Numer. Anal., 53 (2015), pp.~805--819.

\bibitem{WalkerN2011}
{\sc H.~F. Walker and P.~Ni}, {\em Anderson acceleration for fixed-point
  iterations}, SIAM J. Numer. Anal., 49 (2011), pp.~1715--1735.

\bibitem{WangHS2021}
{\sc D.~Wang, Y.~He, and H.~De~Sterck}, {\em {On the asymptotic linear
  convergence speed of Anderson acceleration applied to ADMM}}, J. Sci.
  Comput., 88 (2021).
\newblock Article number 38, 35 pp.

\bibitem{ZhangDB2020}
{\sc J.~Zhang, B.~O'Donoghue, and S.~Boyd}, {\em {Globally convergent type-I
  Anderson acceleration for nonsmooth fixed-point iterations}}, SIAM J. Optim.,
  30 (2020), pp.~3170--3197.

\bibitem{ZhongCG2018}
{\sc H.-X. Zhong, G.-L. Chen, and X.-P. Guo}, {\em {Convergence of the
  Newton-HSS method under the Lipschitz condition with the L-average}}, J.
  Comput. Anal. Appl., 25 (2018), pp.~952--964.

\end{thebibliography}

\end{document}